\documentclass[12pt]{amsbook}
\usepackage{amssymb, latexsym, amsthm} 
\newtheorem{thm}{Theorem}[section] 
\newtheorem{algo}[thm]{Algorithm}

\newtheorem{conj}[thm]{Conjecture}
\newtheorem{cor}[thm]{Corollary}
\newtheorem{defn}[thm]{Definition}

\newtheorem{exmpl}[thm]{Example}
\newtheorem{lem}[thm]{Lemma}
\newtheorem{prop}[thm]{Proposition}

\newtheorem{rem}[thm]{Remark}

\newtheorem{ques}[thm]{Question}

\DeclareMathOperator{\mychar}{char} %

\newcommand\operA[2]{{\if!#2!\operatorname{#1}\else{\operatorname{#1}_{#2}^{\phantom{I}}}\fi}} 
\newcommand\set[1]{\{#1\}}
\newcommand\cent[1]{\mathrm{cent}\left(#1\right)}
\newcommand\freegen[1]{{\left<#1\right>}}
\newcommand\fring{{D\freegen{x_1,\dots,x_N}}}
\newcommand\charac[1]{\mathrm{char}\left(#1\right)}

\newcommand\eq[1]{{(\ref{#1})}}
\newcommand\Cref[1]{{Corollary~\ref{#1}}}%

\newcommand\tensor[1][]{{\otimes_{#1}}}
\def\sub{\subseteq}

\def\tr{{\operatorname{Tr}}}

\def\norm{{\operatorname{N}}}
\def\ind{{\operatorname{ind}}}
\def\rad{{\operatorname{Rad}}}

\def\Z{\mathbb{Z}}

\def\s{\sigma}
\def\Gal{{\operatorname{Gal}}}

\newcommand\cores[1][]{{\operatorname{cor}_{#1}}}

\newcommand\Norm[1][]{\operA{N}{#1}}

\newcommand\mul[1]{{#1^{\times}}} 
\newcommand\dimcol[2]{{[{#1}\!:\!{#2}]}} 
\newcommand{\Trace}[1][]{\if!#1!\operatorname{Tr}\else{\operatorname{Tr}_{#1}^{\phantom{I}}}\fi} 
\newcommand\eqs[3][--]{{\eq{#2}{#1}\eq{#3}}}

\long\def\forget#1\forgotten{{}} %
\newcommand\suchthat{{\,:\ \,}}

\def\({\left(}
\def\){\right)}

\def\CX{\mathcal X}

\def\lra{{\,\longrightarrow\,}}

\def\s{{\sigma}}

\def\sub{\subseteq}

\newcommand\paper[7]{{{#1},\ {\it{#2}},\ {#3}\ {\bf{#4}}\if!#5!\else(#5)\fi,\ {#6},\ ({#7}).}} 
\newcommand\book[4]{{{#1},\ {{#2}}{\if!#3!\relax\else{,\ {#3}}\fi}{\if!#4!\relax\else{,\ {#4}}\fi}.}} 
\newcommand\submitted[3]{{{#1},\ {\it{#2}}, submitted{\if!#3!{}\else, ({#3})\fi}.}} 

\newif\iffurther
\furtherfalse

\newif\ifcleanup %
\cleanupfalse

\newif\ifXY 
\XYtrue     
\ifXY

\input xy
\input xyidioms.tex
\usepackage{xy}
\xyoption{all} %
\fi 

\begin{document}
\pagenumbering{gobble}

\allowdisplaybreaks

\title{$p$-Central~Subspaces~of Central~Simple~Algebras}

\author{Adam Chapman \\ \mbox{} \\ \large{Department of Mathematics} \\ \large{Ph.D. Thesis} \\ \large{Submitted to the Senate of Bar-Ilan University}
\\ \vspace{9cm} \large {Ramat-Gan, Israel} \quad\quad\quad\quad\quad\quad \large{August 2013}}

\maketitle

\newpage

\begin{center} \mbox{} \\ \mbox{} \\\mbox{} \\\mbox{} \\\mbox{} \\\mbox{} \\\mbox{} \\
\noindent{This work was carried out under the supervision of \linebreak
\large Prof. Uzi Vishne \linebreak
\normalsize Department of Mathematics, Bar-Ilan University.}
\end{center}

\newpage
\textit{This thesis is dedicated to the memory of my beloved grandmother Ahuva Luz (n\'{e}e Frenkel), who died of cancer during the first year of my PhD.}

\tableofcontents

\pagebreak
\chapter*{Abstract}
\pagenumbering{roman}
We study central simple algebras in various ways, focusing on the role of $p$-central subspaces.
The first part of my thesis is dedicated to the study of Clifford algebras. The standard Clifford algebra of a given form is the generic associative algebra containing a $p$-central subspace whose exponentiation form is equal to the given form. There is an old question as for whether these algebras have representations of finite rank over the center, and jointly with Daniel Krashen and Max Lieblich we managed to provide a positive answer. Different generalizations of the structure of the Clifford algebra are presented and studied in that part too.
The second part is dedicated to the study of $p$-central subspaces of given central simple algebras, mainly tensor products of cyclic algebras of degree $p$. Among the results, we prove that $5$ is the upper bound for the dimension of 4-central subspaces of cyclic algebras of degree 4 containing pairs of standard generators.
The third part is dedicated to chain lemmas. Chain lemmas are of importance in the theory of central simple algebras, because they form one approach to solving the word problem for the Brauer group. We prove the chain lemma for biquaternion algebras, both in characteristic 2 and characteristic not 2, and prove some partial results on the chain lemmas for cyclic algebras of degree $p$.
The fourth part is dedicated to the more computational aspects of the theory. It contains results on quaternion polynomial equations and on left eigenvalues of quaternion matrices.

\chapter*{Introduction}
\pagenumbering{arabic}
A finite dimensional associative algebra $A$ over a field $F$ is called central simple if it has no proper two-sided ideals and $Z(A)=F$.
If $A$ and $B$ are two central simple algebras over $F$ then $A \otimes B$ is also central simple over $F$.
Consequently, the set of isomorphism classes of central simple algebras over $F$ forms a semigroup.

According to Wedderburn, if $A$ is central simple over $F$ then $A$ decomposes uniquely as $M \otimes D$ where $M$ is a matrix algebra over $F$ and $D$ is a central division algebra over $F$.
Both $M$ and $D$ are in particular central simple algebras.
We say that $A$ and $B$ are Brauer equivalent if they have the same underlying division algebra.
Consequently, the set of central simple algebras over $F$ modulo that equivalence relation forms a commutative monoid.

It is known that $A \otimes A^{\operatorname{op}}$ is a matrix algebra, and therefore this monoid is a group, and is called the Brauer group of $F$.

This group is known to be a torsion group, i.e. for every central simple algebra $A$ there exists some positive integer $e$ such that $A \otimes \stackrel{(e \ \text{times})}{\dots} \otimes A$ is a matrix algebra over $F$. The minimal such $e$ is called the exponent of $A$ and denoted by $\exp(A)$.

A splitting field of a given central simple algebra $A$ over $F$ is a field extension $K/F$ for which $A \tensor_F K$ is a matrix algebra over $K$, i.e. $A \tensor_F K=M_d(K)$ for some integer $n$.
A splitting field is known to exist for any central simple algebra, for example any maximal subfield of the algebra is a splitting field. Since $[A:F]=[A \tensor  K:K]$, $[A:F]$ is the square of some integer, called the degree of $A$ and denoted by $\deg(A)$.
The degree of the underlying division algebra $D$ of $A$ is called the index of $A$ and is denoted by $\ind(A)$.

It is known that $\exp(A) | \ind(A) | \deg(A)$.

One way to study the structure of a given division or central simple algebra $A$ over some center $F$ is to focus on some subsets of elements with special behavior, such as $d$-central elements. A noncentral element $y$ is called $d$-central if $y^d$ is in the center and $y^k$ is not in the center for any $1 \leq k \leq d-1$. When $d$ is a prime, we often use the letter $p$, and refer to these elements as $p$-central elements.

The $d$-central elements are of special importance in the structure theory of division algebras and of central simple algebras in general, through their connection to cyclic
field extensions and cyclic algebras.

Every maximal subfield of a division algebra has dimension equal to the degree.
The algebra is called cyclic if it has a maximal subfield which is cyclic Galois over the center.

Hamilton's quaternion algebra is the classical example of a
cyclic algebra of degree $2$ over the real numbers. The first
examples of arbitrary degree were constructed by Dickson
\cite{Dickson}, as follows: Let $L/F$ be an $n$-dimensional cyclic
Galois extension with $\s$ a generator of $\Gal(L/F)$, and let
$\beta \in \mul{F}$. Then $\bigoplus_{i=0}^{p-1} L y^i$, subject to
the relations $yu = \s(u)y$ (for $u \in L$) and $y^n = \beta$, is
a cyclic algebra of degree $d$, denoted by $(L/F,\s,\beta)$; every cyclic algebra has this form. In particular, every cyclic
algebra of degree $d$ has a $d$-central element.

If $F$ contains a primitive $d$th root of unity, then each $d$-central element
of a division algebra generates a cyclic maximal subfield.
However, there are central division algebras with $d$-central
elements which are not cyclic. The first example, for $d = 4$, was
given by Albert, and an example with $n = q^2$ for an
arbitrary prime $q$ was recently constructed by Matzri, Rowen and
Vishne \cite{MRV}. Nevertheless, Albert proved that if $p$ is prime then every central division algebra
with a $p$-central element is cyclic, regardless of the characteristic of the field or the existence of a primitive $p$th root of unity.

For prime $d$, when $F$ is of characteristic prime to $d$ and contains a primitive $d$th root of unity $\rho$, a cyclic maximal
subfield has the form $L = F[x]$ where $x$ is $d$-central, so
every cyclic algebra has the `symbol algebra' form
$$(\alpha,\beta)_{d,F} = F[x,y : x^d = \alpha, \ y^d = \beta, \ yx = \rho xy]$$
emphasizing even further the role of $d$-central elements in
presentations of cyclic algebras.

For prime $p$, if $F$ is of characteristic $p$ then every cyclic algebra of degree $p$ over $F$ has the form
$$[\alpha,\beta)_{p,F} = F[x,y : x^p-x = \alpha, \ y^p = \beta, \ yx - xy = y].$$
In this case, along the $p$-central elements there are also the Artin-Schreier elements, i.e. non-central elements satisfying the condition $x^p-x \in F$.

A $d$-central space $V$ is an $F$-vector subspace of $A$ in which all the nonzero elements are $d$-central.
For example, in the above presentation, $Fx+Fy$ is a $d$-central space. Furthermore, $F x+F[x] y$ is $d$-central.
The existence of $p$-central spaces tells us a lot about the structure of the algebra as we shall soon see.

The decomposition of elements with respect to a given special element also stars throughout this thesis.
For characteristic prime to $d$, we use the eigenvector decomposition of elements with respect to conjugation by a certain $d$-central element.

\begin{lem}\label{eigencharnot}
In a given associative algebra $A$ over a field $F$ of characteristic prime to $d$ containing a primitive $d$th root of unity $\rho$, if $x^d \in F^\times$ then for every $y \in A$, $y=y_0+\dots+y_{d-1}$ such that $y_k x=\rho^k x y_k$.
\end{lem}

\begin{proof}
This is the eigenvector decomposition: Take $y_k=\frac{1}{d}(y+\rho^k x y x^{-1}+\dots+\rho^{k (d-1)} x^{d-1} y x^{1-d})$ for $0 \leq k \leq d-1$.
It is an easy calculation to see that $y_k x=\rho^k x y_k$.
\end{proof}

For a prime $p$, and characteristic $p$, there are two interesting types of elements, the Artin-Schreier elements, i.e. elements that satisfy an equation of the form $x^p-x=\alpha$ for some $\alpha \in F$, and $p$-central elements that are defined above, regardless of the characteristic.
In this case, however, the $p$-central elements generate purely inseparable field extensions over the base-field, while the Artin-Schreier elements generate Galois field extensions.

There are still decomposition lemmas with respect to these two special types of elements.

When the characteristic is $p$, we write $[\mu,\nu]=[\mu,\nu]_1=\nu \mu-\mu \nu$ and define $[\mu,\nu]_k$ inductively as $\nu [\mu,\nu]_{k-1}-[\mu,\nu]_{k-1} \nu$. $[\mu,\nu]_0$ is defined to be $\mu$.

\begin{lem}\label{decomcharp}
Given an associative algebra $A$ over a field $F$ of characteristic $p$,
if $x$ is Artin-Schreier then for any $z \in A$, $z=z_0+z_1+\dots+z_{p-1}$ where $[z_k,x]=k z_k$. Similarly, by taking $t_{p-k}=z_k$, $z=t_0+\dots+t_{p-1}$ such that $[x,t_k]=k t_k$.
\end{lem}

\begin{proof}
Let $z_0=z-[z,x]_{p-1}$, and for all $1 \leq k \leq p-1$, $z_k=-(k^{-(p-2)} [z,x]_1+\dots+k^{-1} [z,x]_{p-2}+[z,x]_{p-1})$.
It is an easy calculation to prove that $[z,x]=k z_k$. It is obvious that $z_0+z_1+\dots+z_{p-1}=z$.
\end{proof}

\begin{lem}\label{decomcharp2}
Given an associative algebra $A$ over a field $F$ of characteristic $p$,
if $y$ is $p$-central then for any $z \in A$, there exist $\set{z_k : k \in \mathbb{Z}_p}$ such that for all $k \neq 0$, $[z_k,y]_1=z_{k-1}$ and $[z_0,y]_1=0$, and $z=z_{p-1}-z_{p-2}$.
\end{lem}

\begin{proof}
For $k \neq 0$, let $z_k=[z,y]_{p-1-k}+[z,y]_{p-k}+\dots+[z,y]_{p-1}$.
Let $z_0=[z,y]_{p-1}$.
It is clear that they satisfy the requirements.
\end{proof}

\chapter{Clifford Algebras}

\section{Background}
Let $F$ be an infinite field and $f(a_1,\dots,a_n)$ be a homogeneous form of degree $d$ with $n$ variables over $F$.

The Clifford algebra of $f$, denoted by $C_f$, is defined to be
$$F \left< x_1,\dots,x_n : (a_1 x_1+\dots+a_n x_n)^p=f(a_1,\dots,a_n) \forall a_1,\dots,a_n \in F \right>.$$
This definition is due to Roby \cite{Roby}.

Even though it looks like the number of relations is infinite, $C_f$ is finitely presented.

In \cite{Revoy}, Revoy introduced the following notation in order to describe the finite set of relations: $x_1^{d_1} * x_2^{d_2} * \dots * x_n^{d_n}$. This expression means the sum of all the words that consist of only the letters $x_1,\dots,x_n$ and each letter $x_i$ appears exactly $d_i$ times.
For example, $x^2 * y=x^2 y+x y x+y x^2$.

The Clifford algebra is then the algebra generated over $F$ by $x_1,\dots,x_n$ subject to the following relations: $$x_1^{d_1} * \dots * x_n^{d_n}=\alpha_{d_1,\dots,d_n}$$ for any set of non-negative integers $\set{d_1,\dots,d_n}$ such that $d_1+\dots+d_n=d$, where $\alpha_{d_1,\dots,d_n}$ is the coefficient of $a_1^{d_1} \dots a_n^{d_n}$ in $f(a_1,\dots,a_n)$.

This algebra is clearly invariant under any linear change of the variables of $f$, and it is one of the most important invariants of homogeneous forms in general and of quadratic forms in particular. For nondegenerate quadratic forms, the Clifford algebra is a cohomological invariant (see \cite{BOI}).

Given a central simple algebra $A$ over $F$, if $A$ contains a $d$-central space $V=F v_1+\dots+F v_n$ with a fixed basis $\set{v_1,\dots,v_n}$ then $V$ has a natural exponentiation form $f(a_1,\dots,a_n)=(a_1 v_1+\dots+a_n v_n)^d$. This is a homogeneous form of degree $d$, and it has a Clifford algebra $C_f$.
One can refer to $C_f$ as the Clifford algebra of the $d$-central space $V$ itself and denote it by $C(V)$.

The elements of $V$ generate a subalgebra $F[V]=F[v_1,\dots,v_n]$ of $A$, which is often equal to $A$.
This subalgebra is a finite representation of the Clifford algebra of $V$.
Therefore studying the finite representations of $C(V)$, and in particular its simple images, may shed some light on the structure of $A$ itself.

\begin{rem}\label{repremark}
Throughout this chapter, by a finite representation of $C_f$ (or $C(V)$) we mean a homomorphic image of $C_f$ inside some matrix algebra of finite degree over $F$. We say that $C_f$ is finitely representable if such a representation exists.
The existence of such a representation is equivalent to the existence of simple images of finite dimension.
The rank of a representation is its dimension over $F$. If a representation is simple its degree is the square root of its rank.
\end{rem}

Every $d$-central subspace of a central simple algebra has an exponentiation form which is a homogeneous form of degree $d$. A natural question would be whether every homogeneous form of degree $d$ is the exponentiation form of some $d$-central subspace of a central simple algebra. This question is known as the ``Finite Linearization Problem" and is discussed in Section \ref{rep} where a positive answer is provided.

According to Van den Bergh, the Clifford algebra of a binary form of degree $\geq 4$ has representations of unbounded high ranks. 
However, it is not easy to construct explicit examples of high degree simple representations.
We provide explicit examples of simple images of degree $d^2$ and index $d$ for the Clifford algebra of a diagonal binary form of degree $d$.

The Clifford algebra of a quadratic form or a quadratic space in characteristic not $2$ is a classical object. This algebra is known to be a tensor product of quaternion algebras either over $F$ or over a quadratic extension of the center (see, e.g. \cite{Lam} or \cite{BOI}).

The case of $\charac{F}=2$ was studied by Mammone, Tignol and Wadsworth in \cite{MTW}. They concluded, similarly to the characteristic not $2$ case, that the Clifford algebra is a tensor product of quaternion algebras either over $F$ or over a purely inseparable field extension of it.

Assuming $\charac{F} \neq 2,3$, the case of $d=3$ and $n=2$ was first considered
by Heerema in \cite{Heerema}. Haile studied these algebras in
\cite{Haile} and \cite{Haile2}, and showed that $C_f$ is an Azumaya algebra, whose center is isomorphic to the coordinate ring of the affine
elliptic curve $s^2=r^3-27\Delta$ where $\Delta$ is the
discriminant of $f$. He also proved that the simple homomorphic
images of $C_f$ are cyclic algebras of degree $3$. Moreover,
for every algebraic extension $K/F$ there is a one to one
correspondence between the $K$-points on that elliptic curve and the simple homomorphic images of $C_f$ whose center is $K$.

In Section \ref{short} we generalize this result for any prime $p$, proving that some specific quotient of the Clifford algebra (that is equal to the Clifford algebra in case of $p=3$) is Azumaya whose center is a hyper-elliptic curve and all its simple images are cyclic of degree $3$. This is the Clifford algebra of a short $p$-central space of type $\set{i,p-i}$.
Clifford algebras of short $p$-central spaces of different types appear to have simple images of degree $p^2$ that are easy to construct.

In \cite{Pappacena}, Pappacena generalized the notion of the Clifford algebra to the algebra associated to a monic polynomial (with respect to the first variable) the form $\Phi(z,a_1,\dots,a_n)=z^d-\sum_{k=1}^d f_k(a_1,\dots,a_n) z^{d-k}$ where each $f_k$ is a homogeneous form of degree $k$.
This algebra, denoted there by $C_\Phi$, is defined to be
\begin{equation*}
\begin{aligned}
F\langle x_1,\dots,x_n\colon &  (a_1 x_1+\dots+a_n x_n)^d\\
& = f_1(a_1,\dots,a_n) (a_1 x_1+\dots+a_n x_n)^{d-1}+\dots+\\
&f_{d-1}(a_1,\dots,a_n) (a_1 x_1+\dots+a_n x_n)+f_d(a_1,\dots,a_n)\\
& \textrm{ for all } a_1,\dots,a_n \in F\rangle,
\end{aligned}
\end{equation*}

Pappacena proved in that paper that if $d=2$ then this algebra is isomorphic to the Clifford algebra of a quadratic form, and therefore its structure is known.

In \cite{Kuo}, Kuo studied the Clifford algebra of the polynomial $\Phi(z,a,b)=z^3-e a b z-f(a,b)$ and the results are very similar to the results Haile obtained in \cite{Haile}. The formulas for the simple images of the Clifford algebra are provided there only in case $f(a,b)$ is diagonal.

In Section \ref{Kuos}, we study two cases separately, one of the polynomial $\Phi(z,a,b)=z^3-r b z^2-(e a b+t b^2) z-(\alpha a^3+\beta a^2 b+\gamma a b^2+\delta b^3)$ assuming that the characteristic of $F$ is different from $3$ and $2$ and that it contains a primitive third root of unity, and of the polynomial $\Phi(z,a,b)=z^3-e a b z-f(a,b)$ assuming that the characteristic of $F$ is $3$. In particular we provide formulas for the images of the Clifford algebra studied in \cite{Kuo} in the non-diagonal case.

In Section \ref{sgca} we present a further generalization of the algebra defined by Pappacena, the Clifford algebra of a degree $d$ projective variety. The results from \cite{Haile4} are generalized for any $2$-central variety whose defining equations are mutually diagonalizable quadratic equations.

Section \ref{short} is based on a published paper, written collaboratively with my Ph.D advisor, Uzi Vishne. Section \ref{Kuos} is based on a collaborative work with Jung-Miao Kuo. Sections \ref{rep} and \ref{sgca} are taken from a joint work with Daniel Krashen and Max Lieblich.

\section{The Finite Linearization Problem}\label{rep}

In this section we wish to prove the old conjecture that the Clifford algebra of any given form is finitely representable (see Remark \ref{repremark}).

Let $d$ be a positive integer, $F$ be an infinite field and $f(a_1,\dots,a_n)=\sum_{d_1+\dots+d_n=d} c_{d_1,\dots,d_n} a_1^{d_1} \dots a_n^{d_n}$ be a homogeneous form of degree $d$ in $n$ variables.

In this section we consider one of the two following cases:
\begin{enumerate}
\item\label{case1} The characteristic is prime to $d$.
\item\label{case2} $d$ is prime and equal to the characteristic.
\end{enumerate}

Let $$C_f=F[x_1,\dots,x_n : (a_1 x_1+\dots+a_n x_n)^d=f(x_1,\dots,x_n) \forall a_1,\dots,a_n \in F]$$ be its Clifford algebra.

We say that $f$ has a finite linearization if for some positive integer $m$, there exist matrices $X_1,\dots,X_n \in M_m(F)$ such that $(a_1 X_1+\dots a_n X_n)^d=f(a_1,\dots,a_n)$ for all $a_1,\dots,a_n \in F$.
It is clear that $C_f$ is finitely representable if and only if $f$ has a finite linearization.

The ``Finite Linearization Problem" is the following question:
\begin{ques}
Does $f$ always have a finite linearization?
\end{ques}

This question originally arose in 1928 in Dirac's treatment of the relativistic wave equation in quantum mechanics.
He was mainly interested in the special case of $d=2$ and $n=4$.

It is important to note that given a finite field extension $K/F$, $f$ has a finite linearization over $F$ if and only if it has a finite linearization over $K$. Hence, in Case \ref{case1} we shall assume that $F$ contains a primitive $d$th root of unity $\rho$.

In order to understand even the simplest examples in Case \ref{case1}, one should be familiar with the concept of $\mathbb{Z}_d$-grading.

\subsection{$\mathbb{Z}_d$-grading}

Let $\mathbb{Z}_d$ be the finite group with $d$ elements obtained by taking the additive group of integers $\mathbb{Z}$ modulo its subgroup $d \mathbb{Z}$.

An associative algebra $A$ over $F$ is $\mathbb{Z}_d$-graded if $A=A_0 \bigoplus \dots \bigoplus A_{d-1}$ such that for every $a_j \in A_j$ and $a_k \in A_k$, $a_j a_k \in A_{j+k}$. The elements in $A_j$ are called homogeneous elements of grade $j$.

\begin{exmpl}\label{gradeex}
The matrix algebra $M_d(F)$ can be graded in such a way that each element $e_{k,k+j}$ is homogeneous of grade $j$.
\end{exmpl}

The $\mathbb{Z}_d$-graded tensor product of two $\mathbb{Z}_d$-graded algebras $A$ and $B$ is defined to be the algebra whose elements are sums of $a \otimes_{\mathbb{Z}_d} b$ such that $a \in A$ and $b \in B$, with addition that satisfies $(a+b) \otimes_{\mathbb{Z}_d} (c+d)=a \otimes_{\mathbb{Z}_d} c+a \otimes_{\mathbb{Z}_d} d+b \otimes_{\mathbb{Z}_d} c+b \otimes_{\mathbb{Z}_d} d$ and multiplication that satisfies $(a \otimes_{\mathbb{Z}_d} b_j) \cdot (a_k \otimes_{\mathbb{Z}_d} b)=\rho^{j k} (a a_k) \otimes_{\mathbb{Z}_d} (b_j b)$ for $b_j \in B_j$ and $a_k \in A_k$. In the special case of $d=1$ we get the ordinary tensor product.

The theory of $\mathbb{Z}_d$-graded central simple algebras has been studied by many different mathematicians, such as Wall, Lam, Bahturin, Aljadeff and others (see \cite{Long} or \cite{Koc} for background). Here we shall recall only what we need.
A $\mathbb{Z}_d$-graded central simple algebra is a unital associative algebra over a given field with no proper $\mathbb{Z}_d$-graded two-sided ideals.
It is known that a central simple algebra that has a $\mathbb{Z}_d$-grading is a $\mathbb{Z}_d$-graded central simple algebra.
Furthermore, a $\mathbb{Z}_d$-graded tensor product of $\mathbb{Z}_d$-graded central simple algebras is also a $\mathbb{Z}_d$-graded central simple algebra.
Lastly, a $\mathbb{Z}_d$-graded central simple algebra over $F$ always lives inside a finite matrix algebra over $F$, and therefore if $C_f$ has an image which is a $\mathbb{Z}_d$-graded central simple algebra over $F$ then it is finitely representable.

\subsection{Formerly known results}

\begin{exmpl}\label{excases}
Let us have a look at a given diagonal form $f=\alpha_1 u_1^d+\dots+\alpha_n u_n^d$.

In Case 1, for each $1 \leq k \leq n$, $F[\mu_k : \mu_k^d=\alpha_k]$ is $\mathbb{Z}_d$-graded with $\mu_k$ as the homogeneous element of grade $1$.
There is a $\mathbb{Z}_d$-graded representation $\Phi : C_f \rightarrow \otimes_{\mathbb{Z}_d} \ _{k=1}^n F[\mu_k : \mu_k^d=\alpha_k]$, taking each $x_k$ to $\mu_k$.

In Case 2, we obtain a similar representation by replacing $\otimes_{\mathbb{Z}_d}$ with $\otimes$.

Consequently diagonal forms always have finite linearizations.
\end{exmpl}

Childs proved in \cite{Childs} that if $f$ is similar to a direct sum of unary and binary forms then $f$ has a finite linearization.

Van den Bergh proved in \cite{Van} that in the special case of $d=n=3$, $f$ has a finite linearization.

\subsection{A positive answer in general}
\sloppy
Let $g(a_1,\dots,a_n)=c_{d,0,\dots,0} a_1^d+\dots+c_{0,\dots,0,d} a_n^d$ be the diagonal part of $f$. For example, if $f(a_1,a_2)=c_{2,0} a_1^2+c_{1,1} a_1 a_2+c_{0,2} a_2^2$ then $g(a_1,a_2)=c_{2,0} a_1^2+c_{0,2} a_2^2$.

Let $$C_{g}=F[y_1,\dots,y_n : (a_1 y_1+\dots+a_n y_n)^d=g(a_1,\dots,a_n) \forall a_1,\dots,a_n \in F]$$ be its Clifford algebra.

Let $\Phi_g : C_g \rightarrow B$ be a representation, taking each $y_k$ to $Y_k$.

In Case \ref{case1} we assume that $B$ is $\mathbb{Z}_d$-graded and that $Y_k$ are all of grade $1$. In Case \ref{case2} we do not impose any special assumptions on the representation.

For all $n$-tuples of non-negative integers satisfying $d_1+\dots+d_n=d$, let $c_{d_1,\dots,d_n}$ be the coefficient of $u_1^{d_1} \dots u_n^{d_n}$ in $f$, and let $M_{d_1,\dots,d_n}$ be a copy of $M_d(F)$, graded as in Example \ref{gradeex}.

In Case \ref{case1}, we define a map $\Phi_f : C_f \rightarrow B \otimes_{\mathbb{Z}_d} \ _{c_{d_1,\dots,d_n} \neq 0} M_{d_1,\dots,d_n}$, taking each $x_k$ to $Y_k+\sum_{c_{d_1,\dots,d_n} \neq 0} x_{k;d_1,\dots,d_n}$.
In Case \ref{case2} we define it in a similar way, just with $\otimes$ instead of $\otimes_{\mathbb{Z}_d}$.

The element $x_{k;d_1,\dots,d_n}$ is defined as follows:
\begin{itemize}
\item If $d_k=0$ then it is the zero matrix.
\item If $d_k \neq 0$ and $d_1=\dots=d_{k-1}=0$ then it has $c_{d_1\dots,d_n}$ in the $(n,1)$ entry, $1$ in the entries $(1,2),\dots,(d_k-1,d_k)$, and $0$ elsewhere.
\item Otherwise, it has $1$ in the entries $(d_1+\dots+d_{k-1},d_1+\dots+d_{k-1}+1),\dots,(d_1+\dots+d_k-1,d_1+\dots+d_k)$ and $0$ elsewhere.
\end{itemize}

For example, if $f(u_1,u_2)=\alpha u_1^3+\beta u_1^2 u_2+\gamma u_2^3$ then $x_{1;2,1}=\left( \begin{array}{rrl} 0 & 1 & 0 \\ 0 & 0 & 0 \\ \beta & 0 & 0 \end{array} \right)$ and $x_{2;2,1}=\left( \begin{array}{rrl} 0 & 0 & 0 \\ 0 & 0 & 1 \\ 0 & 0 & 0 \end{array} \right)$.

\begin{thm}\label{posans}
$\Phi_f$ is a representation.
\end{thm}

\begin{proof}
\begin{eqnarray*}
(a_1 \Phi(x_1)+\dots+a_n \Phi(x_n))^d=((a_1 Y_1+\dots+a_n Y_n)+\\ \sum_{c_{d_1,\dots,d_n} \neq 0} (a_1 x_{1;d_1,\dots,d_n}+\dots+a_n x_{n;d_1,\dots,d_n}))^d.
\end{eqnarray*}

In Case \ref{case1}, because of the grading, we have \\$((a_1 Y_1+\dots+a_n Y_n)+\sum_{c_{d_1,\dots,d_n} \neq 0} (a_1 x_{1;d_1,\dots,d_n}+\dots+a_n x_{n;d_1,\dots,d_n}))^d=(a_1 Y_1+\dots+a_n Y_n)^d+\sum_{c_{d_1,\dots,d_n} \neq 0} (a_1 x_{1;d_1,\dots,d_n}+\dots+a_n x_{n;d_1,\dots,d_n})^d$.

In Case \ref{case2} we obtain the same equality because of the characteristic.

Finally, we have $(a_1 Y_1+\dots+a_n Y_n)^p=c_{d,0,\dots,0} a_1^d+\dots+c_{0,\dots,0,d} a_n^d$, and $(a_1 x_{1;d_1,\dots,d_n}+\dots+a_n x_{n;d_1,\dots,d_n})^d=c_{d_1,\dots,d_n} a_1^{d_1} \dots a_n^{d_n}$ and that completes the proof.
\end{proof}

\begin{cor}
The form $f$ always has a finite linearization.
\end{cor}

\begin{proof}
From Example \ref{excases} it is clear that a finite linearization of $g$ that satisfies the required conditions in Theorem \ref{posans} always exists. We obtain a finite linearization for $f$ too using that theorem.
\end{proof}

\begin{rem}
In Case \ref{case2}, since the grading plays no role, one can similarly construct a finite linearization of the diagonal part of $f$ for any given finite linearization of $f$.
\end{rem}

\begin{rem}\label{notirred}
The representation $\Phi_f$ is not necessarily irreducible, even if $\Phi_g$ is.

For example, let $f(a,b)=a^2+2 a b+b^2$ and then $g(a,b)=a^2+b^2$.

The Clifford algebra of $g$ is $C_g=F[y_1,y_2 : y_1^2=y_2^2=1, y_1 y_2=-y_2 y_1]=M_2(F)$.
This algebra is simple, and therefore is equal to all its homomorphic images. $\Phi_g$ can be the identity map on $M_2(F)$, having $Y_1=y_1=\left( \begin{array}{rl} 0 & 1\\1 & 0 \end{array} \right)$ and $Y_2=y_2=\left( \begin{array}{rl} 1 & 0\\0 & -1 \end{array} \right)$.
The $\mathbb{Z}_2$-grading of $M_2(F)$ can be according to the degrees of $y_1$, i.e. the elements $e_{1,1}$ and $e_{2,2}$ are of grade $0$ and $e_{1,2}$ and $e_{2,1}$ are of grade $1$.

The obtained $\Phi_f$ will map $C_f$ to $M_2(F) \otimes_{\mathbb{Z}_2} M_2(F)$ taking $x_1$ to $X_1=\left( \begin{array}{rl} 0 & 1\\1 & 0 \end{array} \right) \otimes_{\mathbb{Z}_2} \left( \begin{array}{rl} 1 & 0\\0 & 1 \end{array} \right)+\left( \begin{array}{rl} 1 & 0\\0 & 1 \end{array} \right) \otimes_{\mathbb{Z}_2} \left( \begin{array}{rl} 0 & 0\\2 & 0 \end{array} \right)$ and $x_2$ to $X_2=\left( \begin{array}{rl} 1 & 0\\0 & -1 \end{array} \right) \otimes_{\mathbb{Z}_2} \left( \begin{array}{rl} 1 & 0\\0 & 1 \end{array} \right)+\left( \begin{array}{rl} 1 & 0\\0 & 1 \end{array} \right) \otimes_{\mathbb{Z}_2} \left( \begin{array}{rl} 0 & 1\\0 & 0 \end{array} \right)$.

The algebra $C_f$ has only one two-sided ideal, the one generated by $x_1-x_2$, and it has exactly two images, itself and itself modulo its ideal. Since $X_1-X_2 \neq 0$, the image of $\Phi_f$, $F[X_1,X_2]$ is isomorphic to $C_f$, and is in particular not irreducible, unlike the image of $\Phi_g$ which is isomorphic to $C_g$ and therefore irreducible.
\end{rem}

\begin{rem}\label{notrank}
The rank of the obtained representation of $C_f$ is not necessarily equal to the rank of the initial representation of $C_g$.

For example, let $f(a,b)=a b$.
Then $g(a,b)=0$.

The base-field $F$ is a representation of $C_g$ of rank $1$. The obtained representation of $C_f$ will be $\Phi_f : C_f \rightarrow M_2(F)$ taking $x_1$ to $\left( \begin{array}{rl} 0 & 0\\1 & 0 \end{array} \right)$ and $x_2$ to $\left( \begin{array}{rl} 0 & 1\\0 & 0 \end{array} \right)$. $\Phi_f$ is an isomorphism. Clearly the rank is $4$, instead of $1$.

This example holds regardless of characteristic.
\end{rem}

\begin{conj}\label{samerankconj}
The obtained representation of $C_f$ is of rank no less than the rank of the initial representation of $C_g$.
\end{conj}

\begin{rem}
The existence of a representation of $C_g$ of rank $r$ does not imply the existence of a representation of $C_f$ of rank $r$, and vice versa.

In Remark \ref{notrank} we saw an example where $C_g$ has a representation of rank $1$, while $C_f$ is simple of rank $4$.

In the opposite direction, let us have a look again at the example given in Remark \ref{notirred}.
There exists a homomorphism $\Phi_f : C_f \rightarrow F[\mu:\mu^2=1]$ taking $x_1$ and $x_2$ to $\mu$. Here the representation is of rank $2$ while $C_g$ is simple of rank $4$.
\end{rem}

\subsection{A note on higher dimensional representations}
\sloppy
In \cite{Van}, Van den Bergh proved that binary forms of degree $\geq 4$ have finite linearizations of unbounded high ranks.

In \cite{Emre}, Emre, Kulkarni and Mustopa proved that ternary cubic forms have finite linearizations of unbounded high ranks.

However, these high rank finite linearizations are not easy to construct explicitly.
Here we shall present an explicit example of a simple finite linearization of degree $d^2$ and index $d$ for any $d \geq 4$ of any diagonal binary form of degree $d$:
\begin{exmpl}
Let $f(a,b)=\alpha a^d+\beta b^d$ be a fixed binary form of degree $d$ over $F$.
Let $A=(\gamma,\delta)_{d,F} \otimes (\mu,\nu)_{d,F}=F[x,y] \otimes F[z,w]$, such that $\alpha=(-1)^{d-1}(\gamma \delta+\mu \nu \gamma)$ and $\beta=\delta+\nu \gamma$.
Let $Y=y+w x$ and $X=y x+w z x$.
By a straight-forward calculation $y (w x+y x+w z x)=\rho (w x+y x+w z x) y$, $y x (w x+w z x)=\rho (w x+w z x) y x$ and $(w x) (w z x)=\rho (w z x) (w x)$. Therefore for any $a,b \in F$, $(a X+b Y)^d=b^d y^d+a^d (y x)^d+b^d (w x)^d+a^d (w z x)^d=a^d ((-1)^{d-1}(\gamma \delta+\mu \nu \gamma))+b^d (\delta+\nu \gamma)=\alpha a^d+\beta b^d$.
The elements $X$ and $Y$ generate the algebra $A$ over $F$.
\end{exmpl}

\begin{proof}
$Z=Y X-\rho X Y=(\rho^{-1}-\rho) y x w x$. This element satisfies $y Z=\rho^2 Z y$ and $(w x) Z=\rho^{-1} Z (w x)$.
Since $d \geq 4$, $\rho^{-1} \neq \rho^2$, and so by conjugating the element $Y$ by $Z$ we can show that $y,w x \in F[X,Y]$.
Now, by conjugating the element $X$ by $w x$ we can show that $y x,w z x \in F[X,Y]$.
The elements $y,y x,w x,w z x$ generate the algebra $A$ over $F$.
\end{proof}

\begin{conj}
We conjecture that the tensor product of cyclic algebras $A=\otimes_{k=1}^n (\alpha_k,\beta_k)_{d,F}=\otimes_{k=1}^n F[x_k,y_k]$ is generated by the elements  $X=\sum_{k=1}^n x_1^{-1} \dots x_{k-1}^{-1} x_k y_1 \dots y_{k-1}$ and $Y=\sum_{k=1}^n x_1^{-1} \dots x_{k-1}^{-1} y_1 \dots y_k$. It is easy to see that $X$ and $Y$ span a $d$-central subspace of $A$, and therefore $A$ would be a finite linearization of its exponentiation form. This way one could construct explicit finite linearizations of unbounded high ranks for any diagonal binary form.
If $A=F[X,Y]$ and Conjecture \ref{samerankconj} is true, then since this representation satisfies the conditions in Theorem \ref{posans}, we can conclude that any form of degree greater or equal to four in at least two variables has finite linearizations of unbounded high ranks.
\end{conj}

\section{The Clifford Algebra of a Short $p$-Central Space}\label{short}
Let $d$ be an integer, and $F$ be an infinite field of characteristic prime to $d$, containing a primitive $d$th root of unity $\rho$. Let $A$ be a central simple algebra over $F$ containing a $d$-central two-dimensional subspace $V=F v+F w$.
The Clifford algebra of $V$ is $C(V)=F[x,y : (a x+b y)^d=(a v+a w)^d \forall a,b \in F]$.

Let the underlying exponentiation form be $f(a,b)=\alpha_0 a^d+\alpha_1 a^{d-1} b+\dots+\alpha_d b^d$ for some coefficients $\alpha_0,\dots,\alpha_d \in F$.

According to Lemma \ref{eigencharnot}, we can decompose $w=w_0+w_1+\dots+w_{d-1}$ such that $w_i v=\rho^i v w_i$ for any $0 \leq i \leq d-1$, and $w_0=\frac{\alpha_1}{d \alpha_0} v$.

Similarly, inside the Clifford algebra, $y=y_0+y_1+\dots+y_{d-1}$ such that $y_i x=\rho^i x y_i$ for any $0 \leq i \leq d-1$, and $y_0=\frac{\alpha_1}{d \alpha_0} x$.

There is a natural homomorphism from $C(V)$ to $F[V]$ taking $x$ to $v$ and $y$ to $w$. This homomorphism also takes each $y_i$ to $v_i$.

Replacing $w$ by $w-\frac{\alpha_1}{d \alpha_0} v$, one can eliminate $w_0$, and therefore we can always assume that $\alpha_1=0$ from the beginning.

\begin{defn}\label{dfnshort}
Fixing the two basic elements $v$ and $w$, including their order, we say that $V$ is of type $T \subseteq \mathbb{Z}_d$ if $w_i=0$ for every $i \not \in T$. We call $V$ short if $T$ is of cardinality $\leq 2$.
\end{defn}

As we said before, $V=F v+F w$ is the image of the $d$-central subspace $V'=F x+F y$ of the Clifford algebra $C(V)$. Even if $V$ is short of type $\set{j,k}$, it does not mean that $V'$ is short.

\begin{exmpl}
Let $V=F v+F w$ be a subspace of the $5$-cyclic algebra $(\alpha,\beta)_{5,F}=F[v,w]$.
Its Clifford algebra $C(V)=F[x,y]$ has the image $B=(\gamma,\delta) \otimes (\frac{\beta-\gamma-\delta}{\gamma},\frac{\alpha}{\gamma \delta})=F[q,r] \otimes F[s,t]$ taking $x$ to $q r t$ and $y$ to $q+q s+r$. The image of $V'$ in $B$ is not short, and therefore $V'$ is not short.
\end{exmpl}

\begin{defn}\label{dfnclshort}
We define the Clifford algebra of the short $d$-central space of type $\set{j,k}$ to be $C_{j,k}(V)=C(V)/\left<y_m : m \neq j,k\right>$.
\end{defn}

In \cite{Chapman} we proved that the following holds if $d=p$ is a prime.
\begin{thm}\label{main7++}
\begin{enumerate}
\item $C_{1,p-1}(V)$ is Azumaya.
\item $Z(C_{1,p-1}(V))=F[X,Y : Y(Y - \alpha_p) =\alpha_0 X^p + p^{-p} \alpha_2^p\alpha_0^{2-p}]$
\item There is a one-to-one correspondence between the $\bar{F}$-rational points on this curve and the simple homomorphic images of $C_{1,p-1}(V)$. Every such image is either $(\alpha_p,\alpha_0)_{p,F}$ (the one corresponding to the point at infinity) or $(\alpha_0,t)_{p,K}$ where $K = F[s,t]$ and $(s,t)$ is an $\bar{F}$-rational point with $t \neq 0$.
\end{enumerate}
\end{thm}

This generalizes the main result in \cite{Haile}, since in the case of $p=3$, $C(V)=C_{1,2}(V)$.

For the case of $p=5$ we actually studied all possible kinds of short $p$-central spaces.
There are essentially only two kinds, because a change of the $p$th root of unity sends $T=\set{j,k}$ to $\lambda T=\set{\lambda j,\lambda k}$ for some $\lambda \in \mathbb{Z}_d^\times$, so there is no real difference between the types $\set{1,4}$ and $\set{2,3}$, and similarly all the other types become $\set{1,3}$.

We proved that if $V$ is of type $\set{1,3}$, then every simple homomorphic image of the Clifford algebra is a product of one or two cyclic division algebras of degree $5$, whose center is some field extension of $F$. Explicit examples were given for both types of images.

In case that the exponentiation form is diagonal, we calculated all the simple images of $C_{1,3}(V)$ explicitly:
\begin{thm}\label{M2}
If $f(a,b)=\alpha a^5+\beta b^5$ then any simple image $A$ of $C_{1,3}(V)$ is one of the following:
\begin{enumerate}
\item $A = (\alpha,\beta^2)_F$.
\item $A = (\alpha,t)_K$ where $K = F(t,s)$ and $s^5 = \alpha^{3} t^2(\beta - t)$.
\item $A = (\alpha,t)_K \tensor[K] (t',t'')_K$ where $K = F(t,t',t'')$ and $t^3+\alpha t'+\alpha^2 t'' = \beta t^2$.
\item $A = (\alpha,t)_K \tensor[K] (t',t'')_K$ where $K = F(t,t',t'',s)$, and  $s^5 = \alpha^3 tt'{t''}^2(\beta t^2 - t^3 - \alpha^2 t t' - \alpha t'')$.
\end{enumerate}
\end{thm}

This provides examples of an algebra of degree $25$ and exponent $5$ that is generated by a short $5$-central space. In the following example we shall see how for any diagonal form of prime degree $p$ there exists an algebra of degree $p^2$ and exponent $p$ generated by a short $p$-central space whose exponentiation is equal to the given one:
\begin{exmpl}
Let $f(a,b)=\alpha a^p+\beta b^p$ be a diagonal form of degree $p$.
Let $A=F[x,y] \otimes F[z,w]$.
Let $X=w y x$, $Y=y+x^2+z^2 y x$.
Now, $(a X+b Y)^p=(a w y x+b y+b x^2+b z^2 y x)^p$.
Since $(a w y x+b y+b z^2 y x)(b x^2)=\rho^2 (b x^2)(a w y x+b y+b z^2 y x)$, $(b y)(a w y x+b z^2 y x)=\rho(a w y x+b z^2 y x)(b y)$ and $(a w y x)(b z^2 y x)=\rho^2 (b z^2 y x)(a w y x)$, we have $(a X+b Y)^p=a^p (a w y x)^p+b^p (y^p+(x^2)^p+(z^2 y x)^p)=a^p X^p+b^p Y^p$.
It is easy to see that $X$ and $Y$ generate $A$.
\end{exmpl}

\section{The Clifford Algebra of a Short $4$-Central Space}

In this section we shall present some results concerning short $d$-central spaces when $d=4$.

Let $F$ be an infinite field of characteristic not $2$, containing a primitive fourth root of unity $i$.
Let $V=F v+F w$ be a $4$-central subspace of some division algebra $A$. Let us assume that $V$ is short of type $\set{j,k}$ for some $1 \leq j < k \leq 4$ (see Definition \ref{dfnshort}). Let $f(a,b)=\alpha_0 a^4+\alpha_1 a^3 b+\alpha_2 a^2 b^2+\alpha_3 a b^3+\alpha_4 b^4$ be its underlying exponentiation form. We would like to study its Clifford algebra $C_{j,k}(V)$ (see Definition \ref{dfnclshort}).

Assuming $\alpha_0 \neq 0$, by replacing $w$ with $w-\frac{\alpha_1}{4 \alpha_0} v$ we may assume that $\alpha_1=0$.

The Clifford algebra $C(V)$ is generated over $F$ by $x$ and $y$ subject to the relations
\begin{eqnarray}\label{cl4rel1}
x^4=\alpha_0
\end{eqnarray}
\begin{eqnarray}\label{cl4rel2}
x^3 * y=\alpha_1=0
\end{eqnarray}
\begin{eqnarray}\label{cl4rel3}
x^2 * y^2=\alpha_2
\end{eqnarray}
\begin{eqnarray}\label{cl4rel4}
x * y^3=\alpha_3
\end{eqnarray}
\begin{eqnarray}\label{cl4rel5}
y^4=\alpha_4
\end{eqnarray}
for some $\alpha_0,\dots,\alpha_4 \in F$.

Let us assume that $\alpha_0 \neq 0$.

As proven in Lemma \ref{eigencharnot}, we can decompose $y=y_0+y_1+y_2+y_3$ where $y_n x=i^n x y_n$ for $1 \leq n \leq 4$. However, $y_0=\frac{\alpha_1}{d \alpha_0} x=0$.

From $x^2 * y^2=\alpha_2$ we obtain
\begin{eqnarray}\label{eqalpha2}
x^2((2+2 i) y_1 y_3+(2-2 i)y_3 y_1+2 y_2^2)=\alpha_2
\end{eqnarray}

The algebra $C_{j,k}(V)$ is defined as in Definition \ref{dfnclshort}.

For the case of $\set{j,k}=\set{1,p-1}$ we obtained a result similar to what we got for prime $p$.
\begin{thm}
Assuming $\alpha_4 \neq \frac{\alpha_2^2}{8 \alpha_0}$ and $\alpha_3=0$, if $\set{j,k}=\set{1,3}$ then $C_{j,k}(V)$ is Azumaya whose center is a hyper-elliptic curve and every simple image of it is cyclic of degree $4$. If $\set{j,k}=\set{1,3}$ and $\alpha_3 \neq 0$ then $C_{j,k}(V)=0$.
\end{thm}

\begin{proof}
According to the definition of $C_{1,3}(V)$, $y_2=0$.
Then Equation (\ref{eqalpha2}) becomes
\begin{eqnarray}\label{eqalpha2type13}
x^2((2+2 i) y_1 y_3+(2-2 i)y_3 y_1)=\alpha_2.
\end{eqnarray}
Relation (\ref{cl4rel5}) becomes, by conjugation by $x$, $\alpha_4=y_1^4+y_1^2*y_3^2+y_3^4$.
Substituting Equation (\ref{eqalpha2type13}) in that relation leaves $\alpha_4=y_1^4+y_3^4+\frac{\alpha_2^2}{8 \alpha_0}$.
Since the algebra is generated by $x,y_1,y_3$, and $y_1^4$ commutes with those three, $y_1^4$ is central.
A straightforward calculation shows that $w=(y_1 y_3+\frac{i \alpha_2}{4} x^{-2})x^{-1}$ commutes with the generators, and therefore is central too.

Let $K=F[y_1^4,w]$. Let us have a look at $C_f \otimes q(K)$.
In this algebra, $y_1$ is invertible. Let $t_2=-\frac{\alpha_2 i}{4} y_1^{-1} x^{-2}$ and $t_1=y_3-t_2$.
Substituting $y_3=t_1+t_2$ in Equation (\ref{eqalpha2type13}) we get $y_1 t_1=i t_1 y_1$. From the relation $y_3 x=-i x y_3$, by conjugation by $y_1$ we obtain $t_1 x=-i x t_1$.
Consequently, $t_1 \in q(K) y_1^{-1} x$, which means that $C_f \otimes q(K)$ is generated by $x$ and $y_1$ over its center, and therefore $C_f \otimes q(K)=(\alpha_0,y_1^4)_{4,q(K)}$.

The center of $C_f$ is the intersection of the center of $C_f \otimes q(K)$ and $C_f$, which is the ring generated over $F$ by $y_1^4$ and $w$.
Now, $\alpha_0 w^4=(x w)^4=y_1^8+(\frac{\alpha_2^2}{8 \alpha_0}-\alpha_4) y_1^4+\frac{\alpha_2^4}{256 \alpha_0^2}$. Setting $s=y_1^4$ and $r=w$, the center of $C_f$ is the coordinate ring of the elliptic curve
\begin{eqnarray}\label{cl4elliptic}
s^2+(\frac{\alpha_2^2}{8 \alpha_0}-\alpha_4) s=\alpha_0 r^4-\frac{\alpha_2^4}{256 \alpha_0^2}.
\end{eqnarray}

Since $\alpha_4 \neq \frac{\alpha_2^2}{8 \alpha_0}$, in every simple homomorphic image of $C_f$, either $y_1$ or $y_3$ is nonzero.
Due to similar calculations as in the algebra $C_f \otimes q(K)$, every simple homomorphic image of $C_f$ is a cyclic algebra of degree $4$ over a field that is a quotient of $Z(C_f)$, and therefore is of the form $F[r_0,s_0]$ where $(r_0,s_0)$ is an $\bar{F}$-rational point on the elliptic curve (\ref{cl4elliptic}).

Let $A$ be a simple homomorphic image of $C_f$, and let $s_0$, $r_0$, $\bar{x}$, $\bar{y_1}$ and $\bar{y_3}$ be the images of $s, r, x, y_1$ and $y_3$ respectively.
Both $s_0$ and $r_0$ are in the center of $A$, which is a field, and therefore if they are nonzero then they are invertible.
If $s_0 \neq 0$ then $A$ is also an image of the algebra $C_f \otimes q(K)$, and then $A=(\bar{x}^4,\bar{y_1}^4)_{4,F[s_0,r_0]}=(\alpha_0,s_0)_{4,F[r_0,s_0]}$.
Otherwise, $\bar{y_3}^4$ must be nonzero.
It is easy to see that $A$ is generated by $\bar{x}$ and $\bar{y_3}$. If $\alpha_2=0$ then $r_0$ must be $0$ too, and then $A=(\bar{y_3}^4,\bar{x}^4)_{4,F}=(\alpha_4,\alpha_0)_{4,F}$. If $\alpha_2 \neq 0$ then $(\frac{4 \alpha_0}{\alpha_2} r_0)^4=\alpha_0$, and so $A$ is the $2 \times 2$ matrix algebra over either $F[\mu : \mu^4=\alpha_0]$ (if $\alpha_0$ has no fourth root in $F$) or over $F$ (otherwise).

There is therefore a one-to-one correspondence between the $\bar{F}$-rational points on Curve (\ref{cl4elliptic}) and the simple homomorphic images of $C_f$, that are all cyclic of degree $4$, which means that $C_f$ is Azumaya.

The last relation to take into consideration is Relation (\ref{cl4rel4}), which under these circumstances becomes $x*y_1^3+x*y_1^2*y_3+x*y_1*y_3^2+x*y_3^3=\alpha_3$.
Taking only the part that commutes with $x$, we obtain $0=\alpha_3$.
The other parts become trivial by applying the other relations.
\end{proof}

The two other possible types are $\set{1,2}$ and $\set{2,3}$. The latter can be obtained from the first by changing $i$ to $-i$, and therefore they are essentially the same.

\begin{thm}
If $V$ is short of type $\set{1,2}$ and its exponentiation form $f(a,b)$ is indecomposable then $C_{1,2}$ is a symbol algebra of degree $4$ over a commutative ring over $F$. The explicit formulas are as follows:
\begin{enumerate}
\item If $\alpha_2 \neq 0$ and $\alpha_3 \neq 0$ then $C_{j,k}(V)=(\alpha_0,\frac{\alpha_3^2 \alpha_2}{2 \alpha_0})_{4,K}=K[\eta,\mu]$ where $K=F[\delta: \delta^4=\frac{2 \alpha_0^2 \alpha_4}{\alpha_3^2 \alpha_2}-\frac{2 \alpha_0 \alpha_2}{4 \alpha_3^2}-\frac{\alpha_0}{16}]$. In this case, $V=F \eta+F (\mu+\mu \eta^{-1} \delta+\frac{\alpha_3 i}{4} \mu^{-2} \eta^{-1})$.
\item If $\alpha_2 \neq 0$, $\alpha_3=0$ and $\alpha_4 \neq \frac{\alpha_2^2}{4 \alpha_0}$ then $C_{j,k}(V)=(\alpha_0,\alpha_4-\frac{\alpha_2^2}{4 \alpha_0})_{4,K}$ where $K=F[\mu,\nu : \mu^2=0, \nu^2=\frac{\alpha_2}{2 \alpha_0}(\alpha_4-\frac{\alpha_2^2}{4 \alpha_0}),\mu \nu=\nu \mu]$
\item If $\alpha_2=0$, $\alpha_3 \neq 0$ and $\alpha_4 \neq 0$ then $C_{j,k}(V)=(\alpha_0,\alpha_4)_{4,K}$ where $K=F[\mu : \mu^2=0]$.
\item If $\alpha_2=\alpha_3=0$ and $\alpha_4 \neq 0$ then $C_{j,k}(V)=(\alpha_0,\alpha_4)_{4,K}$ where $K=F[\mu,\nu: \mu^2=\nu^2=\mu \nu-\nu \mu=0]$.
\end{enumerate}
\end{thm}

\begin{proof}
The type is $\set{1,2}$, which means that $y=y_1+y_2$.
Therefore, Equation (\ref{eqalpha2}) becomes $x^2(2 y_2^2)=\alpha_2$, which means that $y_2^2=\frac{\alpha_2}{2} x^{-2}$.

Assume $\alpha_2 \neq 0$. Then $y_2^2 y_1+y_1 y_2^2=0$, and therefore $y_1=t_1+t_3$ such that $t_k y_2=i^k y_2 t_k$.

From Relation (\ref{cl4rel4}) we obtain (by conjugation by $x$)
\begin{eqnarray}\label{cl4el4p1}
x*y_1^2*y_2=\alpha_3
\end{eqnarray}
and
\begin{eqnarray}\label{cl4el4p2}
x*y_1*y_2^2=0.
\end{eqnarray}

From the Equation (\ref{cl4el4p1}) we obtain
\begin{eqnarray}\label{cl4el4p1d1}
(1+i)x y_1^2 y_2+2x y_1 y_2 y_1+(1-i) x y_2 y_1^2=\alpha_3
\end{eqnarray}
and from the Equation (\ref{cl4el4p2}) we obtain
\begin{eqnarray}\label{cl4el4p1d2}
(1+i) x y_1 y_2^2+(1+(-1)+1+i) x y_2^2 y_1=0.
\end{eqnarray}

Equation (\ref{cl4el4p1d2}) is trivial. From Equation (\ref{cl4el4p1d1}) we obtain
\begin{eqnarray}\label{cl4el4p1d1k}
(2 i+2) t_1 t_3+(2-2 i)t_3 t_1-4 i t_3^2=\alpha_3 y_2^{-1} x^{-1}.
\end{eqnarray}

By conjugation by $y_2$ we obtain the following two relations:
\begin{eqnarray}\label{cl4el4p1d1k1}
(2 i+2) t_1 t_3+(2-2 i)t_3 t_1=0,
\end{eqnarray}
\begin{eqnarray}\label{cl4el4p1d1k2}
-4 i t_3^2=\alpha_3 y_2^{-1} x^{-1}.
\end{eqnarray}

From Equation (\ref{cl4el4p1d1k1}) we obtain $t_1 t_3=i t_3 t_1$.

From Equation (\ref{cl4el4p1d1k2}) we obtain $t_3^2=\frac{\alpha_3 i}{4}y_2^{-1} x^{-1}$.

Assume $\alpha_3 \neq 0$.
Then $y_2=\frac{\alpha_3 i}{4}t_3^{-2} x^{-1}$.

Now, $\alpha_4=(y_1+y_2)^4$. By conjugation by $x$, $(y_1+y_2)^4=y_1^4+y_2^4$, and since $t_1 t_3=i t_3 t_1$, $y_1^4=t_1^4+t_3^4$.
Hence $\alpha_4=t_1^4+\frac{\alpha_3^2 \alpha_2}{32}+\frac{\alpha_2^2}{4 \alpha_0}$, which means that
$t_1^4=\alpha_4-\frac{\alpha_3^2 \alpha_2}{32}-\frac{\alpha_2^2}{4 \alpha_0}$.

Setting $\eta=x$, $\mu=t_3$ and $\delta=x t_3^{-1} t_1$, we have $K=F[\delta]$ as the center of $C_{1,2}(V)$.
$C_{1,2}(V)=F[\eta,\mu] \tensor K=(\alpha_0,\frac{\alpha_3^2 \alpha_2}{2})_{4,F} \tensor K$.
$V=F x+F y=F x+F (y_1+y_2)=F x+F(t_3+t_1+y_2)=F \eta+F (\mu+\mu \eta^{-1} \delta+\frac{\alpha_3 i}{4} \mu^{-2} \eta^{-1})$.

Assume $\alpha_2 \neq 0$ and $\alpha_3=0$. Then $t_3^2=0$. Consequently $t_1^4=\alpha_4-\frac{\alpha_2^2}{4 \alpha_0}$.
Assume $\alpha_4-\frac{\alpha_2^2}{4 \alpha_0} \neq 0$. Setting $\eta=x$, $\mu=t_1$, $\gamma=x^{-1} t_3 t_1^{-1}$ and $\delta=x^{-1} t_1^2 y_2$ we have $K=F[\gamma,\delta]$ as the center of $C_{1,2}(V)$. $C_{1,2}(V)=F[\eta,\mu] \tensor K=(\alpha_0,\alpha_4-\frac{\alpha_2^2}{4 \alpha_0})_{4,F}$.

If $\alpha_4=\frac{\alpha_2^2}{4 \alpha_0}$ then $t_1^4=0$.
Therefore $f(u,v)=\alpha_0 u^4+\alpha_2 u^2 v^2+\frac{\alpha_2^2}{4 \alpha_0} v^4
=\alpha_0 (u^2+\frac{\alpha_2}{2 \alpha_0} v^2)^2$. In this case $f(a,b)$ is decomposable, contradictory to the assumption.

If $\alpha_2=0$ then we cannot decompose $y_1$ to $t_1+t_3$ by to conjugation by $y_2$ because $y_2$ is not invertible. However, from $(1+i) y_1^2 y_2+2 y_1 y_2 y_1+(1-i) y_2 y_1^2=\alpha_3 x^{-1}$ we obtain $y_2=t_1+t_2+t_3$ where $t_3=\frac{\alpha_3}{4} i y_1^{-2} x^{-1}$ and $y_1 t_k=i^k t_k y_1$.

We have $\alpha_4=y_1^4+y_2^4=y_1^4$. Assume that $\alpha_4 \neq 0$.
Now, $y_2^2=0=t_1^2+t_2^2+t_3^2+t_1 t_2+t_2 t_1+t_1 t_3+t_3 t_1+t_2 t_3+t_3 t_2=t_1^2+t_2^2+t_3^2+t_1 t_2+t_2 t_1+2 t_1 t_3$.
By conjugation by $y_1$, $t_1^2+t_3^2=0$, $t_1 t_2+t_2 t_1=0$ and $t_2^2+2 t_1 t_3=0$.
Assume $\alpha_0 \neq 0$ and so $t_1 \in F[t_2,x,y_1]$.
$K=F[x^2 y_1^2 t_2]$ is the center of $C_{1,2}(V)$. $C_{1,2}(V)=(\alpha_0,\alpha_4)_K$.

If $\alpha_3=0$ then $t_1^2=t_2^2=t_1 t_2-t_2 t_1=0$.
$K=F[x^2 y_1^2 t_2, x^{-1} y_1^2 t_1]$ is the center of $C_{1,2}(V)$. $C_{1,2}(V)=(\alpha_0,\alpha_4)_{4,K}$.
\end{proof}

\begin{rem}
If $\alpha_2 \neq 0$, $\alpha_3=0$ and $\alpha_4=\frac{\alpha_2^2}{4 \alpha_0}$ then $f(u,v)=\alpha_0 u^4+\alpha_2 u^2 v^2+\frac{\alpha_2^2}{4 \alpha_0} v^4
=\alpha_0 (u^2+\frac{\alpha_2}{2 \alpha_0} v^2)^2$. In this case $C_{j,k}(V)/\langle t_1,t_3 \rangle=F[x,y_2]=(\mu,\frac{\alpha_2}{2} \mu^{-1})_{2,K}$ where $K=F[\mu : \mu^2=\alpha_0]$.
This means that $\rad(C_{j,k}(V))=\langle t_1,t_3 \rangle$.
\end{rem}

\section{The Clifford Algebra of a Monic Polynomial}\label{Kuos}
This section is based on a joint work with Jung-Miao Kuo.

In \cite{Pappacena}, Pappacena generalized the notion of the Clifford algebra to the algebra associated to a monic polynomial (with respect to the first variable) $\Phi(z,a_1,\dots,a_n)=z^d-\sum_{k=1}^d f_k(a_1,\dots,a_n) z^{d-k}$, where each $f_k$ is a homogeneous form of degree $k$.
This algebra, denoted there by $C_\Phi$, is defined to be
\begin{equation*}
\begin{aligned}
F\langle x_1,\dots,x_n\colon &  (a_1 x_1+\dots+a_n x_n)^d\\
& = f_1(a_1,\dots,a_n) (a_1 x_1+\dots+a_n x_n)^{d-1}+\dots+\\
&f_{d-1}(a_1,\dots,a_n) (a_1 x_1+\dots+a_n x_n)+f_d(a_1,\dots,a_n)\\
& \textrm{ for all } a_1,\dots,a_n \in F\rangle,
\end{aligned}
\end{equation*}

Pappacena proved in that paper that if $d=2$ then this algebra is isomorphic to the Clifford algebra of a quadratic form, and therefore its structure is known.

In \cite{Kuo}, Kuo studied the Clifford algebra of the polynomial $\Phi(z,a,b)=z^3-e a b z-f(a,b)$ and the results are very similar to the results Haile obtained in \cite{Haile}. The formulas for the simple images of the Clifford algebra are provided there only in case $f(a,b)$ is diagonal.

In this section we study the Clifford algebra of the monic polynomial $\Phi(z,a,b)=z^3-f_1(a,b) z^2-f_2(a,b) z-f_3(a,b)$ in two distinct cases:
\begin{enumerate}
\item $\charac{F} \neq 2,3$, $f_1(a,0)=f_2(a,0)=0$.
\item $\charac{F}=3$, $f_1(a,b)=0$ and $f_2(a,b)=e a b$.
\end{enumerate}

\subsection{Case 1}

Let $\Phi(z,a,b)=z^3-r b z^2-(e a b+t b^2) z-(\alpha a^3+\beta a^2 b+\gamma a b^2+\delta b^3)$ where $r,t,e,\alpha,\delta,\beta,\gamma \in F$.

The algebra $C_\Phi$ is by definition
\begin{equation*}
\begin{aligned}
C_\Phi=F\langle x,y\colon &  x^3=\alpha,\\
 & x^2 * y=r x^2+e x+\beta,\\
& x * y^2=r x y+r y x+t x+e y+\gamma,\\
 & y^3=r y^2+t y+\delta\rangle.
\end{aligned}
\end{equation*}
We assume that $\alpha\neq 0$ and $F$ contains a primitive third root of unity $\rho$. According to Lemma \ref{eigencharnot}, since $x^3 \in F^{\times}$, there exist $y_0, y_1, y_2\in C_{\Phi}$ such that
\begin{equation}\label{decomp}
y=y_0+y_1+y_2\ \textrm{ and }\ y_i x=\rho^{i} x y_i.
\end{equation} In this case, we say that $y_i$ $\rho^i$-commutes with $x$. Under this decomposition, the relation $x^2 * y=r x^2+e x+\beta$ is equivalent to  $y_0=(3 \alpha)^{-1} (e x^2+\beta x+\alpha r)$.
Substituting this in the relation $x * y^2=r x y+r y x+t x + e y+\gamma$, we obtain by a straight-forward calculation that
\begin{equation}\label{y1y2}
y_1 y_2=\rho y_2 y_1+\frac{(1-\rho)D_1}{3 \alpha} x^2-\frac{(1-\rho)D_2}{9 \alpha}
\end{equation}
where $$D_1=\gamma+\frac{e r}{3} -\frac{\beta^2}{3 \alpha}\ \textrm{ and }\ D_2= e \beta - 3 \alpha t-\alpha r^2.$$
Let $$w=x^{-1} y_2 y_1+\rho^2 \frac{D_1}{3 \alpha} x+\frac{D_2}{9 \alpha} x^{-1}.$$

\begin{lem}\label{central}
The elements $w,y_1^3$ and $y_2^3$ are in the center of $C_\Phi$.
\end{lem}
\begin{proof}
Clearly, $w$ commutes with $x$. By Equation \eqref{y1y2} we see that $y_i w = wy_i$, $i=1, 2$. Thus, $w$ commutes with $y$ and hence is central in $C_\Phi$. Similarly, one can check that $y_1^3$ and $y_2^3$ are central in $C_\Phi$.
\end{proof}

The relation $y^3=r y^2+t y+\delta$ may be split into three parts due to conjugation by $x$. The part on the left-hand side which $\rho$-commutes with $x$ is $y_0^2 * y_1+y_1^2 * y_2+y_2^2 * y_0$ and on the right-hand side it is $r y_0 * y_1+r y_2^2+t y_1$.
A direct computation shows that $y_0^2 * y_1=(3\alpha)^{-1}e\beta y_1 + 3^{-1}r^2 y_1- (3 \alpha)^{-1}\rho e r x^2 y_1-(3 \alpha)^{-1}\rho^2 \beta r x y_1$,   $y_1^2 * y_2=-(3 \alpha)^{-1}D_2 y_1$, $y_2^2 * y_0=r y_2^2$ and $r y_0 * y_1=3^{-1}2 r^2 y_1-(3 \alpha)^{-1}\rho e r x^2 y_1-(3 \alpha)^{-1}\rho^2 \beta r x y_1$. Thus, $y_0^2 * y_1+y_1^2 * y_2+y_2^2 * y_0$ is equal to $r y_0 * y_1+r y_2^2+t y_1$. Similarly, the part on the left-hand side which $\rho^2$-commutes with $x$ is equal to that on the right-hand side: $y_0^2 * y_1+y_1^2 * y_2+y_2^2 *  y_0 = r y_0 * y_2+r y_1^2+ty_2$.
So let us consider the parts on both sides which commute with $x$:
\begin{equation}\label{y^3}
\begin{split}\left(\frac{1}{3 \alpha} (e x^2+\beta x+\alpha r)\right)^3+y_1^3+y_2^3+\frac{1}{3 \alpha} (e x^2+\beta x+\alpha r) * y_1 * y_2=\\r y_0^2+r y_1 * y_2+t y_0 +\delta.
\end{split}
\end{equation}
We first compute
\begin{equation*}
\begin{aligned}
& (e x^2+\beta x) * y_1 * y_2\\
=\ & e((2+\rho^2)x^2 y_1 y_2 +(2+\rho) x^2 y_2 y_1) + \beta((2+\rho)xy_1 y_2 + (2+\rho ^2)xy_2 y_1)\\
=\ &  e\left(-3\rho^2 x^2 y_2 y_1-\rho D_1 x+\frac{\rho D_2}{3\alpha}x^2\right) + \beta\left(D_1-\frac{D_2}{3 \alpha} x\right) \\
=\ & e\left(-3\rho^2 \alpha w-\frac{D_2}{3 \alpha} x^2\right) +\beta\left(D_1-\frac{D_2}{3 \alpha} x\right)\\
\end{aligned}
\end{equation*}
where the second equality holds by applying Equation \eqref{y1y2}. Substituting this in Equation \eqref{y^3} we then obtain by another straight-forward calculation that
\begin{equation}\label{y1^3y2^3}
D+y_1^3+y_2^3-\rho^2 e w=0,
\end{equation}
where $$D=\frac{e^3}{27 \alpha}+\frac{\beta^3}{27 \alpha^2}-\frac{2 r^3}{27}+\frac{\beta}{3 \alpha} D_1-\frac{r t}{3}-\delta.$$
Consequently, via the decomposition in \eqref{decomp} with $y_0$ taken as $(3 \alpha)^{-1} (e x^2+\beta x+\alpha r)$, $C_{\Phi}$ is an $F$-algebra generated by $x, y_1, y_2$ subject to the relations $x^3=\alpha$, $y_i x=\rho^i xy_i$ and Equations \eqref{y1y2}, \eqref{y1^3y2^3}. Thus we have the following result.

\begin{lem}\label{27fg}
As an  $F[y_1^3, w]$-module, $C_\Phi$ is finitely generated by the 27 elements $x^i y_1^j y_2^k$, where $0\leq i,j,k\leq 2$.
\end{lem}

Let us consider the algebra after the localization $C_\Phi[y_1^{-3}]$.
Since in this algebra $y_1$ is invertible, from the choice of $w$, we have
\begin{equation}\label{y2}
y_2=x y_1^{-1} w-\rho^2 \frac{D_1}{3 \alpha} x^2 y_1^{-1}-\frac{D_2}{9 \alpha} y_1^{-1},
\end{equation}
and so
\begin{equation*}
\begin{aligned}
  & y_2^3 
=\alpha y_1^{-3} w^3-\frac{D_1^3}{27 \alpha} y_1^{-3}-\frac{D_2^3}{729 \alpha^3} y_1^{-3}-\rho^2 \frac{D_1D_2}{9 \alpha} wy_1^{-3}.
\end{aligned}
\end{equation*}
Therefore, substituting this in Equation \eqref{y1^3y2^3} we get
$$D+y_1^3+\alpha y_1^{-3} w^3-\frac{D_1^3}{27 \alpha} y_1^{-3}-\frac{D_2^3}{729 \alpha^3} y_1^{-3}-\rho^2 \frac{D_1D_2}{9 \alpha} wy_1^{-3}-\rho^2 e w=0.$$
Consequently,
\begin{equation}\label{y1^3w}
(D-\rho^2 e w) y_1^3+y_1^6+\alpha w^3-\frac{D_1^3}{27 \alpha}-\frac{D_2^3}{729 \alpha^3}-\rho^2 \frac{D_1D_2}{9 \alpha} w=0.\end{equation}
The last equality also holds in $C_\Phi$.

We next show that the center $Z$ of $C_\Phi$ is $F[y_1^3,w]$ and it is isomorphic to the coordinate ring of the affine elliptic curve
\begin{equation}\label{E}
E\colon (D-\rho^2 e R) S+S^2+\alpha R^3-\frac{D_1^3}{27 \alpha}-\frac{D_2^3}{729 \alpha^3}-\rho^2 \frac{D_1D_2}{9 \alpha} R=0,
\end{equation}
where the discriminant is assumed to be nonzero. Let $E$ also denote the elliptic curve with affine piece given by Equation \eqref{E}.

\begin{prop}\label{isom}
There is an $F$-algebra isomorphism from $C_\Phi[y_1^{-3}]$ into the symbol algebra $(\alpha, S)_{3,F(E)}$ over the function field $F(E)$ of the elliptic curve $E$.
\end{prop}
\begin{proof}
Let $u, v$ be the generators of $(\alpha, S)_{3,F(E)}$ satisfying $u^3=\alpha, v^3=S$ and $vu=\rho uv$.
Let $\phi$ be the $F$-algebra homomorphism from $C_\Phi$ into $(\alpha, S)_{3, F(E)}$ defined as follows
\begin{equation*}
\begin{aligned}
\phi\colon C_\Phi\
 & \rightarrow\  (\alpha, S)_{3,F(E)}\\
x\ & \mapsto\  u \\
y_1\ & \mapsto\  v\\
y_2\ & \mapsto\  u\left(R-\rho^2\frac{D_1}{3\alpha}u-\frac{D_2}{9\alpha}u^{-1}\right)v^{-1}.
\end{aligned}
\end{equation*}
One can check that $x^3=\alpha, y_i x=\rho^i xy_i$ and the relations in Equations \eqref{y1y2} and \eqref{y1^3y2^3} are preserved under the map $\phi$. Thus it is well-defined and $\phi(w)=R$. Furthermore, it induces a homomorphism from $C_\Phi[y_1^{-3}]$ to $(\alpha, S)_{3, F(E)}$, which we also denote by $\phi$.

Notice that from Equations \eqref{y2} and \eqref{y1^3w}, $C_\Phi[y_1^{-3}]$ as an $F[y_1^{\pm 3}]$-module is finitely generated by the 27 elements $x^i y_1^j w^k$, where $0\leq i,j,k\leq 2$. Since the images of these elements are linearly independent over $F[S^{\pm 1}]$ and $\phi$ when restricted to $F[y_1^{\pm 3}]$ is injective, it follows that $\phi$ itself is injective.
\end{proof}

\begin{cor}\label{Z}
The center $Z$ of $C_\Phi$ is $F[y_1^3, w]$ and it is isomorphic to the coordinate ring $F[E]$ of the affine elliptic curve $E$.
\end{cor}
\begin{proof}
By Proposition \ref{isom}, $F[y_1^3, w]\cong F[E]$, a Dedekind domain. Furthermore, we see from its proof that $\phi(C_\Phi)F(E)=(\alpha, S)_{3,F(E)}$. In particular, the center of $\phi(C_\Phi)$ is contained in $F(E)$. Therefore $F[E]=\phi(F[y_1^3, w])\subseteq \phi(Z) \subseteq F(E)$. It follows from Lemma \ref{27fg} and the injectivity of $\phi$ that $Z=F[y_1^3, w]\cong F[E]$.
\end{proof}

Now the center of $C_\Phi[y_1^{-3}]$ is $Z_{(y_1^3)}=F[y_1^{\pm 3}, w]\cong F[E]_{(S)}$ in which $y_1^3$ is invertible. Thus we have the following result.

\begin{cor}\label{L1}
$C_\Phi[y_1^{-3}]$ is the symbol Azumaya algebra $(\alpha, y_1^3)_{3,F[y_1^{\pm 3},w]}$. Similarly, $C_\Phi[y_2^{-3}]=(y_2^3, \alpha)_{3,F[y_2^{\pm 3},w]}$.
\end{cor}

From now on, we restrict ourselves to the following conditions:
$D\neq 0$ and the subalgebra $F[x: x^3=\alpha ]$ is a field.
\begin{prop}\label{eitheror}
In every homomorphic image of $C_\Phi$, either $y_1^3 \neq 0$ or $y_2^3 \neq 0$. In particular, if the image is simple then either $y_1^3$ or $y_2^3$ is invertible.
\end{prop}

\begin{proof}
Assume to the contrary that $y_1^3=y_2^3=0$. Then by Equation \eqref{y1^3y2^3}, $D=\rho^2 ew$. If $e=0$, then $D=0$, a contradiction. If $e\neq 0$, then by the choice of $w$, we have that $y_2 y_1=\rho e^{-1}Dx -(3 \alpha)^{-1}\rho^{2} D_1 x^2 - (9 \alpha)^{-1}D_2$, which is invertible as a nonzero element of the field $F[x]$.
However this means that $y_1$ is invertible too, which is a contradiction.
\end{proof}

\begin{cor}
The algebra $C_\Phi$ is Azumaya of rank 9.
\end{cor}
\begin{proof}
By Corollary \ref{Z} and Lemma \ref{27fg}, $C_\Phi$ is finitely generated as a module over its center $Z=F[y_1^3, w]$. For every maximal ideal $m$ of $Z$, it follows from Proposition \ref{eitheror} and Corollary \ref{L1} that $C_\Phi/mC_\Phi$ is a central simple algebra of degree 3 over the field $Z/m$. Therefore, $C_\Phi$ is Azumaya of rank 9.
\end{proof}

\begin{rem}
Another way to prove that $C_\Phi$ is Azumaya is the following: Every $\Phi(Z,X,Y)=Z^3-\sum_{k=1}^3 f_k(X, Y) Z^{3-k}$ can be linearly transformed over $\bar{F}$ into the one with $f_1=0$,  $f_2=eXY$ and $f_3=X^3 +Y^3$ for some $e\in \bar{F}$ (in characteristic not 2 or 3), and therefore that $C_{\Phi}$ is Azumaya follows immediately from \cite{Kuo} and the fact that the construction of $C_{\Phi}$ is functorial in $F$.
\end{rem}

We are finally able to describe explicitly the simple homomorphic images of $C_\Phi$.

\begin{thm}\label{11corresp}
There is a one-to-one correspondence between the simple homomorphic images of $C_\Phi$ and the Galois orbits of $\bar{F}$-rational points on the affine elliptic curve $E$ as follows: the Galois orbit containing $(R_0,S_0)$ on $E$ gives rise to the $F(R_0,S_0)$-central simple algebra $(\alpha, S_0)_{3,F(R_0,S_0)}$ if $S_0\neq 0$ and $(\rho^2 e R_0-D, \alpha)_{3,F(R_0,S_0)}$ if $S_0=0$.
\end{thm}
\begin{proof}
Since $C_\Phi$ is Azumaya, there is a one-to-one correspondence between its simple homomorphic images and maximal ideals of its center $Z\cong F[E]$. Furthermore, $y_1^3, w$ in the center correspond to $S, R$. Thus by Equation \eqref{y1^3y2^3}, $y_2^3$ corresponds to $\rho^2 e R-S-D$. Therefore, the result follows from Proposition \ref{eitheror} and Corollary \ref{L1}.
\end{proof}

Define the function $\Psi$ from the group $E(F)$ of $F$-rational points on the elliptic curve $E$ into the Brauer group of $F$ as follows
\begin{eqnarray*}
\Psi\colon E(F)
 & \rightarrow Br(F)\\
(R_0,S_0) & \mapsto \begin{cases}
[(\alpha, S_0)_{3,F}] & \textrm{if}\ S_{0}\neq0\\
[(\rho^2 e R_0-D, \alpha)_{3,F}] & \textrm{if}\ S_{0}=0\end{cases}\\
O & \mapsto 1.
\end{eqnarray*}
We next show that the arguments used in \cite[Section 4]{Kuo} can be applied here to show that $\Psi$ is a group homomorphism.

\begin{prop}
The function $\Psi$ is a group homomorphism.
\end{prop}
\begin{proof}
Identify $Z=F[y_1^3, w]$ with $F[E]$. Similar to the proof of \cite[Corollary 4.3]{Kuo}, the Brauer class of $C_\Phi$ in $Br(F(E))$ is unramified everywhere. Thus, the algebra $C_\Phi$ can be extended to a Brauer class in $Br(E)$. Also, $C_\Phi\otimes_{F[E]} F(E)=(\alpha, S)_{3,F(E)}=(\alpha, T)_{3,F(E)}$, where $T = R^3/S^2$. By Equation \eqref{E} we see that
\begin{equation}\label{T}
T = \frac{D-\rho^2 e R}{-\alpha S}-\frac{1}{\alpha}+\frac{D_1^3}{27\alpha^2 S^2}+\frac{D_2^3}{729 \alpha^4 S^2}+\frac{\rho^2 D_1D_2R}{9\alpha^2 S^2}.
\end{equation}
Let $\nu$ be the discrete valuation on $F(E)$ corresponding to $O$. Then $\nu(R)=-1$ and $\nu(S)=-3/2$. Thus  $\nu(T)=0=\nu(\alpha)$, and hence the specialization of $C_\Phi\otimes_{F[E_a]} F(E)$ at $O$ is $(\alpha, \bar{T})_{3,F}$ where $\bar{T}$ is the image of $T$ in the residue field of $O$. By Equation \eqref{T}, $\bar{T}=-1/\alpha= N_{F(\sqrt[3]{\alpha})/F}((-1/\alpha)\sqrt[3]{\alpha^2})$. Thus, the specialization at $O$ of the class of $C_\Phi$ in $Br(E)$ is trivial. Therefore, similar to the proof of \cite[Theorem 4.1]{Kuo}, the result now follows from Lemma 3.2 and Theorem 3.5 of \cite{CK}.
\end{proof}

Since $C_\Phi$ is Azumaya of rank 9, one can check that the homogeneous polynomial $\Phi(X,Y,Z)$ over $F$ is then absolutely irreducible. Let $C$ denote the cubic curve given by the equation $\Phi(X,Y,Z)=0$. The computations in \cite{A} show that the elliptic curve $E$ is the Jacobian of the cubic curve $C$. We have the following two analogues of Proposition 4.5 and Theorem 4.6 of \cite{Kuo} with similar proofs, which we therefore skip.

\begin{prop}\label{F(C)}
The group homomorphism $\Psi\colon E(F)\rightarrow Br(F)$ maps onto the relative Brauer group $Br(F(C)/F)$.
\end{prop}

\begin{prop}
The Azumaya algebra $C_\Phi$ is split if and only if the cubic curve $C$ has an $F$-rational point.
\end{prop}

\subsection{Case 2}

Let $\Phi(z,a,b)=z^3-e a b z-(\alpha a^3+\beta a^2 b+\gamma a b^2+\delta b^3)$ for some $e,\alpha,\delta,\beta,\gamma \in F$, $\charac{F}=3$ and $\alpha \neq 0$.

$C_\Phi$ is by definition
\begin{equation*}
\begin{aligned}
F\langle x,y\colon &  x^3=\alpha,\\
 & y^3=\delta, \\
 & x^2 * y=e x+\beta,\\
& x * y^2=e y+\gamma\rangle.
\end{aligned}
\end{equation*}

We treat the two cases of $e=0$ and $e \neq 0$ separately.

\subsubsection{$e=0$}

In this case, $C_\Phi$ is simply the ordinary Clifford algebra of the form $f(X,Y)=\alpha X^3+\beta X^2 Y+\gamma X Y^2+\delta Y^3.$
The element $x$ is $3$-central. Therefore, according to Lemma \ref{decomcharp2}, we can decompose $y$ as
$$y=y_2-y_1$$
such that
\begin{equation}\label{y012}
x y_2-y_2 x=y_1, x y_1-y_1 x=y_0, \textrm{ where }y_0 x=x y_0.
\end{equation}
Substituting this in the relation $x^2 * y=\beta$, by a straight-forward calculation we get $y_0=\beta.$
Thus from the relation $x * y^2=\gamma$, we then get
\begin{equation}\label{y12}
y_1 y_2-y_2 y_1=\gamma.
\end{equation}
Substituting this further in $y^3=\delta$ leaves
\begin{equation}\label{y21^3}
y_2^3-y_1^3=\delta.
\end{equation}
Therefore, $C_\Phi$ is an $F$-algebra generated by $x, y_1, y_2$ subject to the relations $x^3=\alpha$, Equation \eqref{y012},  where $y_0=\beta$, and Equations \eqref{y12}, \eqref{y21^3}. We shall see in particular that, unless $f$ is diagonal, $C_\Phi$ is Azumaya.

Let $w=\beta y_2+\gamma x+y_1^2$. It is a straight-forward calculation to see that $w$, $y_1^3$ and $y_2^3$ commute with $x$, $y_1$ and $y_2$, and therefore they are central in $C_\Phi$. Consider the following affine curve
$$E_\Delta : s^2=r^3+\Delta,$$
where $\Delta=-\gamma^3 \alpha+\gamma^2 \beta^2-\beta^3\delta+ \beta^6.$
We next show that in the case of $\beta \neq 0$, $C_\Phi$ is Azumaya of rank 9 and its center is isomorphic to the coordinate ring of $E_\Delta$.

\begin{lem}
If $\beta \neq 0$ then the subalgebra $F[w,y_1^3]$ of the center of $C_\Phi$ is isomorphic to the coordinate ring $F[r,s]$ of the affine curve $E_\Delta$. In particular it is an integral domain.
\end{lem}

\begin{proof}
If $\beta \neq 0$, then $y_2=\beta^{-1}(w-\gamma x-y_1^2)$ and substituting it in Equation \eqref{y21^3} yields $\beta^{-3} (w^3-\gamma^3 \alpha-y_1^6+\gamma^2 \beta^2)-y_1^3=\delta,$ or equivalently,
\begin{equation*}
w^3+\Delta=(y_1^3 -\beta^3)^2.
\end{equation*}
Consequently $F[w,y_1^3]$ is the $F$-subalgebra generated by $w$ and $y_1^3$ subject only to the relation in the equation above. Thus the map defined by  sending $r, s$ to $w, y_1^3-\beta^3$ clearly gives an $F$-algebra isomorphism.
\end{proof}

Note that $E_\Delta$ is smooth (and then an affine elliptic curve) if and only if its discriminant is nonzero or $\Delta\neq 0$.
In this case, its coordinate ring is a Dedekind domain. In the following, for any integral domain $R$, $q(R)$ stands for its quotient field.

\begin{thm}\label{betaneq}
If $\beta \neq 0$ then
\begin{enumerate}
\item[(1)] $C_\Phi$ is Azumaya of rank 9.
\item[(2)] The center of $C_\Phi$ is the subalgebra $F[w,y_1^3]$, and it is isomorphic to the coordinate ring of $E_\Delta$.
\item[(3)] There is a one-to-one correspondence between the Galois orbits of $\bar{F}$-rational points on $E_\Delta$ and the simple homomorphic images of $C_\Phi$, taking each Galois orbit containing $(r_0,s_0)$ to the degree 3 cyclic algebra $[\alpha\beta^{-3} (s_0+\beta^3),\alpha)_{3,F[r_0,s_0]}$.
\end{enumerate}
\end{thm}

\begin{proof}
In this case, $y_2=\beta^{-1}(w-\gamma x-y_1^2)$. Let $z=\beta^{-1} x y_1$. It is a straight-forward calculation to see that $x z-z x=x$ and $z^3-z=\alpha\beta^{-3} y_1^3$. Consequently, in $C_\Phi \otimes_{F[w,y_1^3]} q(F[w,y_1^3])$, $x$ and $z$ generate over $q(F[w,y_1^3])$ a cyclic algebra of degree $3$ in which $x$ is $3$-central and $z$ is Artin-Schreier. The subalgebra $q(F[w,y_1^3])[x,z]$ in fact contains all the generators of $C_\Phi$, and therefore $q(F[w,y_1^3])[x,z]=C_\Phi \otimes q(F[w,y_1^3])$. In particular, the center of  $C_\Phi \otimes q(F[w,y_1^3])$ is $q(F[w,y_1^3])$, and hence the center of $C_\Phi$  is $F[w,y_1^3]$, which is isomorphic to the coordinate ring $F[r,s]$ of the affine curve $E_\Delta$ by the Lemma above. Identifying $F[w,y_1^3]$ with $F[r,s]$, we have $r=w$ and $s=y_1^3-\beta^3$.

Let there be a simple homomorphic image $A$ of $C_\Phi$. Let $r_0, s_0, x'$ and $y_1'$ be the images in $A$ of $r, s, x$ and $y_1$, respectively. In particular $x'^3=\alpha$ and $y_1'^3=s_0+\beta^3$. Furthermore, $A$ is generated by $x'$ and $z'=\beta^{-1} x' y_1'$ over $F[r_0,s_0]$, where $z'^3-z'=\alpha\beta^{-3} (s_0+\beta^3)$. These two elements satisfy $x' z'-z' x'=x'$, and therefore $A$ is a cyclic algebra of degree 3 over $F[r_0,s_0]$ in which $x'$ is $3$-central and $z'$ is Artin-Schreier. Hence $A$ has the symbol presentation  $[z'^3-z',x'^3)_{3,F[r_0,s_0]}=[\alpha\beta^{-3} (s_0+\beta^3),\alpha)_{3,F[r_0,s_0]}$. In particular, this implies that $C_\Phi$ is Azumaya of rank 9. Consequently, the simple homomorphic images of $C_\Phi$ are determined by the maps taking $F[r,s]$ to $F[r_0,s_0]$ for $\bar{F}$-rational points $(r_0,s_0)$ on the curve $E_\Delta$, whose formula is given as above, and this provides a one-to-one correspondence between the Galois orbits of the $\bar{F}$-rational points on $E_\Delta$ and the simple homomorphic images of $C_\Phi$.
\end{proof}

In case $\beta=0$, $\gamma \neq 0$ and furthermore $\delta\neq 0$, we can simply switch the roles of $x$ and $y$ and get a similar result to Theorem \ref{betaneq}. What remains is the case of $\beta=\gamma=0$.

\begin{thm}
If $\beta=\gamma=0$ then
\begin{enumerate}
\item[(1)] The center of $C_\Phi$ is the polynomial ring $F[y_1]$.
\item[(2)] The algebra $C_\Phi[y_1^{-1}]$ is Azumaya of rank 9 with the Laurent polynomial ring $F[y_1,y_1^{-1}]$ as its center.
\item[(3)] There is a one-to-one correspondence between the Galois orbits of $\bar{F}^\times$ and the simple homomorphic images of $C_\Phi[y_1^{-1}]$, taking each Galois orbit containing $s_0 \in \bar{F}^\times$ to $[\alpha (s_0^3+\delta) s_0^{-3},\alpha)_{3,F[s_0]}$.
\item[(4)] The algebra $C_\Phi$ is not Azumaya.
\end{enumerate}
\end{thm}

\begin{proof}
In this case, the algebra $C_\Phi$ is an $F$-algebra generated by $x,y_1, y_2$ subject to the relations $x^3=\alpha$, $[x, y_1]=[y_2, y_1]=0$, $x y_2-y_2 x=y_1$ and $y_2^3-y_1^3=\delta$. Therefore $y_1$ is central in $C_\Phi$ and it generates over $F$ a free algebra in one indeterminate.

The algebra $C_\Phi \otimes_{F[y_1]} q(F[y_1])$ contains the elements $z=x y_2 y_1^{-1}$ and $x$. By a straight-forward calculation we see that $x z-z x=x$ and $z^3-z=\alpha y_2^3 y_1^{-3}$. Since $y_2^3-y_1^3=\delta$, we obtain $z^3-z=\alpha (\delta+y_1^3) y_1^{-3} \in q(F[y_1])$. Thus the $q(F[y_1])$-subalgebra of $C_\Phi \otimes q(F[y_1])$ generated by $x,z$ is cyclic of degree 3, and since it contains all the generators of $C_\Phi$, we see that $q(F[y_1])[x,z]=C_\Phi \otimes q(F[y_1])$. Therefore the center of $C_\Phi \otimes q(F[y_1])$ is $q(F[y_1])$, and hence the center of $C_\Phi$ is $F[y_1]$.

Let $A$ be a simple homomorphic image of $C_\Phi[y_1^{-1}]$. The image of $y_1$ in $A$ is some element $s_0 \in \bar{F}^\times$. Let $x'$ and $y_2'$ be the images of $x$ and $y_2$ in $A$. Now, $x'$ and $z'=x' y_2' s_0^{-1}$ generate a cyclic $F[s_0]$-subalgebra of degree 3, and since they also generate $A$ over $F[s_0]$, we conclude that $A$ is a cyclic algebra over $F[s_0]$ of degree $3$ with the symbol presentation $[z'^3-z',x'^3)_{3,F[s_0]}=[\alpha (s_0^3+\delta) s_0^{-3},\alpha)_{3,F[s_0]}$. Therefore $C_\Phi[y_1^{-1}]$ is Azumaya of rank 9 and the statement (3) follows.

The algebra $C_\Phi$ is not Azumaya, however, because there is one image that is obtained by sending $y_1$ to $0$, namely the commutative $F$-algebra generated by the images $\bar{x}, \bar{y_2}$ of $x$, $y_2$, satisfying $\bar{x}^3=\alpha, \bar{y_2}^3=\delta$.
\end{proof}

\subsubsection{$e \neq 0$}

By changing the variable $X$ with $X'=e X$, we could assume that $e=1$ in the first place.
Now, by choosing the new pair of variables $X'=X+Y$ and $Y'=X-Y$, we then may assume the polynomial $\Phi$ is of the form $$\Phi(Z,X,Y)=Z^3-(X^2-Y^2) Z-(\alpha X^3+\beta X^2 Y+\gamma X Y^2+\delta Y^3).$$
The algebra $C_\Phi$ thus in this case is
$$F\langle x,y\colon  x^3-x=\alpha, y^3+y=\delta, x^2 * y-y=\beta, x * y^2+x=\gamma\rangle.$$

The element $x$ is Artin-Schreier.
According to Lemma \ref{decomcharp}, $y=y_0+y_1+y_2$ such that $y_k x-x y_k=k y_k$ for $k=0,1,2$.
Substituting that in $x^2 * y-y=\beta$ leaves $y_0=-\beta$.
From $x * y^2+x=\gamma$ we then obtain $y_1 y_2-y_2 y_1+x=\gamma$.
Furthermore, a straight-forward calculation shows that $y^3+y=\delta$ becomes
\begin{equation}\label{y12^3}
y_1^3+y_2^3=\delta+\beta^3+\beta.
\end{equation}
One can check that $w=y_2 y_1-x^2+(1-\gamma) x$, $y_1^3$ and $y_2^3$ are central in $C_\Phi$.

\begin{lem}
The subalgebra $F[w,y_1^3]$ of the center of $C_\Phi$ is isomorphic to the coordinate ring of the affine curve $$E : s^2=r^3+r^2-(\gamma^2+\gamma) r-\alpha^2-\alpha \gamma^3+\alpha \gamma+(\delta+\beta^3+\beta)^2.$$ In particular it is an integral domain.
\end{lem}

\begin{proof}
A straight-forward calculation shows that
\begin{equation*}
\begin{aligned}
w^3 & =y_2^3 y_1^3+w^2+(\gamma^2+\gamma) w-\alpha^2-\alpha \gamma^3+\alpha \gamma\\
 & =(\delta+\beta^3+\beta) y_1^3-y_1^6+w^2+(\gamma^2+\gamma) w-\alpha^2-\alpha \gamma^3+\alpha \gamma.
\end{aligned}
\end{equation*}
Thus $F[w, y_1^3]$ is the algebra over $F$ generated by $w$ and $y_1^3$ subject only to the relation in the equation above. The map defined by sending $r, s$ to $-w, y_1^3+(\delta+\beta^3+\beta)$ then gives an $F$-algebra isomorphism.
\end{proof}


\begin{thm}
Assuming that $\delta+\beta^3+\beta \neq 0$,
\begin{enumerate}
\item[(1)] The algebra $C_\Phi$ is Azumaya of rank 9.
\item[(2)] The center of $C_\Phi$ is the subalgebra $F[w, y_1^3]$, which is isomorphic to the coordinate ring of $E$.
\item[(3)] There is a one-to-one correspondence between the Galois orbits of $\bar{F}$-rational points on $E$ and the simple homomorphic images of $C_\Phi$, taking each Galois orbit containing point $(r_0,s_0)$ to the algebra $[\alpha,s_0-(\delta+\beta^3+\beta))_{3,F[r_0,s_0]}$ if $s_0 \neq \delta+\beta^3+\beta$, and to $[-\alpha,\delta+\beta^3+\beta)_{3,F[r_0]}$ if $s_0=\delta+\beta^3+\beta$.
\end{enumerate}
\end{thm}

\begin{proof}
The element $y_1$ is invertible in $C_\Phi \otimes_{F[w,y_1^3]} q(F[w,y_1^3])$, and so inside this algebra $y_2=(w+x^2-(1-\gamma)x) y_1^{-1}$. Thus $C_\Phi \otimes q(F[w,y_1^3])$ is generated over $q(F[w,y_1^3])$ by $x$ and $y_1$.
Since $x$ is Artin-Schreier, $y_1^3$ is central and $y_1 x-x y_1=y_1$, the algebra $C_\Phi \otimes q(F[w,y_1^3])$ is the cyclic algebra $[\alpha,y_1^3)_{3,q(F[w,y_1^3])}$.
Thus the center of $C_\Phi$ being the intersection of $C_\Phi$ and the center of $C_\Phi \otimes q(F[w,y_1^3])$ is $F[w,y_1^3]$.

Every homomorphism from $C_\Phi$ to a simple algebra $A$ takes $F[w,y_1^3]$ to a field $F[r_0,s_0]$ where $(r_0,s_0)$ is an $\bar{F}$-rational point on the affine curve $E$ and $y_1^3$ is sent to $s_0-(\delta+\beta^3+\beta)$ by the lemma above. If $s_0\neq\delta+\beta^3+\beta$ then $A$ is generated by the images $x', y_1'$ of $x, y_1$ such that $A$ is the cyclic algebra $[x'^3-x',y_1'^3)_{3,F[r_0,s_0]}=[\alpha,s_0-(\delta+\beta^3+\beta))_{3,F[r_0,s_0]}$.
If $s_0=\delta+\beta^3+\beta$ then $y_1^3$ is sent to $0$ and hence $y_2$ is sent to the invertible element $s_0=\delta+\beta^3+\beta$ by Equation \eqref{y12^3}. This means that $A$ is generated by the images $x', y_2'$ of $-x$, $y_2$, satisfying $y_2' x'-x' y_2'=y_2'$. Thus $A$ is the cyclic algebra $[x'^3-x',y_2'^3)_{3,F[r_0]}=[-\alpha,\delta+\beta^3+\beta)_{3,F[r_0]}$. In particular, it implies that $C_\Phi$ is Azumaya of rank 9 and the statement (3) follows.
\end{proof}

\begin{rem}
If $\delta+\beta^3+\beta=0$ then for similar arguments as in the last proof, $C_\Phi[y_1^{-3}]$ is Azumaya of rank 9, and there is a one-to-one correspondence between its simple homomorphic images and the Galois orbits of the $\bar{F}$-rational points $(r_0,s_0)$ on $E$ with $s_0 \neq 0$, taking each such Galois orbit to the algebra $[\alpha,s_0)_{3,F[r_0,s_0]}$.

In this case, the algebra $C_\Phi$ is not necessarily Azumaya, for instance if furthermore  $\gamma^3-\gamma-\alpha=0$ then $F$ is a simple homomorphic image of $C_\Phi$, and then $C_\Phi$ is definitely not Azumaya.
\end{rem}

\section{The Clifford Algebra of a Degree $d$ Projective Variety}\label{sgca}

Let $d$ be an integer, $F$ a field and $A$ a central simple $F$-algebra.

In this section we present a further generalization of the Clifford algebra of a monic polynomial, namely the Clifford algebra of a degree $d$ projective variety.

Assume that $V$ is a projective subvariety of $A$, i.e there exists a set of equations $S$, such that $V=\set{u_1 v_1+\dots+u_n v_n : (u_1,\dots,u_n) \in Z(S)}$ for some linearly independent $v_1,\dots,v_n \in A$. We call $V$ a degree $d$ variety if there exist forms $f_1,\dots,f_d$ such that for each $1 \leq i \leq p$, $f_i$ is a form of degree $i$ with $n$ variables, and $$(a_1 v_1+\dots+a_n v_n)^d=\sum_{k=1}^d f_k (a_1,\dots,a_n) (a_1 v_1+\dots+a_n v_n)^{d-k}$$ for all $(a_1,\dots,a_n) \in Z(S)$.

The variety $V$ is called $p$-central if $f_1=\dots=f_{p-1}=0$.

We define the Clifford algebra of $V$ to be
\begin{eqnarray*}
C(V)=F[x_1,\dots,x_n : (a_1 x_1+\dots+a_n x_n)^d=\\ \sum_{k=1}^d f_k (a_1,\dots,a_n) (a_1 v_1+\dots+a_n v_n)^{d-k}\\ \forall (a_1,\dots,a_n) \in Z(S)]
\end{eqnarray*}
As before, there is a natural epimorphism $C(V) \rightarrow F[V]$ taking $x_i$ to $v_i$ for each $i$.

Since the algebra only depends on the choice of $S$ and $f_1,\dots,f_d$, one can address the algebra as $C_{S,f_1,\dots,f_p}$.

If $f_1=\dots=f_{d-1}=0$ then the original variety $V$ is $d$-central in $A$. In this case we may write $C_{S,f}$ to denote the Clifford algebra.

If $S=\emptyset$ then $C_{S,f}$ is simply the standard Clifford algebra $C_f$.

In this section we shall focus on the case of $d=2$ and $f_1=0$.

\begin{rem}
It is easy to show that if $\charac{F} \neq 2$ and $V$ is a degree $2$ variety then $\set{v-\frac{\tr(v)}{2} : v \in V}$ is a 2-central variety. Therefore, in case of $d=2$, we can assume that $f_1=0$ from the beginning.
\end{rem}

Let us have a look for example at $V=\set{u_1 v_1+u_2 v_2+u_3 v_3 : (u_1,u_2,u_3) \in Z(S)}$ with $S=\set{u_1 u_3-u_2^2}$ and $(u_1 v_1+u_2 v_2+u_3 v_3)^2=f(u_1,u_2,u_3)$ for some ternary quadratic form $f$.

Because of the relation $u_1 u_3-u_2^2 \in S$,
\begin{eqnarray*}
(u_1 x_1+u_2 x_2+u_3 x_3)^2=u_1^2 x_1^2+u_2^2 (x_2^2+x_1 * x_3)+\\u_3^2 x_3^2+u_1 u_2 x_1 * x_2+u_2 u_3 x_2 * x_3.
\end{eqnarray*}
Consequently, $C_{S,f}$ is the algebra generated over $F$ by $x_1,x_2,x_3$ subject to the following relations:
\begin{enumerate}
\item $x_1^2=\alpha_{1,1}$
\item $x_3^2=\alpha_{3,3}$
\item $x_1 * x_3+x_2^2=\alpha_{2,2}+\alpha_{1,3}$
\item $x_1 * x_2=\alpha_{1,2}$
\item $x_2 * x_3=\alpha_{2,3}$
\end{enumerate}
where for each $1 \leq j,k \leq 3$, $\alpha_{j,k}$ is the coefficient of $u_j u_k$ in $f$.

Therefore, $C_{S,f}$ is exactly the algebra associated to the quartic form $g(x)={\alpha_{3,3}} x^4+\alpha_{2,3} x^3+(\alpha_{2,2}+\alpha_{1,3}) x^2+\alpha_{1,2} x+\alpha_{1,1}$, as defined in \cite{Haile4}.

In that paper, Haile and Han proved that all the simple images of $C_{S,f}$ are quaternion algebras. We will now generalize this result to any 2-central variety with simultaneously diagonalizable defining equations.

\begin{thm}
Let $V$ be a 2-central variety with defining equations $S$ and exponentiation form $f$.
If $S$ are simultaneously diagonalizable then the images of $C_{S,f}$ are all tensor products of up to $\lfloor \frac{n}{2} \rfloor$ quaternion algebras.
\end{thm}

\begin{proof}
We can assume that $S$ is a system of diagonal equations, because diagonalizing the system corresponds to a linear change of the generators of the Clifford algebra, and it does not change the algebra.
For any $k \neq m$, $x_m$ and $x_k$ satisfy a relation $x_m x_k+x_k x_m=\alpha_{k,m}$.
Consequently, $x_k^2 x_m=x_k (-x_m x_k+\alpha_{k,m})=\alpha_{k,m} x_k-x_k x_m x_k=\alpha_{k,m}-(-x_m x_k+\alpha_{k,m}) x_k=x_m x_k^2$.
Therefore, $x_k^2$ commutes with every $x_m$.
Hence $x_k^2$ is in the center of $C_{S,f}$.
Consequently, in every simple image of the algebra, $F x_1+\dots+F x_n$ is a $2$-central space, and so the image is a tensor product of up to $\lfloor \frac{n}{2} \rfloor$ quaternion algebras.
\end{proof}

\begin{cor}
In case of $\charac{F} \neq 2$, if $S$ contains one equation then it is diagonalizable and therefore all the images of $C_{S,f}$ are tensor products of quaternion algebras. In particular, the algebra studied by Haile and Han in \cite{Haile4} satisfies this property.
\end{cor}

\chapter{$d$-Central Spaces in Tensor Products of Cyclic Algebras}

\section{Background}
Let $d$ be an integer, and $F$ be an infinite field of characteristic prime to $d$ containing a primitive $d$th root of unity $\rho$.

Let $A$ be a tensor product of $n$ cyclic algebras of degree $d$, $(\alpha_1,\beta_1)_{d,F} \otimes \dots \otimes (\alpha_n,\beta_n)_{d,F}=F[x_1,y_1] \otimes \dots \otimes F[x_n,y_n]$.

Let $V_0=F$ and $V_k=F[x_k] y_k+V_{k-1} x_k$ for any $1 < k \leq n$.
Assume that $v^d \in F$ for all $v \in V_{k-1}$ for a certain $k$.
Every element of $V_k$ is of the form $f(x_k) y_k+v x_k$ for some $f(x_1) \in F[x_1]$ and $v \in V_{k-1}$. Since $v$ commutes with $x_k$ and $y_k$, and $y_k x_k=\rho x_k y_k$, $(f(x_k) y_k+v x_k^{a_k})^d=(f(x_k) y_k)^d+v^d x_k^{d a_k}=\norm_{F[x_k]/F}(f(x_k)) y_k^d+v^d x_k^{d a_k} \in F$ (see \cite[Remark 2.5]{Chapman}).
For any $1 \leq m \leq d-1$, if $f(x_k) \neq 0$ then the eigenvector of $(f(x_k) y_k+v x_k)^m$ corresponding to the eigenvalue $\rho^m$ with respect to conjugation by $x_k$ is $(f(x_k) y_k)^m$, which is not zero, and therefore $(f(x_k) y_k+v x_k)^m \not \in F$.
If $f(x_k)=0$ then what is left is $v x_k$, and of course $v^m x_k^m \not \in F$.
Consequently, $V_k$ is $d$-central.
Since $V_0=F$, by induction $V_k$ is $d$-central for every $1 \leq k \leq n$. The dimension of each $V_k$ is $d k+1$.

For $d=2$, it follows from the theory of Clifford algebras that for any $k \leq n$, $V_k$ is maximal with respect to inclusion.
Furthermore, it is known that every maximal space is of some odd dimension $2 k+1$ and can be obtained in the same way as $V_k$ by some decomposition of the algebra as a tensor product of quaternion algebras.

In \cite{Matzri}, Matzri and Vishne noted that for $d=3$ and $n=1$, every maximal $3$-central space is a subspace of $V_1$.

In Section \ref{maxSec} we prove that each $V_K$ is maximal with respect to inclusion when $d=p$ for some prime $p$.
In Section \ref{maxnonprimeSec} we focus on one cyclic algebra of degree $p^k$ where $p$ is prime and show that it contains a family of $p^k$-subspaces, of which $V_1$ is a special case.

A finite set $B=\set{b_1,\dots,b_n} \subseteq A$ consisting of $F$-linearly independent invertible elements is called a \textbf{$p$-central set} if
\begin{enumerate}
\item For any $1 \leq k \leq n$, $b_k^p=\alpha_k \in F$.
\item For any $1 \leq m < k \leq n$, $b_m b_k=\rho^{c_{m,k}} b_k b_m$ for some $c_{m,k} \in \mathbb{Z}$.
\end{enumerate}

This term was introduced by Rowen in \cite[Vol II, pp. 248-251]{Rowen}. A $p$-central pair is a $p$-central set of cardinality $2$.

It is known that every nondegenerate quadratic space is spanned by a $2$-central set. Therefore it generates a tensor product of quaternion algebras.

In \cite{Raczek}, Raczek proved that every $3$-dimensional $3$-central subspace of a cyclic algebra of degree $3$ is of the form $F \mu+F \nu+F(\lambda_1 \mu \nu^2+\lambda_2 \mu^2 \nu^2)$ where $\mu$ and $\nu$ form a $2$-central pair.

In Section \ref{5dim4degSec} we prove that $5$ is the maximal dimension of a $4$-central subspace of a cyclic algebra of degree $4$ containing a $4$-central pair.
In Section \ref{pset3Sec} we classify $p$-central subspaces of cyclic algebras of degree $p$ containing $p$-central sets of the form $\set{x,y,x y}$ and $5$-central spaces containing any $5$-central set of size $3$.

Section \ref{threesec} is dedicated to $3$-central spaces spanned by $3$-central sets.

In Section \ref{influenceSec} we study the effect of the existence of $3$-central spaces in algebras of fixed degrees. We focus on degree 3.
We show that for a field extension $K/F$, if a central simple $K$-algebra $A$ of degree 3 contains an $F$-vector subspace $V$ such that $v^3 \in F$ for all $v \in F$ and $[V:F]=3$ then $A$ is a restriction of a central simple $F$-algebra. We provide a counterexample in case $[V:F]=2$.

\section{Maximal $p$-Central Spaces}\label{maxSec}

Here we shall assume that $d=p$ for some prime $p$.
The following result appeared in my Master's thesis \cite{Chapmanthesis} and is repeated here in a refined manner for completeness.

\begin{thm}
For any $k \leq n$, $V_k$ is maximal with respect to inclusion.
\end{thm}

\begin{proof}
Let $V=V_k$. $V$ has a standard basis $$B=\set{x_i^j y_i x_{i+1} \dots x_k : 1 \leq i \leq k, 0 \leq j \leq p-1} \cup \set{x_1 x_2 \dots x_k}.$$
Let $z$ be a nonzero element in the algebra $A$.
This element can be expressed as a linear combination of the monomials $x_1^{c_1} y_1^{e_1}\dots x_n^{c_n} y_n^{e_n}$.
Let us assume negatively that $V+F z$ is $p$-central.
Consequently, $w^{p-1} * z \in F$ for every $w \in B$.
Since we can subtract from $z$ the appropriate linear combination of the elements of $B$, we can assume that $w^{p-1} * z=0$ for every $w \in B$.

Let us pick one monomial $t=x_1^{c_1} y_1^{e_1}\dots x_n^{c_n} y_n^{e_n}$.

If $e_1=e_2=\dots=e_n=0$ then $t$ commutes with $x_1 x_2 \dots x_k \in V$. Since $(x_1 x_2 \dots x_k)^{p-1} * z=0$, the coefficient of $t$ in $z$ is zero.

Otherwise, let $i$ be the maximal integer for which $e_i \neq 0$.
The monomial $t$ commutes with the element $x_i^r y_i x_{i+1} \dots x_k \in V$ where $r \equiv c_i e_i^{-1} \pmod{p}$.
Since $(x_i^r y_i x_{i+1} \dots x_k)^{p-1} * z=0$, the coefficient of $t$ in $z$ is zero.

Therefore, the coefficient of $t$ in $z$ is always zero, which means that $z=0$, and that is a contradiction.
\end{proof}

\section{Family of $p^k$-Central Spaces}\label{maxnonprimeSec}

Assume $d=p^k$ for some prime $p$. In a cyclic algebra $A=F[x,y]=(\alpha,\beta)_{d,F}$ there is a $d$-central space of the form $V_1$ as defined above.
Apparently this space belongs to a larger family of $d$-central subspaces of this algebra.

Let $V=F[x^{p^e}] y+F[y^{p^{k-e}}] x$ for some $0 \leq e \leq k-1$.

\begin{prop}
The space $V$ is $p^k$-central.
\end{prop}

\begin{proof}
Every element in $V$ is of the form $f(x^{p^e}) y+g(y^{p^{k-e}}) x$ where $f$ and $g$ are polynomials.
Its $p^k$th power is $(\norm_{F[x^{p^e}]/F}(f(x^{p^e})))^{p^{k-e}} \beta+(\norm_{F[y^{p^{k-e}}]/F}(g(y^{p^{k-e}})))^{p^e} \alpha \in F$.
It is clear that the lower powers are not in the center.
\end{proof}

\begin{conj}
The $p^k$-central space $V$ is maximal with respect to inclusion.
\end{conj}

\textbf{Idea:}
Assume to the contrary, that there exists $z \in A \setminus V$ such that $V+F z$ is $p^k$-central.
Let $0 \leq i,n \leq p^e-1$, $0 \leq j,m \leq p^{k-e}-1$. We consider the coefficient of the monomial $x^{i+j p^e} y^{m+n p^{k-e}}$ in $z$.
If $i=j=0$ then the monomial commutes with $y$ and therefore its coefficient in $z$ is zero.
Similarly if $m=n=0$ then the coefficient is zero.

If $i=m=0$ and $n,j \neq 0$ then the relation
$$\tr(z * (x^{(p^{k-e}-j) p^e} y) * y^{p^k-n p^{k-e}-1})=0$$ holds, because $1 \leq 1+1+p^k-n p^{k-e}-1 \leq p^k-1$.

If $i=0$, $m=1$ and $n=0$ then $x^{i+j p^e} y^{m+n p^{k-e}} \in V$ which means that we can assume its coefficient in $z$ is zero.

If $i=0$ and $m \geq 2$ then the relation
$$\tr(z * (x^{(p^{k-e}-j) p^e} y) * y^{p^k-m-1-n p^{k-e}})=0$$ holds, because $1 \leq 1+1+p^k-m-1-n p^{k-e} \leq p^k-1$.

Similar relations hold if $m=0$, $i=1$ and $j=0$ or $m=0$ and $i \geq 2$.
This covers all the options with either $i=0$ or $m=0$,

Let us assume that $i,m \neq 0$.
Therefore the relation $$\tr(z * (x^{(p^{k-e}-j-1) p^e} y) * y^{p^{k-e}-m-1} * (y^{(p^e-n-1) p^{k-e}} x) * x^{p^e-i-1})=0$$ holds, because $1+1+p^{k-e}-m-1+1+p^e-i-1 \leq p^e+p^{k-e}-1 \leq p^k -1$.

If we manage to prove that all the relations above are nontrivial then it will mean that $z=0$.

\begin{rem}
In case of $p=k=2$ all the relations above turn out to be nontrivial, and the $4$-central spaces are indeed maximal.
\end{rem}

\section{$5$-Dimensional $4$-Central Spaces}\label{5dim4degSec}
Let $A$ be a central division algebra of degree $4$ over a field $F$ containing a primitive fourth root of unity $i$ and of characteristic not $2$.

The aim of this section is to prove the following theorem:
\begin{thm}
The upper bound for the dimension of $4$-central spaces containing pairs of standard generators is $5$.
\end{thm}

The rest of this section will deal with proving this theorem.
Assume to the contrary, that there exists a $6$-dimensional $4$-central space $W$ containing a pair $x$ and $y$ satisfying $y x=i x y$.
For any element $q \in W$ we write $q=\sum_{m=0}^3 \sum_{n=0}^3 q_{m,n} x^m y^n$.
From $\tr(q)=0$ we get $q_{0,0}=0$.
Because $\tr(x * y * q)=0$ we always get $q_{3,3}=0$.
There exists a subspace $V$ of dimension $5$ such that for every $q \in V$, $q_{2,2}=0$.
Since $\tr(x^k*q)=\tr(y^k*q)=0$ for $k=1,2$, and we can always subtract $q_{1,0} x+q_{0,1} y$ from $q$, we have $V=Fx+Fy+V'$ such that $V'=\set{q \in V : q_{k,0}=q_{0,k}=0 \forall k}$.

\begin{prop}
The projection of $V$ on $F x y+F x y^2+F x^2 y$ is of dimension no greater than $2$.
\end{prop}

\begin{proof}
Assume to the contrary, that it is of dimension $3$.
Then there exist $z,w,t \in V$ such that $z_{1,1},w_{1,2},t_{2,1} \neq 0$ while $z_{1,2}=z_{2,1}=w_{1,1}=w_{2,1}=t_{1,1}=t_{1,2}=0$.
From $\tr(w^2*y)=0$ we get $w_{3,1}=0$. From $\tr(w^2)=0$ we get $w_{3,2}=0$.
From $\tr(w^3)=0$ we get $w_{1,3}=0$ or $w_{2,3}=0$.
Similarly, $t_{1,3}=t_{2,3}=0$ and either $t_{3,1}=0$ or $t_{3,2}=0$.
From $\tr(w * t * x)=0$ we get $w_{1,3}=0$ and from $\tr(w * t * y)=0$ we get $t_{3,1}=0$.
From $\tr(z * t * x)=0$ we get $z_{1,3}=0$. From $\tr(z * w * y)=0$ we get $z_{3,1}=0$.
From $\tr(z * w)=0$ we get $z_{3,2}=0$ and from $\tr(z * t)=0$ we get $z_{2,3}=0$.
Consequently, $z=z_{1,1} x y$.
But then from $\tr(z * w * t)=0$ we get $z_{1,1} w_{1,2} t_{2,1}=0$, a contradiction.
\end{proof}

\begin{cor}\label{zform}
$V'$ contains an element $z$ of the form $z=z_{1,3} x y^3+z_{2,3} x^2 y^3+z_{3,2} x^3 y^2$ or $z=z_{3,1} x^3 y+z_{2,3} x^2 y^3+z_{3,2} x^3 y^2$.
\end{cor}

\begin{proof}
The projection of $V'$ on $F x y+F x y^2+F x^2 y$ is of dimension no greater than $2$. Therefore there exists a nonzero element $z \in V'$ in the kernel of this projection, i.e. $z_{1,1}=z_{1,2}=z_{2,1}=0$. Since $\tr(z^k)=\tr((z_{1,3} x y^3+z_{3,1} x^3 y)^k)$ for $k=2,3$, $z_{1,3} x y^3+z_{3,1} x^3 y$ must also be $4$-central. Therefore, since $z_{1,3} x y^3+z_{3,1} x^3 y$ is in the cyclic field extension $F[x^3 y]/F$ of degree $4$, $z_{3,1}=0$ or $z_{1,3}=0$.
\end{proof}

\begin{thm}
$V$ is of the form $F[\mu] \nu+F \mu$ where $\mu \nu=i^k \nu \mu$ for $k=\pm 1$.
\end{thm}

\begin{proof}
Without loss of generality, $V'$ contains some nonzero $z$ of the form $z=z_{3,1} x^3 y+z_{2,3} x^2 y^3+z_{3,2} x^3 y^2$.

Let us assume that $z_{3,1},z_{2,3} \neq 0$.

Let $q$ be an arbitrary element in $V'$. We can assume that $q_{3,1}=0$.

From $\tr(z * q * x)=0$ we get $q_{1,1}=0$.
From $\tr(z * q * y)=0$ we get $q_{1,2}=0$.

Let us assume negatively that $q_{2,1} \neq 0$.
From $\tr(q^2 * x)=0$ we get $q_{1,3}=0$.
From $\tr(q^2)=0$ we get $q_{2,3}=0$.
But then $\tr(q * z)=0$ yields $q_{2,1} z_{2,3}=0$, a contradiction.
Consequently, $q_{2,1}=0$.
Like in Corollary \ref{zform}, $q$ is of the form $q=q_{1,3} x y^3+q_{2,3} x^2 y^3+q_{3,2} x^3 y^2$ or $q=q_{3,1} x^3 y+q_{2,3} x^2 y^3+q_{3,2} x^3 y^2$. Since the same holds also for $q+z$, $q$ must be of the form $q=q_{3,1} x^3 y+q_{2,3} x^2 y^3+q_{3,2} x^3 y^2$. Hence, $V'=F x^3 y+F x^2 y^3+F x^3 y^2$.

Let us assume that $z_{3,1}=0$ and $z_{3,2},z_{2,3} \neq 0$.
We have $V'=V''+F z$.
Let us assume negatively that the projection of $V''$ on $F x^2 y+F x y^2$ is of dimension two.
Let $q$ be an arbitrary element of $V''$. From $\tr(z * q)=0$ we get $q_{2,1} z_{2,3}+q_{1,2} z_{3,2}=0$.
But that is a contradiction.
Therefore the projection of $V''$ on $F x^2 y+F x y^2$ is of dimension no greater than $1$.
Consequently, $V''$ contains an element $q$ where $q_{1,2}=q_{2,1}=0$.
From $\tr(q * z * x)=0$ we get $q_{1,1}=0$.
Then by similar arguments to Corollary \ref{zform} and the previous paragraph, $q$ must be of the form $q=q_{3,1} x^3 y+q_{2,3} x^2 y^3+q_{3,2} x^3 y^2$. Hence, $V'=F x^3 y+F x^2 y^3+F x^3 y^2$. If $q_{3,1} \neq 0$ or $q_{1,3} \neq 0$ then we are done.
Otherwise, $F q+F z=F x^3 y^2+F x^2 y^3$.
We shall later solve this case separately.

Let us assume that $z_{3,1},z_{3,2} \neq 0$ and $z_{2,3}=0$.
We have $V'=V''+F z$.
Let $q$ be an arbitrary element of $V''$.
From $\tr(z * q * y)=0$ we get $i q_{1,1} z_{3,2}+q_{1,2} z_{3,1}=0$.
From $\tr(z * q)=0$ we get $i q_{1,2} z_{3,2}+q_{1,3} z_{3,1}=0$.
Since $V''$ is of dimension $2$, and its projection on $F x y+F x y^2+F x y^3$ is of dimension at most one, we can assume that $q_{1,1}=q_{1,2}=q_{1,3}=0$.
We can also assume that $q_{3,1}=0$.
We claim that $q_{2,1}=0$.
Assume to the contrary.
From $\tr(q^2)=0$ we get $q_{2,3}=0$.
Therefore $q=q_{2,1} x^2 y+q_{3,2} x^3 y^2$.
From $\tr(q * z^2)=0$ we get $q_{2,1} z_{3,1} z_{3,2}=0$, a contradiction.
Therefore $q_{2,1}=0$. Then for similar reasons as in Corollary \ref{zform} $q$ must be of the form $q=q_{3,1} x^3 y+q_{2,3} x^2 y^3+q_{3,2} x^3 y^2$.
If $q_{3,1} \neq 0$ or $q_{1,3} \neq 0$ then we are done.
Otherwise, $F q+F z=F x^3 y^2+F x^2 y^3$.
We shall later solve this case separately.

What remains is to check the cases of $z=z_{3,1} x^3 y$ and $z=z_{3,2} x^3 y^2$ separately. (The case of $z=z_{2,3} x^2 y^3$ is similar to the latter.)

Let us assume that $z=z_{3,2} x^3 y^2$.
We have $V'=V''+F z$ where $V''=\set{v \in V : \tr(z^4 v)=0}$.
Let $q$ be an arbitrary element of $V''$.
In particular, $q_{3,2}=0$.
From $\tr(q * z * y)=0$ we get $q_{1,1}=0$.
From $\tr(q * z)=0$ we get $q_{1,2}=0$.
We claim that $q_{2,1}=0$.
Assume to the contrary.
From $\tr(q^2 * x)=0$ we get $q_{1,3}=0$.
From $\tr(q^2)=0$ we get $q_{2,3}=0$.
Then $q=q_{2,1} x^2 y+q_{3,1} x^3 y$.
Since the dimension of $V''$ is two, we can assume that $q_{3,1} \neq 0$.
But then from $\tr(q^2 * z)=0$ we get $q_{2,1} q_{3,1} z=0$, a contradiction.
Therefore $q_{2,1}=0$. Then for similar reasons as in Corollary \ref{zform}, $q$ must be of the form $q=q_{3,1} x^3 y+q_{2,3} x^2 y^3$ or $q=q_{1,3} x y^3+q_{2,3} x^2 y^3$. $V''$ is spanned by $q$ and some element $t$. If $q=q_{3,1} x^3 y+q_{2,3} x^2 y^3$ then $t$ must also be of the form $t=t_{3,1} x^3 y+t_{2,3} x^2 y^3$, and if $q=q_{1,3} x y^3+q_{2,3} x^2 y^3$ then $t$ must also be of the form $t=t_{1,3} x y^3+t_{2,3} x^2 y^3$. Hence, $V''=F x^3 y+F x^2 y^3$ or $V''=F x y^3+F x^2 y^3$. Consequently, $V=F[x^3 y] y+F x^3 y$ or $V=F[x y^3] y+F x y^3$.

Let us assume that $z=z_{3,1} x^3 y$.
We have $V'=V''+F z$.
Let $q$ be an arbitrary element of $V''$.
We can assume that $q_{3,1}=0$.
From $\tr(z * q)=0$ we get $q_{1,3}=0$.
From $\tr(z * q * y)=0$ we get $q_{1,2}=0$.
Assume $q_{2,1} \neq 0$. Then from $\tr(q^2)=0$ we get $q_{2,3}=0$.
Assuming $q_{3,2} \neq 0$ would lead to a contradiction because then $\tr(q^2 * z)=0$ implies $z_{3,1} q_{3,2} q_{2,1}=0$.
Consequently, $q_{3,2}=0$.
Now, $V''$ is spanned by $q$ and some element $t$.
We can assume that $t_{2,1} \neq 0$ too. Then $t=t_{1,1} x y+t_{2,1} x^2 y$, and $V''=F x y+F x^2 y$. In this case $V=F[x] y+F x$.
Assume $q_{2,1}=0$.
Then for similar reasons as in Corollary \ref{zform} $q$ must be of the form $q=q_{2,3} x^2 y^3+q_{3,2} x^3 y^2$. Hence $V''=F x^2 y^3+F x^3 y^2$ and $V=F[x^3 y] y+F x^3 y$.
\end{proof}

The space $V$ as described in the last theorem is maximal, and therefore not contained in a larger $4$-central space $W$, contradiction.

\section{$p$-Central Spaces containing $p$-Central Sets of Size $3$}\label{pset3Sec}

Let $p$ be a prime number, $F$ be an infinite field of characteristic not $p$ containing a primitive $p$th root of unity, $\rho$.
Let $A$ be a cyclic algebra of degree $p$.
Let $V$ be some $p$-central $F$-vector subspace of $A$ of dimension at least $4$.
Assume that $V$ contains a $p$-set of cardinality $3$.
By replacing $\rho$ with some $\rho^k$ for some integer $k$, we can assume that $V$ contains $x$, $y$ and $x^i y^j$ for some integers $i,j$ where $A=F[x,y : x^p=\alpha, y^p=\beta, y x=\rho x y]$.
We would like to prove that $V \subseteq F[w] z+F w$ for some $w,z \in A$ such that $w z=\rho^k z w$ for some integer $k$.
So far, we have managed to do it assuming either $p=5$ or that one element in the $p$-central set is a product of the two others (i.e. $i=j=1$).

\subsection{One element is the product of the two others}

Assume $i=j=1$.

\begin{thm}\label{ij1}
The space $V$ is contained in either $F[x] y+F x$ or $F[y] x+F y$.
\end{thm}

\begin{proof}
We have $V=F x+F y+F x y+F z$ for some $z$.
For any $a,b,c,d \in F$ and any $0<k<p$, $\tr((a x+b y+c x y+d z)^k)=0$.
Consequently, $\tr(x^i *y^j * (x y)^m * z^n)=0$ for any $0<i+j+m+n<p$.
There is the decomposition $z=\sum z_{i,j}$ such that $z_{i,j} \in F x^i y^j$.

How do we prove that if $i,j>1$ then $z_{i,j}=0$?

Without loss of generality, assume $i>j$.
Now, $\tr(y^{i-j} * (x y)^{p-i} * z)=0$, because $i-j+p-i+1=p-j+1<p$.
However, $\tr(y^{i-j} * (x y)^{p-i} * z)=y^{i-j} * (x y)^{p-i} * z_{i,j}$.
If $z_{i,j} \neq 0$ then $y^{i-j} * (x y)^{p-i} * z_{i,j} \neq 0$, because every monomial in this sum becomes $\rho^t$ for some integer $t$, and the number of monomials in this sum is prime to $p$.
Therefore, $z_{i,j}=0$.

So far we proved that $z=z_{1,1}+\dots+z_{1,p-1}+z_{2,1}+\dots+z_{p-1,1}$.

Let us assume that $z_{i,1} \neq 0$ for some $i>1$.

If $i>j$ then take $\tr(y^{i-j} * (x y)^{p-i-1} * z^2)=0$. (This is true because $i-j+p-i-1+2=p-j+1<p$.)
However $\tr(y^{i-j} * (x y)^{p-i-1} * z^2)=y^{i-j} * (x y)^{p-i-1} * z_{i,1}*z_{1,j}$.
For the same reason as before (the number of summands is prime to $p$), this expression is zero if and only if either $z_{i,1}=0$ or $z_{1,j}=0$, which means that $z_{1,j}=0$.

If $i<j$ then take $\tr(y^{j-i} * (x y)^{p-j-1} * z^2)=0$ and continue similarly to prove that $z_{1,j}=0$.

In conclusion, if $V$ contains an element in $F[x] y+F x$ that does not appear in $F[y] x+F y$ then $V \subseteq F[x] y+F x$. Conversely, if $V$ contains an element in $F[y] x+F y$ that does not appear on $F[x] y+F x$ then $V \subseteq F[y] x+F y$.
\end{proof}

\subsection{The degree five case}

Assume now that $p=5$.

\begin{thm}
The space $V$ is contained in one of the following: $F[x] y+F x$, $F[y] x+F y$ or $F[x^3 y^2] x+F (x^3 y^2)^i$ for some $1 \leq i \leq 4$.
\end{thm}

\begin{proof}
The element $x^i y^j$ is contained in $V$. Therefore $i+j \leq 6$, because otherwise it contradicts the fact that $\tr(x^{5-i} * y^{5-j} * (x^i y^j))=0$. Furthermore, $\tr(x * y * (x^i y^j)^2)=0$, which means that the case $i=j=2$ is impossible. Consequently, the possibilities for $x^i y^j$ are $x y^j$, $x^i y$, $x^3 y^3$, $x^2 y^3$, $x^3 y^2$, $x^2 y^4$ and $x^4 y^2$.

Let $z \in V \setminus (F x+F y+F x^i y^j)$. Since $\tr(x^{5-i} * y^{5-j} * (x^i y^j))=0$ for $i+j \geq 7$, $z=z_{1,1}+z_{1,2}+z_{1,3}+z_{1,4}+z_{2,1}+z_{3,1}+z_{4,1}+z_{2,2}+z_{2,3}+z_{3,2}+z_{3,3}+z_{2,4}+z_{4,2}$ where $z_{m,n} \in F x^m y^n$.

Now, $0=\tr(x * y * z^2)=\tr(x * y * z_{2,2}^2)+\tr(x * y * z_{1,3} * z_{3,1})+\tr(x * y * z_{1,1} * z_{3,3})+\tr(x * y * z_{1,2} * z_{3,2})+\tr(x * y * z_{2,1} * z_{2,3})$. This means that if either $z_{1,3}=0$ or $z_{3,1}=0$, $z_{2,3}=0$ or $z_{2,1}=0$, $z_{1,2}=0$ or $z_{3,2}=0$ and either $z_{3,3}=0$ or $z_{1,1}=0$ then $z_{2,2}=0$.

The case of $i=j=1$ has already been dealt with in the Theorem \ref{ij1}.

Assume $x^i y^j=x^2 y$. $0=\tr((x^2 y)^2 * y * z)=\tr((x^2 y)^2 * y * z_{1,2})$, and therefore $z_{1,2}=0$.
$0=\tr((x^2 y)^2 * z)=\tr((x^2 y)^2 * z_{1,3})$, and therefore $z_{1,3}=0$.
$0=\tr((x^2 y) * x^2 * z)=\tr((x^2 y) * x^2 * z_{1,4})$, and therefore $z_{1,4}=0$.
$0=\tr((x^2 y) * x * y * z)=\tr((x^2 y) * x * y z_{2,3})$, and therefore $z_{2,3}=0$.
$0=\tr((x^2 y) * y^2 * z)=\tr((x^2 y) * y^2 * z_{3,2})$, and therefore $z_{3,2}=0$.
$0=\tr((x^2 y) * y * z)=\tr((x^2 y) * y * z_{3,3})$, and therefore $z_{3,3}=0$.
$0=\tr((x^2 y) * x * z)=\tr((x^2 y) * x * z_{2,4})$, and therefore $z_{2,4}=0$.
$0=\tr((x^2 y)^3 * z)=\tr((x^2 y)^3 * z_{4,2})$, and therefore $z_{4,2}=0$.

Since $z_{3,3}=z_{3,2}=z_{2,3}=z_{1,3}=0$, $z_{2,2}=0$.

Consequently, $V \subseteq F[x] y+F x$.

Assume $x^i y^j=x^3 y$. $0=\tr((x^3 y)^3 * z)=\tr((x^3 y)^3 * z_{1,2})$, and therefore $z_{1,2}=0$.
$0=\tr((x^3 y) x * y * z)=\tr((x^3 y) x * y * z_{1,3})$, and therefore $z_{1,3}=0$.
$0=\tr((x^3 y) * x * z)=\tr((x^3 y) * x * z_{1,4})$, and therefore $z_{1,4}=0$.
$0=\tr((x^3 y) * y * z)=\tr((x^3 y) * y * z_{2,3})$, and therefore $z_{2,3}=0$.
$0=\tr((x^3 y)^2 * x * z)=\tr((x^3 y)^2 * x * z_{3,3})$, and therefore $z_{3,3}=0$.
$0=\tr((x^3 y) * z)=\tr((x^3 y) * z_{2,4})$, and therefore $z_{2,4}=0$.
$0=\tr((x^3 y)^2 * y * z)=\tr((x^3 y)^2 * y * z_{4,2})$, and therefore $z_{4,2}=0$.

Since $z_{1,2}=z_{1,3}=z_{2,3}=z_{3,3}=0$, $z_{2,2}=0$.

Since $\tr(x * (x^3 y) * z^2)=0$, we have $x * (x^3 y) * z_{3,2}^2=0$, which means that $z_{3,2}=0$, and therefore $V \subseteq F[x] y+F x$.

Consequently, $V \subseteq F[x] y+F x$.

Assume $x^i y^j=x^4 y$. then $V$ must be a subspace of $F[x] y+F x+F[x^4 y] x+F x^4 y$ because the traces of the followings are nonzero:
$0=\tr((x^4 y) * y^2 * z)=\tr((x^4 y) * y^2 * z_{1,2})$, and therefore $z_{1,2}=0$.
$0=\tr((x^4 y) * y * z)=\tr((x^4 y) * y * z_{1,3})$, and therefore $z_{1,3}=0$.
$0=\tr((x^4 y)^2 * y * z)=\tr((x^4 y)^2 * y * z_{2,2})$, and therefore $z_{2,2}=0$.
$0=\tr((x^4 y) * z)=\tr((x^4 y) * z_{1,4})$, and therefore $z_{1,4}=0$.
$0=\tr((x^4 y)^2 * z)=\tr((x^4 y)^2 * z_{2,3})$, and therefore $z_{2,3}=0$.
$0=\tr((x^4 y)^3 * z)=\tr((x^4 y)^3 * z_{3,2})$, and therefore $z_{3,2}=0$.

Since $\tr(z^2)=0$ we have either $z_{3,1}=0$ or $z_{2,4}=0$.
Since $\tr(z^2 * y)=0$ we have either $z_{2,1}=0$ or $z_{3,3}=0$.
Since $\tr(z^2 * y^2)=0$ we have either $z_{1,1}=0$ or $z_{4,2}=0$.

If $z_{3,1}=z_{2,1}=z_{4,2}=0$ then since $\tr(z^3)=0$, either $z_{1,1}=0$ or $z_{3,3}=0$, and since $\tr(z^3 * y)=0$, either $z_{1,1}=0$ or $z_{2,4}=0$.

If $z_{3,1}=z_{1,1}=z_{3,3}=0$ then since $\tr(z^3)=0$, either $z_{2,1}=0$ or $z_{4,2}=0$, and since $\tr(z^2 * x)=0$, either $z_{2,1}=0$ or $z_{2,4}=0$.

If $z_{1,1}=z_{2,1}=z_{2,4}=0$ then since $\tr(z^3 * y)=0$, either $z_{3,1}=0$ or $z_{4,2}=0$, and since $\tr(z^3 * x)=0$, either $z_{3,1}=0$ or $z_{3,3}=0$.

Similarly, if $z_{3,1}=z_{3,3}=z_{4,2}=0$ then either $z_{2,4}=0$ or $z_{2,1}=z_{1,1}=0$, if $z_{2,1}=z_{2,4}=z_{4,2}=0$ then either $z_{3,3}=0$ or $z_{3,1}=z_{1,1}=0$, and if $z_{1,1}=z_{2,4}=z_{3,3}=0$ then either $z_{4,2}=0$ or $z_{3,1}=z_{2,1}=0$.

All in all, $V$ is contained in either $F[x] y+F x$ or $F[x^4 y] x+F x^4 y$.

Similarly, if $x^i y^j=x y^i$ for $i=2$ or $i=3$ then $V \subseteq F[y] x+F y$, and if $V$ contains $x y^4$ then either $V \subseteq F[y] x+F y$ or $V \subseteq F[x y^4] y+F x y^4$.

Assume $x^i y^j=x^2 y^3$.
$0=\tr((x^2 y^3) * x * y * z)=\tr((x^2 y^3) * x * y * z_{2,1})$, and therefore $z_{2,1}=0$.
$0=\tr((x^2 y^3) * y * z)=\tr((x^2 y^3) * y * z_{3,1})$, and therefore $z_{3,1}=0$.
$0=\tr((x^2 y^3)^3 * z)=\tr((x^2 y^3)^3 * z_{4,1})$, and therefore $z_{4,1}=0$.
$0=\tr((x^2 y^3) x^2 * z)=\tr((x^2 y^3) x^2 * z_{1,2})$, and therefore $z_{1,2}=0$.
$0=\tr((x^2 y^3)^2 * y *z)=\tr((x^2 y^3)^2 * y *z_{1,3})$, and therefore $z_{1,3}=0$.
$0=\tr((x^2 y^3)^2 *z)=\tr((x^2 y^3)^2 *z_{1,4})$, and therefore $z_{1,4}=0$.
$0=\tr((x^2 y^3) * z)=\tr((x^2 y^3) * z_{3,2})$, and therefore $z_{3,2}=0$.
$0=\tr((x^2 y^3) * x * z)=\tr((x^2 y^3) * x * z_{2,2})$, and therefore $z_{2,2}=0$.

Since $\tr(x * (x^2 y^3) * z^2)=0$, $x * (x^2 y^3) * z_{1,1}^2=0$, which means that $z_{1,1}=0$. Consequently, $V \subseteq F[x^2 y^3] y+F x^2 y^3$.

Consequently, $V \subseteq F[x^2 y^3] y+F x^2 y^3+F x y$

Similarly, if $x^i y^j=x^3 y^2$ then $V \subseteq F[x^3 y^2] x+F x^3 y^2$.

Assume $x^i y^j=x^2 y^4$.
$0=\tr((x^2 y^4)^2 * y * z)=\tr((x^2 y^4)^2 * y * z_{1,1})$, and therefore $z_{1,1}=0$.
$0=\tr((x^2 y^4)^2 * x * z)=\tr((x^2 y^4) * x * z_{2,1})$, and therefore $z_{2,1}=0$.
$0=\tr((x^2 y^4)^2 * z)=\tr((x^2 y^4)^2 * z_{1,2})$, and therefore $z_{1,2}=0$.
$0=\tr((x^2 y^4)^2 * x * z)=\tr((x^2 y^4)^2 * x * z_{1,3})$.
$0=\tr((x^2 y^4) * z)=\tr((x^2 y^4) * z_{3,1})$, and therefore $z_{3,1}=0$.

Since $z_{1,1}=z_{2,1}=z_{1,2}=z_{1,3}=0$, $z_{2,2}=0$.

Since $\tr(x * (x^2 y^4) * z^2)=0$, $x * (x^2 y^4) * z_{1,3}^2=0$, which means that $z_{1,3}=0$.

Now, $\tr(z^k)=\tr((z_{3,2}+z_{2,3}+z_{4,1}+z_{1,4})^k)$ for $1 \leq k \leq 4$. Therefore, since $z$ is $5$-central, so is $z_{3,2}+z_{2,3}+z_{4,1}+z_{1,4}$. However, this is an element of the field $F[(x y^4)^k : k \neq 0]$. Consequently, all of the four summands but one are equal to zero. Hence $V \subseteq F [x y^4] x+F (x y^4)^k$ for some $1 \leq k \leq 4$.

Similarly, if $x^i y^j=x^4 y^2$ then $V \subseteq F [x y^4] x+F (x y^4)^k$ for some $1 \leq k \leq 4$.

Assume $x^i y^j=x^3 y^3$.
$0=\tr(x * y * (x^3 y^3) * z)=\tr(x * y * (x^3 y^3) * z_{1,1})$, and therefore $z_{1,1}=0$.
$0=\tr(x * (x^3 y^3) * z)=\tr(x * (x^3 y^3) * z_{1,2})$, and therefore $z_{1,2}=0$.
$0=\tr(y * (x^3 y^3) * z)=\tr(y * (x^3 y^3) * z_{2,1})$, and therefore $z_{2,1}=0$.
$0=\tr((x^3 y^3) * z)=\tr((x^3 y^3) * z_{2,2})$, and therefore $z_{2,2}=0$.

Since $\tr(x * (x^3 y^3) * z^2)=0$, $z_{3,1}=0$, and since $\tr(y * (x^3 y^3) * z^2)=0$, $z_{1,3}=0$.

From here on the proof is similar to what we already did in the case of $x^i y^j=x^2 y^4$, to prove that $V \subseteq F [x y^4] x+F (x y^4)^k$ for some $1 \leq k \leq 4$.
\end{proof}

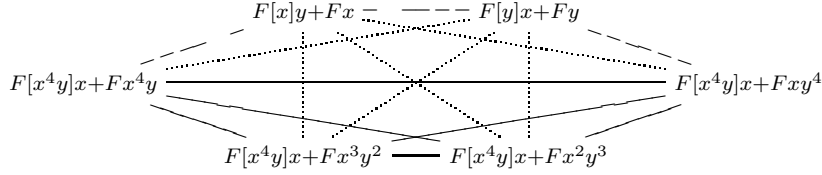
\begin{figure}
\begin{equation*}
\xymatrix@R=12pt@C=18pt{
& {\scriptstyle{F[x] y+F x}}\ar@{--}|(0.5){}[r]\ar@{--}|(0.5){}[dl]\ar@{..}|(0.5){}[dd]
 \ar@{..}|(0.5){}[ddr]\ar@{..}|(0.5){}[drr]& {\scriptstyle{F[y] x+F y}}
 \ar@{..}|(0.5){}[dll]\ar@{..}|(0.5){}[ddl]\ar@{..}|(0.5){}[dd]
 \ar@{--}|(0.5){}[dr]
  & {}\\
{\scriptstyle{F[x^4 y] x+F x^4 y}}\ar@{-}|(0.5){}[rrr]\ar@{-}|(0.5){}[dr]\ar@{-}|(0.5){}[drr] & {} & {} & {\scriptstyle{F[x^4 y] x+F x y^4}}\ar@{-}|(0.5){}[dll]\ar@{-}|(0.5){}[dl]\\
{} & {\scriptstyle{F[x^4 y] x+F x^3 y^2}}\ar@{-}|(0.5){}[r] & {\scriptstyle{F[x^4 y] x+F x^2 y^3}} & {}
}
\end{equation*}
\caption{The maximal $5$-central spaces containing $x,y,x^i x^j$ and the dimensions of their intersections. A continuous line stands for intersection of dimension $5$, a broken line for $3$ and a dotted line for $2$.}\label{diag}
\end{figure}

\begin{cor}
\begin{enumerate}
\item The maximal $5$-central spaces containing $5$-central sets of size $3$ are then of dimension $6$.
\item The intersection between every two $5$-central spaces containing $x,y$ and some third element of the form $x^i y^j$ can be $2$, $3$ or $5$. See Diagram \ref{diag} for the spaces and their intersections.
\item A $5$-central space $V$ of dimension greater or equal to $4$ which contains a $5$-central set of size $3$ is contained in $4$ different $p$-central spaces of degree $6$ if and only if $V \subseteq F[z] w+F z$ for some $w$ and $z$ satisfying $z w=\rho^k w z$ for some integer $k$. Otherwise, $V$ is contained in exactly one $6$ dimensional $p$-central space.
\end{enumerate}
\end{cor}

\section{$3$-Central Spaces spanned by $3$-Central Sets}\label{threesec}
Let $A$ be a central simple algebra over an infinite field $F$ of characteristic not $3$ containing a primitive 3rd root of unity $\rho$.

Let $\mathcal{X}$ be the set of all 3-central elements in $A$.
We build a directed graph $(\mathcal{X},E)$ by drawing an edge from $y$ to $x$
$$\xymatrix@C=20px@R=20px{y \ar@{->}[r]  & x}$$ if $y x y^{-1}=\rho x$.
For any subset $B \subset \mathcal{X}$, $(B,E_B)$ is the subgraph obtained by taking the vertices in $B$ and all the edges between them.

\begin{rem}\label{rhosetsize2}
If $\set{x,y}$ is a 3-central set spanning a 3-central space then either $\xymatrix@C=20px@R=20px{ x\ar@{->}[r]  & y}$ or $\xymatrix@C=20px@R=20px{ x\ar@{<-}[r]  & y}$.
\end{rem}

\begin{proof}
If $x y=y x$ then $x^2*y=3 x^2 y \in F$ which means that $y \in F x$, contradiction.
\end{proof}

According to \cite[Corollary 2.2]{Chapman}, a set $\set{x_1,\dots,x_m}$ spans a 3-central space if and only if every subset of cardinality three $\set{x_i,x_j,x_k}$ spans a 3-central space.
Therefore we will start with the set of cardinality 3.

\begin{lem}\label{size3}
Given a 3-central set $\set{x,y,z}$,  $F x+F y+F z$ is 3-central if and only if (up to some permutation on $\set{x,y,z}$) either
\begin{equation*}
\xymatrix@R=12pt@C=18pt{
{x}\ar@{->}|(0.5){}[r]\ar@{->}|(0.5){}[d] & {y}\\
{z}\ar@{->}|(0.5){}[ru] & {}
}
\end{equation*}
or $x y z \in F$, in which case
\begin{equation*}
\xymatrix@R=12pt@C=18pt{
{x}\ar@{<-}|(0.5){}[r]\ar@{->}|(0.5){}[d] & {y}\\
{z}\ar@{->}|(0.5){}[ru] & {}
}
\end{equation*}
\end{lem}

\begin{proof}
From Remark \ref{rhosetsize2}, the only possible graphs (up to permutation of the vertices) are the two graphs above.
In the first case, $x * y * z=0$, so there are no extra conditions.
In the second case, $x * y * z=-3 \rho^{-1} x y z \in F$.
The opposite direction is a straight-forward computation.
\end{proof}

Let $B$ be a 3-central set spanning a 3-central space. We will now study the properties of the directed graph $(B,E_B)$.
By a cycle we always mean a \textbf{simple directed cycle}.

\begin{prop}\label{direction}
If $(B,E_B)$ contains a cycle of length $3$
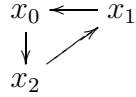
\begin{equation*}
\xymatrix@R=12pt@C=18pt{
{x_0}\ar@{<-}|(0.5){}[r]\ar@{->}|(0.5){}[d] & {x_1}\\
{x_2}\ar@{->}|(0.5){}[ru] & {}
}
\end{equation*}
then for every $y \in B \setminus \set{x_0,x_1,x_2}$, either $\xymatrix@C=20px@R=20px{x_k \ar@{->}[r]  & y}$ for any $k \in \set{0,1,2}$ or $\xymatrix@C=20px@R=20px{x_k \ar@{<-}[r]  & y}$ for any $k \in \set{0,1,2}$.
\end{prop}

\begin{proof}
If $\xymatrix@C=20px@R=20px{x_0 \ar@{->}[r]  & y}$ and $\xymatrix@C=20px@R=20px{x_1 \ar@{<-}[r]  & y}$ then
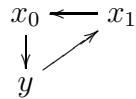
\begin{equation*}
\xymatrix@R=12pt@C=18pt{
{x_0}\ar@{<-}|(0.5){}[r]\ar@{->}|(0.5){}[d] & {x_1}\\
{y}\ar@{->}|(0.5){}[ru] & {}
}
\end{equation*}
which means that $y x_0 x_1 \in F$. Since $x_0 x_1 x_2 \in F$, we get $y \in F x_2$, which contradicts the linear independence.
The rest of the proof repeats the same idea.
\end{proof}

\begin{prop}
The cycles of $(B,E_B)$ are vertex-disjoint.
\end{prop}

\begin{proof}
First assume that
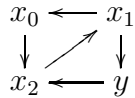
\begin{equation*}
\xymatrix@R=12pt@C=18pt{
{x_0}\ar@{<-}|(0.5){}[r]\ar@{->}|(0.5){}[d] & {x_1}\ar@{->}|(0.5){}[d]\\
{x_2}\ar@{->}|(0.5){}[ru] &  y\ar@{->}|(0.5){}[l] & {}
}
\end{equation*}
Then $y x_1 x_2 \in F$ whereas $x_0 x_1 x_2 \in F$, which means that $y \in F x_0$, and that contradicts the linear independence.

Assume that
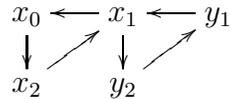
\begin{equation*}
\xymatrix@R=12pt@C=18pt{
{x_0}\ar@{<-}|(0.5){}[r]\ar@{->}|(0.5){}[d] & {x_1}\ar@{<-}|(0.5){}[r]\ar@{->}|(0.5){}[d] & y_1\\
{x_2}\ar@{->}|(0.5){}[ru] &  y_2\ar@{->}|(0.5){}[ru] & {}
}
\end{equation*}
From Proposition \ref{direction} we have $\xymatrix@C=20px@R=20px{x_0 \ar@{->}[r]  & y_2}$ and $\xymatrix@C=20px@R=20px{y_1 \ar@{->}[r]  & x_0}$.
But then
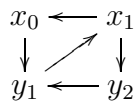
\begin{equation*}
\xymatrix@R=12pt@C=18pt{
{x_0}\ar@{<-}|(0.5){}[r]\ar@{->}|(0.5){}[d] & {x_1}\ar@{->}|(0.5){}[d]\\
{y_1}\ar@{->}|(0.5){}[ru] &  y_2\ar@{->}|(0.5){}[l] & {}
}
\end{equation*}
and we saw already that this is impossible.
\end{proof}

\begin{prop}
There are no cycles of length greater than 3.
\end{prop}

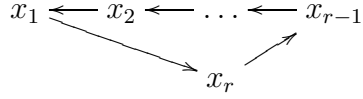
\begin{proof}
Assume
\begin{equation*}
\xymatrix@R=12pt@C=18pt{
{x_1}\ar@{<-}|(0.5){}[r]\ar@{->}|(0.5){}[drr] & {x_2}\ar@{<-}|(0.5){}[r] & \dots\ar@{<-}|(0.5){}[r]& x_{r-1}\\
& &{x_r}\ar@{->}|(0.5){}[ru] & {}
}
\end{equation*}
for some $r \geq 4$.
Let $i$ be the maximal integer between $1$ and $r$ such that $\xymatrix@C=20px@R=20px{x_i \ar@{->}[r]  & x_1}$.
Now, $\xymatrix@C=20px@R=20px{x_1 \ar@{->}[r]  & x_{i+1}}$.
Therefore
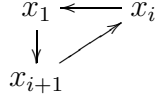
\begin{equation*}
\xymatrix@R=12pt@C=18pt{
{x_1}\ar@{<-}|(0.5){}[r]\ar@{->}|(0.5){}[d] & {x_i}\\
{x_{i+1}}\ar@{->}|(0.5){}[ru] & {}
}
\end{equation*}
If $i \geq 3$ then according to Proposition \ref{direction}, $\xymatrix@C=20px@R=20px{x_1 \ar@{->}[r]  & x_{i-1}}$, which implies that $i \neq 3$, or in other words $i \geq 4$. Let $j$ be the minimal index for which $\xymatrix@C=20px@R=20px{x_1 \ar@{->}[r]  & x_{j+1}}$. In particular $\xymatrix@C=20px@R=20px{x_j \ar@{->}[r]  & x_1}$. Now, $j+1 \leq i-1$, which means that
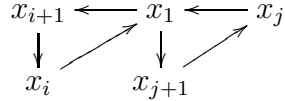
\begin{equation*}
\xymatrix@R=12pt@C=18pt{
{x_{i+1}}\ar@{<-}|(0.5){}[r]\ar@{->}|(0.5){}[d] & {x_1}\ar@{<-}|(0.5){}[r]\ar@{->}|(0.5){}[d] & x_j\\
{x_i}\ar@{->}|(0.5){}[ru] &  x_{j+1}\ar@{->}|(0.5){}[ru] & {}
}
\end{equation*}
But this is impossible.
If $i=2$ then according to Proposition \ref{direction}, $\xymatrix@C=20px@R=20px{x_4 \ar@{->}[r]  & x_1}$ which contradicts the maximality of $i$.
\end{proof}

As a consequence we obtain the following theorem:
\begin{thm}
A 3-central subset $B$ of $\mathcal{X}$ spans a 3-central space if and only if the graph $(B,E_B)$ satisfies the following axioms:
\begin{enumerate}
\item For every two distinct elements $x,y \in B$, either $\xymatrix@C=20px@R=20px{ x\ar@{->}[r]  & y}$ or $\xymatrix@C=20px@R=20px{ x\ar@{<-}[r]  & y}$
\item All cycles are of length 3.
\item The product of all the elements in a cycle is in $F$.
\item The cycles are vertex-disjoint.
\end{enumerate}
\end{thm}

\begin{proof}
The straight-forward direction is an immediate result of what we did so far.
The opposite direction is a result of the fact that every three elements in this set span a 3-central space according to Lemma \ref{size3}.
\end{proof}

The following remark may help the reader get an idea of how the graph $(B,E(B))$ looks like:

\begin{rem}
Assume $B$ is a 3-central set spanning a 3-central space.
Let $\sim$ be the following equivalence relation: $x \sim y$ if and only if $x=y$ or $x$ and $y$ belong to the same cycle in $(B,E_B)$.
As we already saw, this equivalence relation is also direction preserving in the sense that if $\xymatrix@C=20px@R=20px{ x\ar@{->}[r]  & z}$ and $x \sim y$ then $\xymatrix@C=20px@R=20px{ y\ar@{->}[r]  & z}$ and if $\xymatrix@C=20px@R=20px{ z\ar@{->}[r]  & x}$ and $x \sim y$ then $\xymatrix@C=20px@R=20px{ z\ar@{->}[r]  & y}$.
Define an order $\leq$ on the equivalence classes: $[x] \leq [z]$ if $[x]=[z]$ or $\xymatrix@C=20px@R=20px{ z\ar@{->}[r]  & x}$.
Then the set of equivalence classes is a fully ordered set.

One can therefore visualize the graph as graded into levels, where in each level we have either one element or a cycle, and each element has edges going from it to all the elements in the lower levels.
\end{rem}

\begin{cor}
Given a 3-central set $B$ spanning a 3-central space, if $\#B=m$ then the longest path
$\xymatrix@C=20px@R=20px{ x_1\ar@{->}[r]  & x_2\ar@{->}[r] & \dots\ar@{->}[r] & x_r}$ in the graph $(B,E_B)$ satisfying $\xymatrix@C=20px@R=20px{ x_i\ar@{->}[r]  & x_j}$ for any $1 \leq i <j \leq r$ is of length no less than $m-\lfloor \frac{m}{3} \rfloor$.
\end{cor}

\begin{proof}
Take $B$ and take off exactly one element from each cycle.
The number of elements taken off is at most $\lfloor \frac{m}{3} \rfloor$, and what is left satisfies the required condition.
\end{proof}

\begin{cor}\label{upperbound}
The maximal 3-central set spanning a 3-central space in $A$ is of cardinality $3 n+1$.
\end{cor}

\begin{proof}
We are already familiar with 3-central spaces spanned by $3$-central sets of size $3 n+1$.
According to the previous corollary, if we have a $p$-central set $B$ of size $3 n+2$ spanning a 3-central space then we have a path in $(B,E_B)$
$$\xymatrix@C=20px@R=20px{ x_1\ar@{->}[r]  & x_2\ar@{->}[r] & \dots\ar@{->}[r] & x_{2 n+2}}$$ satisfying $\xymatrix@C=20px@R=20px{ x_i\ar@{->}[r]  & x_j}$
for any $1 \leq i <j \leq 2 n+2$. Then the set $B$ generates over $F$ a tensor product of $n+1$ cyclic algebras of degree 3
$$F[x_1,x_2] \otimes F[x_1 x_2^{-1} x_3,x_1 x_2^{-1} x_4] \otimes \dots \otimes F[(\prod_{k=1}^n x_{2 k-1} x_{2 k}^{-1}) x_{2 n+1},(\prod_{k=1}^n x_{2 k-1} x_{2 k}^{-1}) x_{2 n+2}],$$
contradiction.
\end{proof}

\section{Algebras of Fixed Degrees with $3$-Central Subspaces}\label{influenceSec}
In this section we study the effect of the existence of $3$-central spaces in algebras of fixed degrees. We focus on degree 3.
We show that for a field extension $K/F$, if a central simple $K$-algebra $A$ of degree 3 contains an $F$-vector subspace $V$ such that $v^3 \in F$ for all $v \in F$ and $[V:F]=3$ then $A$ is a restriction of a central simple $F$-algebra. We provide a counterexample in case $[V:F]=2$.

\subsection{Dimension 3}

\begin{lem}\label{dim3lem}
If a simple (noncentral) $F$-algebra contains a 3-dimensional $F$-vector subspace with third powers in $F$ then it is a restriction of a central simple $F$-algebra.
\end{lem}

\begin{proof}
From \cite{Raczek} it is known that the $F$-vector space with third powers in $F$ contains two elements $\xi$ and $\mu$ such that $\xi \mu=\rho \mu \xi$.
Consequently the algebra is a restriction of $F[\xi,\mu]$ which is a cyclic algebra of degree $3$ over $F$.
\end{proof}

\subsection{Dimension 2}
Let $K/F$ be an extension of dimension $3$, with a third root of
unity $\rho \in F$. Let $\alpha \in \mul{F}$ and $b \in \mul{K}$.

\begin{lem} \label{chap_dec_com}
A division $F$-algebra $D$ contains a 3-central space $F x+F y$ such that $x^3=\alpha$, $y^3=\beta$, $x^2 * y=0$ and $x * y^2=3 \delta$ iff
there exists an element $u$ such that $x u=\rho u x$ and $\frac{\beta-b}{\alpha b^2}+\frac{\delta^3}{\alpha^2 b^3}$ has a cubic root in $F[b]$ where $b=u^3$.
\end{lem}

\begin{proof}
($\Rightarrow$) According to \cite{Haile}, there exists an element $u$ for which $x u=\rho u x$ and $y=u+a_1 u^2 x+\frac{3 \delta}{\alpha b (\rho-1)^2} u^2 x^2$ where $a_1 \in \cent{F[x,u]}=F[b]$.
$y^3=b+a_1^3 \alpha b^2+\frac{27 \delta^3}{\alpha^3 b^3 (\rho-1)^6} \alpha^2 b^2=\beta$, hence
$a_1^3=\frac{\beta-b}{\alpha b^2}+\frac{\delta^3}{\alpha^2 b^3}$.

($\Leftarrow$) If indeed $\frac{\beta-b}{\alpha b^2}+\frac{\delta^3}{\alpha^2 b^3}$ has a cubic root in $F[b]$ then it can be denoted by $a_1$, and so the element $y=u+a_1 u^2 x+\frac{3 \delta}{\alpha b (\rho-1)^2} u^2 x^2$ satisfies $y^3=\beta$, $x^2 * y=0$ and $x * y^2=\delta$.
\end{proof}

\begin{rem} \label{chap_dec_com_K}
The standard generator $x$ in the
cyclic algebra $A = (\alpha,b)_K$ extends to a 2-dimensional
3-central space (with $\gamma = 0$ and $\delta$ as in Lemma \ref{chap_dec_com})
over $F$ iff
$$(\delta^3 + \alpha \beta b - \alpha b^2)\alpha$$
is a third power in $K$ for a suitable $\beta \in F$.
\end{rem}

\begin{proof}
The same proof as in Lemma \ref{chap_dec_com}. Mind that $(\delta^3 + \alpha \beta b - \alpha b^2) \alpha=(\frac{\beta-b}{\alpha b^2}+\frac{\delta^3}{\alpha^2 b^3}) \alpha^3 b^3$.
\end{proof}

\begin{rem}
We may assume $\delta = 0$ (orthogonal space) or $\delta = 1$
(non-orthogonal).
\end{rem}

Namely, for some $d \in K$, $(\delta^3 + \alpha \beta b - \alpha
b^2)\alpha = -d^3$. Let $\theta = \frac{\alpha\beta}{2}$, and put
$b =\frac{1}{\alpha}c+\frac{1}{2}\beta$ for $c \in K$: the
equation becomes
\begin{equation}\label{Basic1}
c^2 = d^3 + \delta^3 \alpha + \theta^2.
\end{equation}

{\bf Fact}. Let $c \in K$ and
$\theta \in F$. $Fx$ extends to a 2-dimensional
3-central space over $F$ in $(\alpha,c+\theta) = K[x,y]$ iff
for some $\delta \in F$ and $d \in K$, \eq{Basic1} holds.
%

\begin{thm}
Let $k$ be a field with third root of unity, and let $\delta \in
k$. There exist: \begin{itemize} \item  a field $F$ containing $k$
\item with a cubic Galois extension $K/F$ \item and a cyclic algebra $A$
of degree $3$ over $K$, \end{itemize}
such that $A$ admits a
$2$-dimensional $3$-central $F$-space of type $\delta$, and
$\cores[K/F]A$ is non-trivial. In particular $A$ is not restricted
from $F$.
\end{thm}

The proof occupies the rest of this section.

\begin{lem}\label{Lem0}
Suppose $K$ has commuting automorphisms $\tau_0,\tau_1,\tau_2$ of
order $2$ and an automorphism $\s$ of order $3$, such that $\s
\tau_\ell \s^{-1} = \tau_{\ell+1 \pmod{3}}$. Let $F_0 =
K^{\tau_0,\tau_1,\tau_2,\sigma}$. Let $\alpha,\delta,\theta \in
F_0$. Suppose $d \in K_0 = K^{\tau_0,\tau_1,\tau_2}$. Suppose $c
\in K$ is an element such that $\tau_{\ell}\s^{\ell'}(c) =
(-1)^{\delta_{\ell,\ell'}}\s^{\ell'}(c)$, where Kronecker's delta
applies to $\ell$ and $\ell'$ modulo $3$. Assume $c^2 = d^3 +
\delta^3\alpha + \theta^2$.

If $A = (\alpha,c+\theta)_K$ is not split, then its corestriction
to $F$ is not split as well.
\end{lem}
\begin{proof}
Write $c_\ell = \s^{\ell}c$ and $d_\ell = \s^\ell(d)$. By the
projection formula, the corestriction is $\cores[K/F]A =
(\alpha,(c_0+\theta) (c_1+\theta)(c_2+\theta))_F$. We may assume
$\alpha$ is not a cube in $K$, so let $\tilde{K} = K[x \suchthat
x^3 = \alpha]$, with the action of $\Gal(K/F_0)$ extended by
acting trivially on $x$, and let $\tilde{F} = \tilde{K}^\s$.

Suppose
$$(c_0+\theta) (c_1+\theta)(c_2+\theta) = \Norm[\tilde{K}/\tilde{F}](f)$$
for some $f \in \tilde{F}$. By assumption $\alpha \in F_0$, so the
$\tau_\ell$ commute with the elements of
$\Gal(\tilde{K}/\tilde{F})$. Taking the norm with respect to
$\tau_1$ and $\tau_2$, we get
$$(c_0+\theta)^4 (c_1^2-\theta^2)^2(c_2^2-\theta^2)^2 =
\Norm[\omega](\Norm_{\tau_1}\Norm_{\tau_2}f);$$ notice that the
$\tau_\ell$ do not act on $\tilde{F}$, so the most we can say is
that $\Norm_{\tau_1}\Norm_{\tau_2}f \in
\tilde{K}^{\tau_1,\tau_2}$.

But $c_{\ell}^2 - \theta^2 = d^3 + \delta^3 \alpha =
\Norm[\tilde{K}/\tilde{F}](d_\ell+\delta x)$, with $d_\ell +
\delta x \in \tilde{K}_0$; so $(c_0+\theta)^4$, and therefore
$c_0+\theta$, are norms in the extension
$\tilde{K}^{\tau_1,\tau_2}/K^{\tau_1,\tau_2}$. This proves that the algebra $A_0 =
(\alpha,c_0+\theta)_{K^{\tau_1,\tau_2}}$ is split, and so $A = A_0
\tensor[K^{\tau_1,\tau_2}] K$ is split as well.
\end{proof}

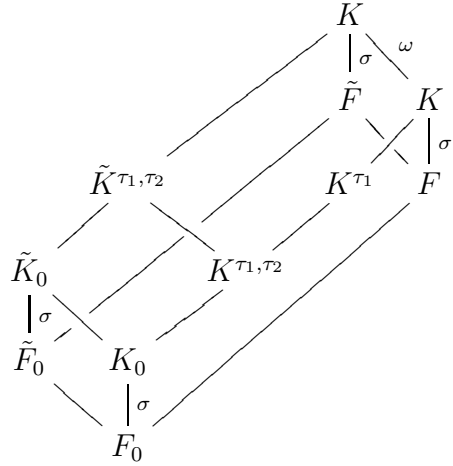
\begin{figure}
\begin{equation*}
\xymatrix@R=14pt@C=8pt{ %
{} & {} & {} & \tilde{K} 
\ar@{-}[d]^{\s} \ar@{-}[dr]^{\omega} \ar@{-}[lldd] & {} \\
{} & {} & {} & \tilde{F} \ar@{-}|(0.5){\!\hole\!}[dr] \ar@{-}|(0.49)\hole|(0.84)\hole[lllddd] & K 
\ar@{-}[d]^{\s}  \ar@{-}[ld]\\
{} & \tilde{K}^{\tau_1,\tau_2} \ar@{-}[dl] \ar@{-}[dr] & {} & K^{\tau_1} \ar@{-}[ld] & F \ar@{-}[lllddd] \\
\tilde{K}_0 
\ar@{-}[d]^{\s} \ar@{-}[dr] & {} &  K^{\tau_1,\tau_2}  \ar@{-}[ld] & {} & {} \\
\tilde{F}_0 \ar@{-}[dr] & K_0 
\ar@{-}[d]^{\s} & {} & {} & {} \\
{} & F_0 & {} & {} & {} \\
}
\end{equation*}
\end{figure}

Let us realize the construction of Lemma \ref{Lem0}. Let $k$ be a field
with $3$rd root of unity $\rho \in k$, and let $\delta \in k$ be
arbitrary, but fixed (eventually we take $\delta = 0$ or $\delta =
1$). Let $\tilde{K}_0$ be the transcendental extension
$\tilde{K}_0 = k(x,\theta,d_0,d_1,d_2)$, and set $\alpha = x^3$.
Let $\tilde{K}$ be the field extension $\tilde{K}_0[c_0,c_1,c_2]$,
subject to the relations
$$c_\ell^2 = d_\ell^3 + (\delta^3\alpha+\theta^2)$$
for $\ell = 0,1,2$. Clearly $\dimcol{\tilde{K}}{\tilde{K}_0} = 8$.
Define an automorphism $\sigma$ of $\tilde{K}$ by fixing $x$ and
$\theta$, and permuting the $d_{\ell}$ and $c_{\ell}$ cyclically.
Then let $\tilde{F} = \tilde{K}^{\s}$ and $\tilde{F}_0 =
\tilde{K}_0^{\sigma}$. Note that $\Gal(\tilde{K}/\tilde{F}_0)$ is
a wreath product $\Z_2 \wr \Z_3$; $\tilde{K}_0$ is Galois over
$\tilde{F}_0$, but $\tilde{F}$ is not.

Let $\omega$ be the automorphism of $\tilde{K}$ defined by
$\omega(x) = \rho x$ and fixing all other generators. Let $K =
\tilde{K}^{\omega}$ and similarly $F = \tilde{F}^\omega$, $K_0 =
\tilde{K}_0^{\omega}$ and $F_0 = \tilde{F}_0^{\omega}$. Since
$\omega$ commutes with $\Gal(\tilde{K}/\tilde{K}_0)$, we have that
$\tilde{K} = \tilde{K}_0 \tensor[F_0] K$.

We take $K/F$ to be the field extension asserted in the theorem,
with the algebra $A = (\alpha,c_0+\theta)_K$. It remains to prove
that $A$ is a division algebra. This may seem obvious, as the
construction is fairly generic, but \eq{Basic1} imposes a severe
restriction --- implying, in fact, that $\Norm[\tau_0](c+\theta)$
is a norm in $\tilde{K}/K$, so $\cores[K/K^{\tau_0}] A$ is
split.

Take $c = c_0$ and $d = d_0$, so \eq{Basic1} is solved. Then
$\Norm[K/F](c+\theta) = (c_0+\theta)(c_1+\theta)(c_2+\theta)$,
which is clearly in $F$. We need to show that this element is not
a norm in $\tilde{F}/F$.

Let $\tau_0,\tau_1,\tau_2 \in \Gal(K/K_0)$ be defined by
$\tau_i(c_i) = -c_i$ and $\tau_i(c_j) = c_j$ for $j \neq i$.

Suppose %
\begin{equation}\label{Basic-o}
c_0+\theta = \Norm[\omega](h)
\end{equation}
for some $h \in \tilde{K}^{\tau_1,\tau_2} = \tilde{K}_0[c_0]$. Let
$k' = k(d_1,d_2)$. So we need to prove that $c_0+\theta$ is not a
norm from $k'(\theta,x,d_0)[c_0]$ to $k'(\theta,\alpha,d_0)[c_0]$.

\begin{lem}\label{lemC}
Let $k'$ be any field of characteristic not $2$ or $3$, and let
$c$ be defined by \eq{Basic1}. The element $c +\theta$ is not a
norm in the extension $k(\theta,x,d)[c]/k(\theta,x^3,d)[c]$, where
$c^2 = d^3 + \delta^3 x^3+\theta^2$.
\end{lem}
\begin{proof}
Indeed, write $h = h_2^{-1}(h_0 + h_1 c)$ for $h_0,h_1,h_2 \in
k[x,\theta,d_0]$. Then $\Norm[\omega](h_2) (c+\theta) =
\Norm[\omega](h_0 + h_1 c)$, which are the equations in
Lemma \ref{lemB} below, showing that $h_2 = 0$ contrary to assumption.
\end{proof}

\begin{lem}\label{lemB}
Let $h_0,h_1,h_2 \in k[\theta,x,d_0]$ be polynomials satisfying
\begin{eqnarray}
\Norm[\omega](h_2) & = & \Trace[\omega](h_0
\omega(h_0)\omega^2(h_1)) +
\Norm[\omega](h_1)(\delta^3x^3+\theta^2) \label{f1}
\\
\theta \Norm[\omega](h_2) & = & \Norm[\omega](h_0) +
\Trace[\omega](h_1
\omega(h_1)\omega^2(h_0))(\delta^3x^3+\theta^2), \label{f2}
\end{eqnarray}
where $\omega$ is the automorphism defined above. Then
$h_0,h_1,h_2 = 0$.
\end{lem}
\begin{proof}
Since for every $f \in k[\theta,x,d_0]$ we have that $\omega(f)
\equiv f \pmod{x}$ and $\Norm_{\omega}(f) \equiv f^3$, reduction
modulo $x$ gives
\begin{eqnarray}
\bar{h}_2^3 & = & 3 \bar{h}_0^2 \bar{h}_1 + (d_0^3+\theta^2)
\bar{h}_1^3 \label{ff1}
\\
\theta \bar{h}_2^3 & = & \bar{h}_0^3+  3 (d_0^3+\theta^2)
\bar{h}_0 \bar{h}_1^2 \label{ff2}
\end{eqnarray}
for the residues $\bar{h}_0,\bar{h}_1,\bar{h}_2 \in
k[\theta,d_0]$. But then $(\bar{h}_0 - \theta \bar{h}_1)^3 =
\bar{h}_0^3 - 3 \theta \bar{h}_1 \bar{h}_0^2 + 3 \theta^2
\bar{h}_0 \bar{h}_1^2 - \theta^3 \bar{h}_1^3 = (\theta \bar{h}_2^3
- 3 d_0^3 \bar{h}_0 \bar{h}_1^2) - (\theta \bar{h}_2^3 - \theta
d_0^3 \bar{h}_1^3)= (\theta \bar{h}_1 - 3 \bar{h}_0)d_0^3
\bar{h}_1^2 $, namely \begin{equation}\label{make} (\bar{h}_0 -
\theta \bar{h}_1)^3 = (\theta \bar{h}_1 - 3 \bar{h}_0)d_0^3
\bar{h}_1^2.
\end{equation}
This implies $\bar{h}_0 \equiv \theta {h}_1 \pmod{d_0}$, so we can
write $\bar{h}_0 = \theta \bar{h}_1 + d_0 \bar{h}_0'$ for $h_0'
\in k[\theta,d_0]$. Plugging this back in \eq{make} and dividing
by $d_0^3$, we get
$$\bar{h}_0'^3  +  3 d_0 \bar{h}_0' \bar{h}_1^2 + 2
\theta \bar{h}_1^3 = 0,$$ which by Lemma \ref{lemA} implies that $h_0'
= 0$. Thus $\bar{h}_0 = \theta \bar{h}_1$, and \eq{ff1} gives
$$\bar{h}_2^3 = (d_0^3+\theta^2 + 3\theta)\bar{h}_1^3.
$$
But $d_0^3 + \theta^2 + 3\theta$ is not a cube in $k(\theta,d_0)$,
so necessarily $\bar{h}_1 = \bar{h}_2 = 0$, implying $\bar{h}_0 =
0$ as well. This proves $h_0,h_1,h_2$ are all divisible by $x$, so
replacing each $h_\ell$ by $x^{-1}h_\ell$ we get a solution of
smaller degree to \eqs{f1}{f2}, ad infinitum.
\end{proof}

\begin{lem}\label{lemA}
If $f, g \in k[\theta,d_0]$ satisfy
$$f^3 + 3 d_0 f g^2 + 2 \theta g^3 = 0$$
then $f = 0$.
\end{lem}
\begin{proof}
Otherwise $f^{-1}g$ is a root of $\lambda^3 + 3 d_0 \lambda + 2 \theta$,
which is generic over $k(\theta,d_0)$ since $\mychar k \neq 2,3$,
and thus irreducible.
\end{proof}

\chapter{Chain Lemmas}
\section{Background}
It was proven by Merkurjev and Suslin in \cite{MS} that the group $_d Br(F)$ (the $d$-torsion of $Br(F)$) is generated by cyclic algebras of degree $d$ for any integer $d$, if $F$ is a field of characteristic prime to $d$ containing a primitive $d$th root of unity.

It had been proven earlier by Albert in \cite{Albert} that $_p Br(F)$ is generated by cyclic algebras of degree $p$ if $F$ is a field of characteristic $p$.

The word problem for $_d Br(F)$ can be phrased in terms of tensor products of cyclic algebras:
Are two given two tensor products of cyclic algebras Brauer equivalent?

Given two tensor products of cyclic algebras, if they are not of the same length then one can add a matrix algebra to the shorter side and make them have the same length. In this case, being Brauer equivalent is the same as being isomorphic as central simple $F$-algebras.

Two different symbols might present the same algebra. For example, the real quaternion algebra is presented by both $(-1,-1)_{2,\mathbb{R}}$ and $(-1,-2)_{2,\mathbb{R}}$.

Of course, two different tensor products of cyclic algebras might present the same algebra even if the multiplicands are not pair-wise isomorphic. For example, $(-1,-1)_{2,\mathbb{R}} \otimes (-1,-1)_{2,\mathbb{R}}=(1,1)_{2,\mathbb{R}} \otimes (1,1)_{2,\mathbb{R}}$, even though $(-1,-1)_{2,\mathbb{R}}$ is a division algebra while $(1,1)_{2,\mathbb{R}}$ is a matrix algebra (and in particular they are not isomorphic).

The idea of the chain lemma is to come up with a set of basic steps with which one can produce of the different symbol presentations of an algebra as tensor products of cyclic algebras, given one symbol presentations to start with.

So far there are known chain lemmas for quaternion algebras, biquaternion algebras (see Section \ref{Biquaternion}) and cyclic algebras of degree $3$.
The latter was proven in case of characteristic not $3$ by Rost in \cite{Rost} and in case of characteristic $3$ by Vishne in \cite{Vishne}. In \cite{HKT}, Haile Kuo and Tignol provided alternative proofs for Rost's result, using composition algebras.

\section{Chains of $p$-Central Elements in $p$-Cyclic Algebras}

Throughout this section, let $p$ be a given prime and $F$ be an infinite field.

We consider always one of the two cases:
\begin{enumerate}
\item\label{Case1} $\charac{F}=p$.
\item\label{Case2} $\charac{F} \neq p$ and $F$ contains a primitive $p$th root of unity $\rho$.
\end{enumerate}

In this section, we focus mainly on Case \ref{Case2} (except for Subsection \ref{char3}). Instead of asking whether two symbol presentations of the same cyclic algebra are connected by a chain of basic steps, we ask whether two given $p$-central elements $x$ and $z$ are connected by a chain of $p$-central elements $x=x_1,x_2,\dots,x_n=z$ such that $x_k x_{k+1}=\rho^{d_k} x_{k+1} x_k$ for any $1 \leq k \leq n-1$ and a set of integers $\set{d_k}$.

\subsection*{The technique}

Let $A$ be a cyclic algebra of degree $p$ over $F$ in Case \ref{Case2}.

If $x \in A$ is $p$-central, then one can
decompose $A$ under the conjugation action of $x$ into $A = \sum_{i=0}^{p-1} A_i$, where $A_i$ is the eigenspace for the eigenvalue $\rho^i$.
In particular, if we have two $p$-central elements $x$ and $z$, then $z=z_0+\dots+z_{p-1}$ such that $z_i x=\rho^i x z_i$ and $x=x_0+\dots+x_{p-1}$ where $x_i z=\rho^i z x_i$.
The indices can be considered to be elements in $\mathbb{Z}/p \mathbb{Z}$.

We recall that a $p$-central space $F x+F z$ is short of type $\set{i,j}$ if $z=z_i+z_j$.

Let $\CX$ be the set of $p$-central elements, and let $E=\CX \times \CX$.
The pair $(\CX,E)$ forms a complete directed graph.
We label each edge $(x,z) \in E$ with the set $\set{i \in \mathbb{Z}/p \mathbb{Z} : z_i \neq 0}$ and denote it by $l(x,z)$.
The weight of each edge is then $\sharp l(x,z)$ and is denoted by $w(x,z)$.

When drawing an edge, we can either write the label explicitly $\xymatrix@C=20px@R=20px{x \ar@{->}[rr]^{\set{a_1,\dots,a_k}} &  & z}$, or simply mention its weight $\xymatrix@C=20px@R=20px{x \ar@{->}[r]^{k}   & z}$. If we want to describe the label of each direction, then we shall write the label of the direction from left to right above the edge and the label of the direction from right to left bellow the edge $\xymatrix@C=20px@R=20px{x \ar@{<->}[rr]^{\set{a_1,\dots,a_k}}_{\set{b_1,\dots,b_m}} &  & z}$. The same goes for weights $\xymatrix@C=20px@R=20px{x \ar@{<->}[r]^{k}_{m}  & z}$.
When $w(x,z)=1$ then also $w(z,x)=1$, and therefore we shall simply draw $\xymatrix@C=20px@R=20px{x \ar@{<->}[r]  & z}$.

\begin{rem}
It is not true in general that $w(x,z)=w(z,x)$, except for the trivial case of $w(x,z)=1$.
For example, if $p=3$ and $A=(\alpha,\beta)_{3,F}=F[x,y : x^3=\alpha, y^3=\beta, y x=\rho x y]$ then for $z=y+x^2 y^2$ we have
$x=(-\rho \beta \alpha-\rho^2 \alpha^{-1}) (z-x^2 y z-(\rho^2 \alpha^2 \beta)^{-1} (x^2 y)^2 z)$.
In this case, $x_0=(-\rho \beta \alpha-\rho^2 \alpha^{-1}) z$, $x_1=(-\rho \beta \alpha-\rho^2 \alpha^{-1}) (-x^2 y)$ and $x_2=(-\rho \beta \alpha-\rho^2 \alpha^{-1}) (-(\rho^2 \alpha^2 \beta)^{-1} (x^2 y)^2 z)$.
Consequently, $w(x,z)=2 \neq 3=w(z,x)$.
\end{rem}

\begin{defn}
We call a chain of edges of weight $1$ a ``Rost chain".
\end{defn}

Subsection \ref{generalcase} focuses on edges of weight $2$ in both directions. The main result in that section is

{\bf Theorem~\ref{twotwo}}
If $\xymatrix@C=20px@R=20px{x \ar@{<->}[r]^{2}_{2}  & z}$ then there exists a Rost chain connecting $x$ and $z$ of length $2$.

\smallskip

Subsection \ref{p3} is dedicated to showing how the chain lemma for the case of $p=3$ is obtained as a result of more general statements that hold for any arbitrary prime $p$.

Subsection \ref{char3} provides a better upper bound for the distance between two Artin-Schreier elements than what appears in \cite{MV2}.

Subsection \ref{p5} deals with the special case of $p=5$. The main result in that section is

{\bf Theorem~\ref{twonozero}}
If $\xymatrix@C=20px@R=20px{x \ar@{->}[r]^{2}  & z}$, and $0 \not \in l(z,x)$ then there is a Rost chain connecting $x$ and $z$.

\smallskip

\begin{defn}
The commutator $[x,z]_d$ has a different meaning in each case.

In Case \ref{Case1}, we write $[x,z]=[x,z]_1=z x-x z$ and define $[x,z]_k$ inductively as $[x,z]_{k-1} x-x [x,z]_{k-1}$. $[x,z]_0$ is defined to be $z$.

In Case \ref{Case2}, $[x,z]_d=z x-\rho^d x z$.
We define inductively inductively: $$[x,\dots,x,x,z]_{d_1,d_2,\dots,d_k}=[x,\dots,x,[x,z]_{d_1}]_{d_2,\dots,d_k}.$$
\end{defn}

\begin{rem}
In Case \ref{Case2}, $[x,\dots,x,z]_{d_1,d_2,\dots,d_k}=0$ if and only if $l(x,z) \subseteq \set{d_1,d_2,\dots,d_k}$
\end{rem}

\begin{proof}
Straightforward calculation shows that the part of $[x,\dots,x,z]_{d_1,d_2,\dots,d_k}$ which acts on $x$ with $\rho^i$ is $(\rho^i-\rho^{d_1}) \cdot \dots \cdot (\rho^i-\rho^{d_k}) x^k z_i$.
This is equal to zero if and only if $i \in \set{d_1,d_2,\dots,d_k}$.
\end{proof}


\subsection{Edge of weight $2$ in both directions}\label{generalcase}

In this section we study some basic properties of the graph $(\CX,E)$ in attempt to prove that if $w(x,z)=w(z,x)=2$ then $x$ and $z$ are connected by a Rost chain of length $2$.

\begin{prop}\label{weight22}
If $\xymatrix@C=20px@R=20px{x \ar@{<->}[r]^{2}_{2}  & z}$ then
\begin{enumerate}
\item $l(z,x)=-l(x,z)$, i.e. $x=x_i+x_j$ and $z=z_{-i}+z_{-j}$ for some $i \neq j$.
\item $x_i x_j=\rho^{j-i} x_j x_i$ and $z_{-i} z_{-j}=\rho^{i-j} z_{-j} z_{-i}$.
\end{enumerate}
\end{prop}

\begin{proof}
We have $x=x_i+x_j$ and $z=z_m+z_n$ such that $z_m x=\rho^m x z_m$, $z_n x=\rho^n x z_n$, $x_i z=\rho^i z x_i$ and $x_j z=\rho^j z x_j$.

Let us consider the equality $[x,x,z]_{m,n}=0$. This holds because $z=z_m+z_n$.
On the other hand, if we substitute $x=x_i+x_j$ in this expression we get the following set of equations (due to conjugation by $z$):
\begin{enumerate}
\item $(z x_i-\rho^m x_i z) x_i-\rho^n x_i (z x_i-\rho^m x_i z)=0$
\item $(z x_j-\rho^m x_j z) x_j-\rho^n x_j (z x_j-\rho^m x_j z)=0$
\item $(z x_j-\rho^m x_j z) x_i-\rho^n x_i (z x_j-\rho^m x_j z)+(z x_i-\rho^m x_i z) x_j-\rho^n x_j (z x_i-\rho^m x_i z)=0$
\end{enumerate}

From the first equation we obtain $(\rho^{-i}-\rho^m)(\rho^{-i}-\rho^n) x_i^2 z=0$ and from the second equation we obtain $(\rho^{-j}-\rho^m)(\rho^{-j}-\rho^n) x_i^2 z=0$. Henceforth, without loss of generality $m=-i$ and $n=-j$.

From the third equation we obtain $(\rho^{-j}-\rho^{-i}) \rho^{-i} x_j x_i z-(\rho^{-j}-\rho^{-i}) \rho^{-j} x_i x_j z=0$. Consequently, $x_i x_j=\rho^{j-i} x_j x_i$.
Due to symmetry, we also have $z_{-i} z_{-j}=\rho^{i-j} z_{-j} z_{-i}$.
\end{proof}

\begin{cor}
For any $x,z \in \CX$, $\xymatrix@C=20px@R=20px{x \ar@{<->}[r]^{2}_{2}  & z}$ if and only if $z \in F y^i+F y^j x^{(j-i) i^{-1}}$ for some $y \in \CX$ satisfying $y x=\rho x y$.
\end{cor}

\begin{proof}
If $\xymatrix@C=20px@R=20px{x \ar@{<->}[r]^{2}_{2}  & z}$ then  $z \in F y^i+F y^j x^{(j-i) i^{-1}}$ according to Proposition \ref{weight22}.

In order to prove the opposite direction it is enough to check what happens if $i=1$, i.e. check wether $w(z,x)=2$ if $z=y+y^d x^{d-1}$ for some $y \in \CX$ satisfying $y x=\rho x y$.
This is true, because $[x,x,z]_{d,1}=0$.
\end{proof}

\begin{rem}
If $0 \neq k \in l(x,z)$ then $z_k \in \CX$.
However, if $0 \in l(x,z)$ then $z_0$ is not necessarily in $\CX$. For example: If $p=3$ and $A=F[x,y : x^3=\alpha, y^3=\beta, y x=\rho x y]$ then $z=x+x^2+x y-\frac{1}{\beta}(x y)^2$ satisfies $\tr(z)=\tr(z^2)=0$, and so $z^3 \in F$, i.e. $z \in \CX$, even though $z_0=x+x^2 \not \in \CX$.
\end{rem}

\begin{prop} \label{split2cent}
If $\xymatrix@C=20px@R=20px{x \ar@{->}[rr]^{\set{i,j}} &  & z}$, then $z_i,z_j \in \CX$.
\end{prop}

\begin{proof}
We have $F \ni z^p=(z_i+z_j)^p=\sum_{k=0}^p z_i^k * z_j^{p-k}$. There is a unique decomposition into eigenvectors with respect to conjugation by $x$. Since the left-hand side of the equality commutes with $x$, it must be equal to the part of the right-hand side of this equality which commutes with $x$. Therefore $z^p=z_i^p+z_j^p$.
If $i \neq 0$ then $F[x,z_i]$ generates a subalgebra of $A$ whose center is $F[z_i^p]$. However, this subalgebra is noncommutative, and as a cyclic division algebra of prime degree, $A$ has no nontrivial noncommutative subalgebras, which means that $A=F[x,z_i]$ and in particular, $z_i^p \in F$. Similarly, if $j \neq 0$ then $z_j^p \in F$.
If $i=0$ then $j \neq 0$ and so $z_j \in \CX$, and consequently $z_i^p=z^p-z_j^p \in F$,
which means that $z_i \in \CX$.
\end{proof}

\begin{cor}\label{pcent}
If $\xymatrix@C=20px@R=20px{x \ar@{->}[rr]^{\set{i,j}} &  & z}$ then $F z_i+F z_j$ is a $p$-central space. Moreover, its exponentiation form $f(u,v)=(u z_i+v z_j)^p$ is diagonal, i.e. $f(u,z)=u^p z_i^p+v^p z_j^p$.
\end{cor}

\begin{proof}
We have $z^p=(z_i+z_j)^p$. Due to conjugation by $x$ we obtain the relations $z_i^k * z_j^{p-k}=0$ for all $1 \leq k \leq p-1$.
Together with the result from Proposition \ref{split2cent}, the space $F z_i+F z_j$ is therefore $p$-central and it is easy to see why the exponentiation form is diagonal.
\end{proof}

\begin{cor}\label{nozero}
If $\xymatrix@C=20px@R=20px{x \ar@{->}[rr]^{\set{i,j}} &  & z}$ then $0 \not \in l(z_i,z_j)$, and in particular $w(z_i,z_j) \leq p-1$
\end{cor}

\begin{proof}
The exponentiation form of $F z_i+F z_j$ is diagonal, therefore $z_i^{p-1} * z_j=0$. However, $z_i^{p-1} * z_j=p z_i^{p-1} z_{j,0}$, which means that $z_{j,0}=0$.
\end{proof}

\begin{thm}\label{twotwo}
If $\xymatrix@C=20px@R=20px{x \ar@{<->}[r]^{2}_{2}  & z}$ then there exists a Rost chain connecting $x$ and $z$ of length $2$.
\end{thm}

\begin{proof}
We have $x=x_i+x_j$ and $z=z_{-i}+z_{-j}$ according to Proposition \ref{weight22}. Setting $y=x z-\rho^i z x$ we have $y=(\rho^j-\rho^i) z x_j$. Therefore $y^p=(\rho^j-\rho^i)^p z^p x_j^p$. According to Proposition \ref{split2cent} we know that $x_j \in \CX$, and therefore $y \in \CX$. Since $w(x,y)=w(z,y)=1$ we get the Rost chain $\xymatrix@C=20px@R=20px{x \ar@{<->}[r]  & y \ar@{<->}[r]  & z}$.
\end{proof}

\subsection{Alternative proofs for the Chain Lemma for $p=3$}\label{p3}

In this section we show how the chain lemma for $p=3$ is easily obtained as a result of more general statements that hold for any prime $p$.

\begin{rem} \label{notzero}
If $\xymatrix@C=20px@R=20px{x \ar@{->}[rr]^{\set{i,j}} &  & z}$ and $\xymatrix@C=20px@R=20px{z_i \ar@{<->}[r] & z_j}$ then $\xymatrix@C=20px@R=20px{x \ar@{<->}[r] & z_i z_j^{-1} \ar@{<->}[r] & z}$.
\end{rem}

\begin{cor} \label{notzerocor}
If $\xymatrix@C=20px@R=20px{x \ar@{->}[rr]^{\set{0,j}} &  & z}$ then $\xymatrix@C=20px@R=20px{z_0 \ar@{<->}[r] & z_j}$. Moreover, there exists a Rost chain between $x$ and $z$ of length $2$.
\end{cor}

\begin{proof}
Since $A$ is cyclic of degree $p$, and $z_0 \in \CX$ (according to \ref{split2cent}), $z_0=x^k$ for some $k$. Henceforth $w(z_0,z_j)=1$. As a Result of Remark \ref{notzero}, there is a Rost chain between $x$ and $z$ of length $2$.
\end{proof}

\begin{prop}\label{prop2in2}
If $\xymatrix@C=20px@R=20px{x \ar@{->}[rr]^{\set{i,j}} &  & z}$ and $\xymatrix@C=20px@R=20px{z_i \ar@{->}[rr]^{\set{m,n}} &  & z_j}$ then
$m \not \equiv -n \pmod{p}$.
\end{prop}

\begin{proof}
According to Corollary \ref{notzerocor}, $i \neq 0$ because $w(z_i,z_j)=2$. Therefore $z_{j,m} \in F z_i^{j i^{-1}} x^{m (-i)^{-1}}$ and $z_{j,n} \in F z_i^{j i^{-1}} x^{n (-i)^{-1}}$.

We have $m,n \neq 0$ due to Corollary \ref{nozero}.

On the contrary, let us assume that $m \equiv -n \pmod{p}$.

In Corollary \ref{pcent} we saw that $F z_i+F z_j$ is a $p$-central space with a diagonal exponentiation form. Since $l(z_i,z_j)=\set{m,-m}$, this $p$-central space is short of type $\set{m,-m}$.

In \cite{Chapman} it was proven that if $F r+F s$ is a short $p$-central space of type $\set{t,-t}$ whose exponentiation form is diagonal then $s_t s_{-t}=\rho^t s_{-t} s_t$.

Therefore $z_{j,m} z_{j,n}=\rho^m z_{j,n} z_{j,m}$.
Consequently $\rho^{i^{-2} j (m-n)}=\rho^m$, which means that $2 j \equiv i^2 \pmod{p}$.

But now, $l(x^2,z)=\set{2 i,2 j}$ and therefore for similar arguments $2 (2 j) \equiv (2 i)^2 \pmod{p}$, and that is a contradiction.
\end{proof}

\begin{prop}
If $l(x,z) \subseteq \set{0,i,j}$ and $l(z,x) \subseteq \set{0,-i,k}$ for some $i,j,k \in \mathbb{Z}/p \mathbb{Z}$, and $z_0 \in F x^{-1}+F x^m$ and $x_0 \in F z^{-1}+F z^n$ for some $n,m \in \mathbb{Z}$, then there is a Rost chain of length less or equal to $4$ connecting $x$ and $z$.
\end{prop}

\begin{proof}
From $l(x,z) \subseteq \set{0,i,j}$ we have $z=z_0+z_i+z_j$ (some of them may be equal to zero).
Similarly, $x=x_0+x_{-i}+x_{k}$.
Let $t=\frac{z x-\rho^i x z}{1-\rho^i}-\tr(x z)$.
By substituting the details given above, $t \in F x z_j+F x^{m+1}$.
Consequently, $\xymatrix@C=20px@R=20px{x \ar@{<->}[r] & x^{-m} z_j \ar@{<->}[r]  & t}$.
For similar reasons, $\xymatrix@C=20px@R=20px{t \ar@{<->}[r]  & z^{-n} x_k \ar@{<->}[r]  & z}$.
\end{proof}

\begin{rem}
This Proposition generalizes the algebraic proof of Rost's chain lemma as appears in \cite{HKT}. Simply, for $p=3$ all the conditions appearing in this proposition are automatically satisfied:
We have $l(x,z),l(z,x) \subseteq \set{0,1,2}$, and also $z_0 \in F x+F x^2$ (the coefficient of $x^0=1$ must be zero because $\tr(z)=0$) and $x_0 \in F z+F z^2$ (for a similar reason).
\end{rem}

\begin{prop}
For all $x,z \in \CX$ and $i \in \mathbb{Z}/p \mathbb{Z}$, there exists an element $t \in \CX$ such that there is a Rost chain between $x$ and $t$, and $i \not \in l(z,t)$.
\end{prop}

\begin{proof}
Since $x \in \CX$, there exists an element $y \in \CX$ such that $y x=\rho x y$.
The vector space $F x+F[x] y$ is $p$-central of dimension $p+1$ (over $F$). However, the dimension of $\set{w \in A : z w=\rho^i w z}$ is a vector space of dimension $p$ over $F$. Therefore the equation $c_0 y_i+c_1 (x y)_i+\dots+c_{p-1} (x^{p-1} y)_i+c_p x_i=0$ should have a non-trivial solution, where $y_i,\dots,(x^{p-1} y)_i,x_i$ are the parts of $y,\dots,(x^{p-1} y),x$ respectively which act on $z$ with $\rho^i$. Consequently, there exists a nonzero element $t \in F x+F[x] y$ so that $t_i=0$, i.e. $i \not \in l(z,t)$.
Finally, $l(x,t) \subseteq \set{0,p-1}$, and so according to Corollary \ref{notzerocor} there is a Rost chain between $x$ and $t$.
\end{proof}

\begin{rem}
Rost's chain lemma for $p=3$ is an easy result of this proposition.
If $x,z \in \CX$ then there exists $t \in \CX$ such that $x$ and $t$ are connected by a Rost chain and $2 \not \in l(z,t)$.
This means, however, that $l(z,t) \subseteq \set{0,1}$ and so according to Corollary \ref{notzerocor}, $z$ and $t$ are also connected by a Rost chain.
\end{rem}

\subsection{Cyclic algebras of degree $3$ in Characteristic $3$}\label{char3}

Unlike the other sections, in this section we focus on Case \ref{Case2}, and specifically $p=3$.

In \cite{Vishne}, Vishne proved that given a cyclic algebra of degree $3$ over a field of characteristic $3$, one can move from one symbol presentation of the algebra to another by a series of $7$ steps, such that in each step one entry remains unchanged.

In \cite{MV2}, Matzri and Vishne proved that given two symbol presentations, $[\alpha,\beta)$ and $[\gamma,\delta)$, one can get from $[\alpha,\beta)$ to either $[\gamma,\delta)$ or to $[-\gamma,\delta^2)$ by a series of $5$ steps.

Here we shall prove that one can move from $[\alpha,\beta)$ either to $[\gamma,\delta)$ in five steps or to $[-\gamma,\delta^2)$ in three steps.

Like the eigenvector decomposition of elements with respect to a given $p$-central element in Case \ref{Case2}, we have a similar decomposition of elements with respect to a given Artin-Schreier element in Case \ref{Case1}.

In this context, we write $[x,z]=[x,z]_1=z x-x z$ and define $[x,z]_k$ inductively as $[x,z]_{k-1} x-x [x,z]_{k-1}$. $[x,z]_0$ is defined to be $z$.

\begin{lem}\label{decomcharp}
Given an associative algebra $A$ over a field $F$ of characteristic $p$,
if $x$ is Artin-Schreier then for any $z \in A$, $z=z_0+z_1+\dots+z_{p-1}$ where $[x,z_k]=k z_k$.
\end{lem}

\begin{proof}
Let $z_0=z-[x,z]_{p-1}$, and for all $1 \leq k \leq p-1$, $z_k=-(k^{p-1} [x,z]_1+\dots+k [x,z]_{p-2}+[z,x]_{p-1})$.
It is an easy calculation to prove that $[x,z]=k z_k$. It is obvious that $z_0+z_1+\dots+z_{p-1}=z$.
\end{proof}

Let $A$ be a cyclic algebra of degree $3$ over $F$ of characteristic $3$.

\begin{thm}
If $x$ and $z$ are Artin-Schreier then
\begin{enumerate}
\item If one of the elements $z_1,z_2,x_1,x_2$ is zero then there exists some $3$-central element $q$ such that $x q+q x=q$ and either $z q+q z=q$ or $z q+q z=-q$.
\item If $z_1,z_2,x_1,x_2 \neq 0$ then there exists some Artin-Schreier element $t$ and some $3$-central elements $q,r$ such that $x q-q x=q$, $t q-q t=q$, $t r-r t=r$ and $z r-r z=r$.
\end{enumerate}
\end{thm}

\begin{proof}
According to Lemma \ref{decomcharp}, $z=z_0+z_1+z_2$ with respect to $x$ and $x=x_0+x_1+x_2$ with respect to $z$.
Now, $z_0=a+b x+c x^2$ for some $a,b,c \in F$.
However, $0=\tr(z)=\tr(z_0)=c$.
Furthermore, $b=\tr(x z)$.
Similarly, $x_0=d+b z$.

If $z_2=0$ then $z=a+b x+z_1$.
Since $z$ is Artin-Schreier, by the fact that $z^3-z \in F$ we conclude that $b$ is either $1$ or $-1$.

If $z_1 \neq 0$ then take $q=z_1$. Otherwise, take $q=y$ where $y$ satisfies $x y+y x=y$.

Assume $z_1,z_2,x_1,x_2 \neq 0$.
Let $w=x z-z x+x+z$.
On one hand $w=z_0-z_1+x=a+(b+1) x-z_1$.
On the other $w=x_0-x_2+z=d+(b+1) z-x_2$.

The element $w$ is Artin-Schreier if and only if $b \neq -1$. In this case, we will take $q=z_1$, $t=(b+1)^{-1} w$, $r=x_1$.
\end{proof}

\subsection{The special case of $p=5$}\label{p5}

Back to Case \ref{Case2}.
The final goal of this section is to prove the following sufficient condition for the existence of a Rost chain connecting two elements for the case of $p=5$:  $w(x,z)=2$ and $0 \not \in l(z,x)$.

\begin{lem} \label{lem3}
If $\xymatrix@C=20px@R=20px{x \ar@{->}[rr]^{\set{i,j,k}} & & z}$ then $(2 i \equiv j+k \pmod{p}) \wedge (2 j \equiv i+k \pmod{p})$ if and only if $p=3$.
\end{lem}

\begin{proof}
If $2 i \equiv j+k \pmod{p}$ and $2 j \equiv i+k \pmod{p}$ then $3 (i-j) \equiv 0 \pmod{p}$, but $i-j \neq 0$ because $w(x,z)=3$, and therefore $3 \equiv 0 \pmod{3}$, which means that $p=3$.

The other direction is trivial.
\end{proof}

\begin{thm}\label{nothree}
For $p>3$, if $\xymatrix@C=20px@R=20px{x \ar@{<->}[r]^{2}_{d}  & z}$ then $d \neq 3$.
\end{thm}

\begin{proof}
Assume to the contrary, that $l(z,x)=3$.
Then $x=x_i+x_j+x_k$ and $z=z_m+z_n$ for some $m,n,i,j,k \in \mathbb{Z}/p\mathbb{Z}$.
Since $p>3$, at least two of the following hold: $i+k \not \equiv 2 j \pmod{p}$, $j+k \not \equiv 2 i \pmod{p}$, $i+j \not \equiv 2 k \pmod{p}$, because otherwise $p=3$ according to Lemma \ref{lem3}.
Without loss of generality we may assume that $i+k \not \equiv 2 j \pmod{p}$ and $j+k \not \equiv 2 i \pmod{p}$.

Let us look at the equation $[x,x,z]_{m,n}=0$.
If we take only the part which $\rho^{2 i}$-commutes with $z$ we get $(\rho^{-i}-\rho^m)(\rho^{-i}-\rho^n) x_i^2 z=0$, and if we take only the part which $\rho^{2 j}$-commutes with $z$ we get $(\rho^{-j}-\rho^m)(\rho^{-j}-\rho^n) x_j^2 z=0$.
Consequently, $m=-i$ and $n=-j$ without loss of generality.
If we take only the part which $\rho^{i+k}$-commutes with $z$ we obtain $(\rho^{-k}-\rho^{-i}) \rho^{-i} x_k x_i z-(\rho^{-k}-\rho^{-i}) \rho^{-k} x_i x_k z=0$. Consequently, $x_i x_k=\rho^{k-i} x_k x_i$. Similarly $x_j x_k=\rho^{k-j} x_k x_j$.

Now let us look at the equality $[z,z,x]_{i,j}=(\rho^k-\rho^i) (\rho^k-\rho i) z^2 x_k$.
If we substitute $z=z_{-i}+z_{-j}$ on the left-hand side of this equation then we get $(\rho^j-\rho^i) \rho^i z_{-j} z_{-i}-(\rho^j-\rho^i) \rho^j z_{-j} z_{-i}=(\rho^k-\rho^i) (\rho^k-\rho i) x_k$.

Consequently, $x_k$ $\rho^{-i-j}$-commutes with $x$. If $k=0$ it means that $x_k$ commutes with $x$ and then $x_i=x_j=0$.
Otherwise, it means that the part of $x$ which commutes with $x_k$ is equal to zero, but this part is also equal to $x_k$, hence $x_k=0$.
At any rate, we have a contradiction.
\end{proof}

\begin{prop}
For $p=5$ if $\xymatrix@C=20px@R=20px{x \ar@{->}[rr]^{\set{i,j}} &  & z}$ and $\xymatrix@C=20px@R=20px{z_i \ar@{->}[r]^{2}  & z_j}$ then there exists a Rost chain of length $3$ connecting $x$ and $z$.
\end{prop}

\begin{proof}
Let $l(x,z)=\set{i,j}$ and $l(z_i,z_j)=\set{m,n}$.
Without loss of generality we can assume that $m=1$ and $n=3$ or $n=4$.
According to Proposition \ref{prop2in2}, the case of $n=4$ is not possible, and so we assume that $n=3$.

The space $F z_i+F z_j$ is a short $5$-central of type $\set{1,3}$ and his exponentiation form is diagonal.
In \cite{Chapman} it is proven that in this case either $z_{j,1} z_{j,3}=\rho z_{j,3} z_{j,1}$ or $z_{j,1} z_{j,3}=\rho^2 z_{j,3} z_{j,1}$.

In the case the Rost chain
$\xymatrix@C=20px@R=20px{z \ar@{<->}[r]  & z_{j,1}^{-1} (z_i+z_{j,3}) \ar@{<->}[r]  & z_{j,1} \ar@{<->}[r]  & x}$. In the second case the Rost chain is
$\xymatrix@C=20px@R=20px{z \ar@{<->}[r]  & z_{j,3}^{-1} (z_i+z_{j,1}) \ar@{<->}[r]  & z_{j,3} \ar@{<->}[r]  & x}$.
\end{proof}

\begin{thm}
For $p=5$, if $\xymatrix@C=20px@R=20px{x \ar@{->}[rr]^{\set{i,j}} &  & z}$ then $w(z_i,z_j) \neq 3$.
\end{thm}

\begin{proof}
If $l(z_i,z_j)=\set{m,n,k}$ then without loss of generality $m \equiv -n \pmod{5}$ and $m,n \not \equiv -k \pmod{5}$.

The space $F z_i+F z_j$ is a $p$-central with a diagonal exponentiation form according to Corollary \ref{pcent}.
From the relation $(z_i)^3 * (z_j)^2=0$ we get $z_{j,m} z_{j,n}=\rho^m z_{j,n} z_{j,m}$. But again, as in the proof of Proposition \ref{prop2in2}, it means that $2 j \equiv i^2 \pmod{p}$. As before, we shall have a contradiction, because we get $2 (2 j) \equiv (2 i)^2 \pmod{p}$ as well.
\end{proof}

\begin{rem}
There are however cases where $\xymatrix@C=20px@R=20px{x \ar@{->}[rr]^{\set{i,j}} &  & z}$ and $\xymatrix@C=20px@R=20px{z_i \ar@{->}[r]^{2}  & z_j}$ or $\xymatrix@C=20px@R=20px{z_i \ar@{->}[r]^{4}  & z_j}$. In particular, if $A=F[x,y : x^5=\alpha, y^5=\beta, y x=\rho x y]$ and $z=y+(a_1 x+a_2 x^2+a_3 x^3+a_4 x^4) y^{-1}$ then $z$ is $p$-central if and only if $a_2 a_3=(\rho^4-\rho) a_1 a_4$. Consequently, if we take $a_3=a_4=0$ then $w(x,z)=2$ and $w(z_1,z_{-1})=2$, and if we take $a_1=a_4=a_3=1$ and $a_2=(\rho^4-\rho)$ then $w(x,z)=2$ while $w(z_1,z_{-1})=4$.
The case of $\xymatrix@C=20px@R=20px{z_i \ar@{->}[r]^{5}  & z_j}$ is not possible, due to Corollary \ref{nozero}.
\end{rem}

\begin{proof}
In order to prove the statement ``$z=y+(a_1 x+a_2 x^2+a_3 x^3+a_4 x^4) y^{-1}$ then $z$ is $p$-central if and only if $a_2 a_3=(\rho^4-\rho) a_1 a_4$" one should turn to the relations $z_1 * z_4^4=z_1^2 * z_4^3=z_1^3 * z_4^2=z_1^4 * z_4=0$ ($z_1=y$ and $z_4=a_0+a_1 x+a_2 x^2+a_3 x^3+a_4 x^4) y^{-1}$). On one hand, these relations are satisfied if and only if $z \in \CX$ (Corollary \ref{pcent}). On the other hand, it can be checked that these relations are satisfied if and only if $a_0=0$ and $a_2 a_3=(\rho^4-\rho) a_1 a_4$:

Due to the relation $y_1^4 * y_4=0$ we have $a_0=0$. Write $w_i=a_i x^i y_1^{-1}$.

Now, the relation $y_1^3 * y_4^2=0$ provides the following due to conjugation by $y_1$:
\begin{enumerate}
\item $y_1^3 * w_3^2+y_1^3 * w_2 * w_4=0$
\item $y_1^3 * w_1^2+y_1^3 * w_3 * w_4=0$
\item $y_1^3 * w_4^2+y_1^3 * w_1 * w_2=0$
\item $y_1^3 * w_2^2+y_1^3 * w_1 * w_3=0$
\item $y_1^3 * w_1 * w_4+y_1^3 * w_2 * w_3=0$
\end{enumerate}

The first four relations are trivial.
From the fifth we obtain $5 (\rho+1+\rho^{-1}) a_1 a_4 \alpha y_1+5 (\rho^3+\rho^2+1) a_2 a_3 \alpha y_1=0$.
Consequently, $a_2 a_3=(\rho^4-\rho) a_1 a_4$.

The relation  $y_1^2 * y_4^3$ provides the following due to conjugation by $y_1$:
\begin{enumerate}
\item $y_1^2 * w_1 * w_2^2+y_1^2 * w_1^2 * w_3+y_1^2 * w_2 * w_4^2+y_1^2 * w_3^2 * w_4=0$
\item $y_1^2 * w_1 * w_2 * w_3+y_1^2 * w_1^2 * w_4+y_1^2 * w_2^3+y_1^2 * w_3 * w_4^2=0$
\item $y_1^2 * w_1 * w_2 * w_4+y_1^2 * w_2^2 * w_3+y_1^2 * w_4^3+y_1^2 * w_1 w_3^2=0$
\item $y_1^2 * w_1^3+y_1^2 * w_1 * w_3 * w_4+ y_1^2 * w_2 * w_3^2+y_1^2 * w_2^2 * w_4=0$
\item $y_1^2 * w_3^3+y_1^2 * w_2 * w_3 * w_4+y_1^2 * w_1 * w_4^2+y_1^2 * w_1^2 * w_2=0$
\end{enumerate}

The first relation is trivial.
The second relation implies that $5 (\rho^3+\rho^2+1) a_1 a_2 a_3 \alpha x y^{-1}+5 (\rho+1+\rho^{-1}) a_1^2 a_4 \alpha x y_1^{-1}=0$. This is automatically satisfied given $a_2 a_3=(\rho^4-\rho) a_1 a_4$. The same happens with the succeeding relations.

The relation  $y_1 * y_4^4$ provides the following due to conjugation by $y_1$:
\begin{enumerate}
\item $y_1 * w_1 * w_2 * w_3 * w_4+y_1 * w_1^2 * w_4^2+y_1 * w_2^2 * w_3^2+y_1 * w_2^3 * w_4+y_1 * w_3 * w_4^3+y_1 * w_1 * w_3^3+y_1 * w_1^3 * w_2=0$
\item $y_1 * w_1^3 * w_3+y_1 * w_1^2 * w_2^2+y_1 * w_1 * w_2 * w_4^2+y_1 * w_1 * w_3^2 * w_4+y_1 * w_2^2 * w_3 * w_4+y_1 * w_2 * w_3^3+y_1 * w_4^4=0$
\item $y_1 * w_2^3 * w_1+y_1 * w_2^2 * w_4^2+y_1 * w_2 * w_4 * w_3^2+y_1 * w_2 * w_1^2 * w_3+y_1 * w_4^2 * w_1 * w_3+y_1 * w_4 * w_1^3+y_1 * w_3^4=0$
\item $y_1 * w_3^3 * w_4+y_1 * w_3^2 * w_1^2+y_1 * w_3 * w_1 * w_2^2+y_1 * w_3 * w_4^2 * w_2+y_1 * w_1^2 * w_4 * w_2+y_1 * w_1 * w_4^3+y_1 * w_2^4=0$
\item $y_1 * w_4^3 * w_2+y_1 * w_4^2 * w_3^2+y_1 * w_4 * w_3 * w_1^2+y_1 * w_4 * w_2^2 * w_1+y_1 * w_3^2 * w_2 * w_1+y_1 * w_3 * w_2^3+y_1 * w_1^4=0$
\end{enumerate}
All these relations are trivial.
\end{proof}

\begin{thm}\label{twonozero}
If $\xymatrix@C=20px@R=20px{x \ar@{->}[r]^{2}  & z}$, and $0 \not \in l(z,x)$ then there is a Rost chain connecting $x$ and $z$.
\end{thm}

\begin{proof}
The case of $0 \in l(x,z)$ has already been dealt with (Corollary \ref{notzerocor}). The same goes for $w(z,x)=2$ (Theorem \ref{twotwo}).

Since $w(x,z)=2$, $w(z,x) \neq 3$ (as in Theorem \ref{nothree}).

Let us assume that $0 \not \in l(x,z)$ and $w(z,x)=4$.
Consequently, $l(x,z) \subseteq l(z,x)$.

There are two distinct cases: $l(x,z)=\set{1,4}$ and $l(x,z)=\set{1,3}$.

Assume $l(x,z)=\set{1,4}$. By taking the part of equality $[x,x,z]_{4,1}=0$ which $\rho^3$-commutes with $z$ we obtain $(z x_2-\rho^4 x_2 z) x_1-\rho x_1 (z x_2-\rho^4 x_2 z)=0$. Consequently $(\rho^3-\rho^4) \rho^4 x_2 x_1 z-(\rho^3-\rho^4) \rho x_1 x_2 z=0$, which means that $x_1 x_2=\rho^3 x_2 x_1$.
Therefore $x_2=a x_1^2 z^3$ for some $a \in F$.

Now, by taking the part of the equality $[x,x,z]_{4,1}=0$ which $\rho^4$-commutes with $z$ we obtain $(z x_3-\rho^4 x_3 z) x_1-\rho x_1 (z x_3-\rho^4 x_3 z)+(z x_2-\rho^4 x_2 z) x_2-\rho x_2 (z x_2-\rho^4 x_2 z)=0$. Hence $x_3=\frac{(\rho^3-\rho^4) (\rho^3-\rho))}{(\rho^2-\rho^4) (\rho^3-\rho))} a \rho z^5 x_1^3 z+b x_1^3 z^3$ for some $b \in F$, i.e. $x_3=(-\rho^3-1) a z^5 x_1^3 z+b x_1^3 z^3$.

By taking the part of the equality $[x,x,z]_{1,4}=0$ which $\rho^2$-commutes with $z$ we obtain $(z x_3-\rho x_3 z) x_4-\rho^4 x_4 (z x_3-\rho x_3 z)=0$. Consequently $(\rho^2-\rho) \rho x_3 x_4 z-(\rho^2-\rho) \rho^4 x_4 x_3 z=0$, which means that $x_3 x_4=\rho^3 x_4 x_3$. Therefore $x_4=c x_3^3 z$ for some $c \in F$.

By taking the part of the equality $[x,x,z]_{1,4}=0$ which $\rho$-commutes with $z$ we obtain $(z x_2-\rho x_2 z) x_4-\rho^4 x_4 (z x_2-\rho x_2 z)+(z x_3-\rho x_3 z) x_3-\rho^4 x_3 (z x_3-\rho x_3 z)=0$. Now, by taking the projection on the line $F x_1 z^4$, we get that $b=0$.

Henceforth $x \in F[x_1 z^3] x_1$, which means that
$\xymatrix@C=20px@R=20px{x \ar@{<->}[r]  & x_1 z^3 \ar@{<->}[r]  & z}$.

Assume $l(x,z)=\set{1,3}$. Then we have the following equality $(\rho^4-\rho)(\rho^4-\rho^2)(\rho^4-\rho^3) z^3 x_4=[z,z,z,x]_{1,2,3}$.
Substituting $z=z_1+z_3$ in this equality we get that $l(x,x_4)=\set{0,2}$. Consequently there is a Rost chain between $x$ and $x_4$, and because $w(x_4,z)=1$, there is also a Rost chain between $x$ and $z$.
\end{proof}

\section{The Chain Lemma for Biquaternion Algebras}\label{Biquaternion}
For quaternion algebras, regardless of the characteristic, it is known that given two presentations of the same algebra, one could move from one presentation to the other by a chain of up to three steps, such that each step preserves one entry unchanged.

In characteristic not two, if two presentations of the same algebra share a common slot, $(\alpha,\beta)=(\alpha,\beta')$, then it is easy to see that $\beta'=(a^2-b^2 \alpha) \beta$ for some $a,b \in F$.

As a result, all of the different presentations of a given quaternion algebra $(\alpha,\beta)$ can be obtained by a series of steps such that in each step we choose $a,b \in F$ and change the symbol either to $(\alpha,(a^2-b^2 \alpha) \beta)$ or $((a^2-b^2 \beta) \alpha,\beta)$.

In characteristic two, if $[\alpha,\beta)=[\alpha,\beta')$ then it is easy to see that $\beta'=(a^2+a b+b^2 \alpha) \beta$ for some $a,b \in F$. If $[\alpha,\beta)=[\alpha',\beta)$ then $\alpha'=\alpha+a^2+a+b^2 \beta$ for some $a,b \in F$.

As a result, all of the different presentations of a given quaternion algebra $[\alpha,\beta)$ can be obtained by a series of steps such that in each step we choose $a,b \in F$ and change the symbol either to $[\alpha,(a^2+a b+b^2 \alpha) \beta)$ or $[\alpha+a^2+a+b^2 \beta,\beta)$.

The chain lemma for biquaternion algebras in characteristic not two, i.e. algebras of the form $(\alpha,\beta) \otimes (\gamma,\delta)$, was studied recently in \cite{Siv} and \cite{ChapVish}.
In the latter, the chain lemma was studied through quadruples of generators, i.e. a quadruple $(x,y,z,u)$ such that $x^2=\alpha, y^2=\beta, z^2=\gamma, u^2=\delta, x y=-y x, x z=z x, x u=u x, y z=z y, y u=y u, z u=-u z$. The quadruple is divided into two pairs, $(x,y)$ and $(z,u)$. The first one corresponds to the first symbol and the other to the second.

In \cite{ChapVish} the following changes of quadruples of generators are defined:
\begin{itemize}
\item[$\Lambda_3$]: At most three generators are changed. 
\item[$\Lambda_2$]: At most one generator is changed in each pair. 
\item[$\Pi$]: At most one pair is changed.
\item[$\Omega$]: Two generators, one from each pair, are multiplied by the same element from the field generated over $F$ by the product of the two remaining generators.
\item[$\Lambda_1$]: At most one generator is changed. 
\end{itemize}

It was proven in that paper that every two non-commuting square-central elements have a third square-central element commuting with them both.
It is rather easy to prove, using techniques that had been known already to Albert (see \cite{Albert}) that because of this fact, every two different symbol presentations of the same biquaternion algebra are connected by a chain of up to three steps of type $\Lambda_3$ (as opposed to five steps of type $\Lambda_3$ and ten of type $\Pi$ as written in \cite{ChapVish}).
In that paper, it was also proven that each step of type $\Lambda_3$ can be achieved by five steps of type $\Lambda_2$; Each step of type $\Lambda_2$ can be achieved by one step of type $\Omega$ and two of type $\Lambda_1$; A step of type $\Pi$ is known to be achieved by three steps of type $\Lambda_1$.
All in all, one can move from one symbol presentation of the algebra to another by at most $6$ steps of type $\Omega$ and $39$ of type $\Lambda_1$ (as opposed to $10$ and $135$ in \cite{ChapVish}).

In this section we prove a similar chain lemma for biquaternion algebras in case of characteristic $2$.
We study it through quadruples of standard generators.
A quadruple of generators is $(x,y,z,u)$ such that $$x^2+x=\alpha, y^2=\beta, z^2+z=\gamma, u^2=\delta,$$ $$x y+y x=y, x z=z x, x u=u x, y z=z y, y u=u y, z u+u z=u$$ where $[\alpha,\beta) \otimes [\gamma,\delta)$ is the algebra under discussion.
The quadruple consists naturally of two pairs, $(x,y)$ and $(z,u)$.
We are not concerned with the order of the pairs, i.e. $(x,y,z,u)=(z,u,x,y)$.
Of course the order of the elements inside the pair is important, the first element corresponds to a separable field extension of the center and the second corresponds to an inseparable field extension. The first element is Artin-Schreier, and the second element is square-central.

We define the following steps on a quadruple of generators $(x,y,z,u)$:
\begin{itemize}
\item[$\Lambda_3$]: At most three generators are changed.
\item[$\Lambda_2$]: At most one generator is changed in each pair.
\item[$\Pi$]: At most one pair is changed.
\item[$\Omega_s$]: $x$ and $z$ are preserved and $y$ and $u$ are multiplied by $a+b (x+z)$ for some $a,b \in F$
\item[$\Omega_i$]: $y$ and $u$ are preserved and an element of the form $a y u$ is added to $x$ and $z$ for some $a \in F$.
\item[$\Omega_c$] : $y$ and $z$ are preserved and $x$ changes to $x+b y (1+b y)^{-1} z$ and $u$ changes to $(1+b y) u$ for some $b \in F$.
\item[$\Lambda_1$]: At most one generator is changed.
\end{itemize}

We prove that one can move from one quadruple of generators to another by a chain consisting of up to three steps of type $\Lambda_3$.
We prove further that every step of type $\Lambda_3$ can be replaced with up to three steps of type $\Pi$ and two steps of type $\Lambda_2$.
Furthermore, we prove that each step of type $\Lambda$ can be replaced with up to either three steps of type $\Lambda_1$ or two of type $\Lambda_1$ and one of type $\Omega_i$, $\Omega_s$ or $\Omega_c$.
Since $\Pi$ changes only one quaternion algebra, it is known that $\Pi$ can be replaced with up to three steps of type $\Lambda_1$. Consequently, in order to move from one quadruple of generators to another one needs to do up to $45$ steps, where at most $6$ of them are of type $\Omega_i$, $\Omega_s$ or $\Omega_c$ and all the rest are of type $\Lambda_1$.

The basic steps on the quadruples of generators can be easily translated to basic steps on the symbol presentations.

The $\Omega_s$ step changes $[\alpha,\beta) \otimes [\gamma,\delta)$ to $$[\alpha,(a^2+a b+b^2 (\alpha+\gamma)) \beta) \otimes [\gamma,(a^2+a b+b^2 (\alpha+\gamma)) \delta)$$ for some given $a,b \in F$.

The $\Omega_i$ step changes $[\alpha,\beta) \otimes [\gamma,\delta)$ to $$[\alpha+a^2 \beta \delta,\beta) \otimes [\gamma+a^2 \beta \delta,\delta)$$ for some given $a \in F$.

The $\Omega_c$ step changes $[\alpha,\beta) \otimes [\gamma,\delta)$ to $[\alpha+\frac{b^2 \beta \gamma}{1+b^2 \beta},\beta) \otimes [\gamma,\delta(1+b^2 \beta))$ for some $b \in F$.

The $\Lambda_1$ step changes one of the quaternion algebras $[\alpha,\beta)$ to either to $[\alpha,(a^2+a b+b^2 \alpha) \beta)$ or $[\alpha+a^2+a+b^2 \beta,\beta)$ for some $\alpha,\beta \in F$.

Throughout this paper, let $A$ be a fixed biquaternion division algebra over a field $F$ of characteristic two.

\subsection{Decomposition with respect to maximal subfields}

In this section we shall prove that if $A$ contains a maximal subfield, generated either by two Artin-Schreier elements or one Artin-Schreier and one square-central, then it decomposes as the tensor product of two quaternion algebras such that each of the generators is contained in a different quaternion algebra.

These lemmas will be used later on in this paper.

\begin{lem}\label{instep0}
If $x$ and $z$ are commuting Artin-Schreier elements then there exist some square-central elements $u$ and $y$ such that $(x,y,z,u)$ is a quadruple of generators.
\end{lem}

\begin{proof}
If $x$ and $z$ are commuting Artin-Schreier elements then $C_A(F[x])$ is a quaternion algebra containing $z$. This algebra contains some $q$ such that $q^2 \in F[x]$ and $z q+q z=q$.
The involution on $F[x,z]$ satisfying $x^*=x+1$ and $z^*=z$ extends to $A$. In particular, $q^* x=x q^*$, and therefore $q^* \in C_A(F[x])$.
If $q^*=q$ then by taking $u=q$, $u$ is square-central and $z u+u z=u$. Otherwise, we take $u=q+q^*$. In particular $A=A_0 \otimes F[z,u]$. $x$ is in the quaternion subalgebra $A_0$ and therefore there exists some square-central element $y \in A_0$ such that $x y+y x=y$.
\end{proof}

\begin{lem}\label{instep}
If $x$ is Artin-Schreier, $u$ is square-central and $x u=u x$, then there exist some Artin-Schreier element $z$ and some square-central element $y$ such that $(x,y,z,u)$ is a quadruple of generators.
\end{lem}

\begin{proof}
If $x$ is Artin-Schreier and $u$ is a square-central element commuting with $x$ then $C_A(F[x])$ is a quaternion algebra containing $u$. This algebra contains some $q$ such that $q^2+q \in F[x]$ and $q u+u q=u$.
The involution on $F[x,u]$ satisfying $x^*=x+1$ and $u^*=u$ extends to $A$. In particular, $q^* x=x q^*$, and therefore $q^* \in C_A(F[x])$.

For some $\beta \in F$, $u^2=\beta$.
Write $\mu=q (a+b u) q^*$ for some unknown $a,b \in F$.
Since $q+q^*$ is symmetric with respect to $*$ and commutes with $u$, $q+q^*=c+d u$ for some fixed $c,d \in F$.
Obviously $\mu^*=\mu$.
We want $\mu u+u \mu=u$.
It is a straight-forward calculation to see the condition becomes $1=a+a c+b d \beta+(a d+b c) u$.
Consequently, we want the following system to be satisfied:
\begin{eqnarray*}
1 & = & (c+1) a+d \beta b\\
0 & = & d a+c b
\end{eqnarray*}
This system has a solution, unless $c (c+1)=d^2 \beta$.

If $c (c+1) \neq d^2 \beta$ then by taking $z=q (a+b u) q^*$ where $a,b$ is a solution to the system above, $z$ is Artin-Schreier and $z u+u z=u$.

If $c (c+1)=d^2 \beta$ then $(q^*)^2+q^*=(q+c+d u)^2+(q+c+d u)=q^2+c^2+d^2 \beta+d u+q+c+d u=q^2+q$. This means that $q^2+q$ is invariant under $*$, and therefore $q^2+q \in F$. In this case we will take $z=q$.

All in all, one can find an Artin-Schreier element $z$ such that $z u+u z=u$ and $x z=z x$,
which means that $A=A_0 \otimes F[z,u]$. $x$ is in the quaternion subalgebra $A_0$ and therefore there exists some square-central element $y \in A_0$ such that $x y+y x=y$.
\end{proof}

\subsection{A chain consisting of steps of type $\Lambda_3$}

In this section we will show that every two generating quadruples are connected by a chain of up to three steps of type $\Lambda_3$.

\begin{lem}\label{E1lem}
For any two Artin-Schreier elements $x,z$, if they do not commute then the subalgebra $F[x,z]$ is a quaternion algebra, whose center is either $F$ or a quadratic extension of it.
\end{lem}

\begin{proof}
There exist $a,b \in F$ such that $x^2+x=a$ and $z^2+z=b$.
Let $r=x z+z x$, $t=x z+z x+z=r+z$.
It is easy to see that $x r+r x=r$, and $x t+t x=0$.

Since $z=r+t$, $z^2+z+b=r^2+t^2+r t+t r+r+t+b=0$.
Therefore $(z^2+z+b) x+x (z^2+z+b)=r t+t r+r=0$.

Since $s=x+t$ commutes with $x,t,r$, it is in the center of $F[x,z]$.
The elements $x$ and $r$ generate a quaternion algebra over the center of $F[x,z]$, and since $t$ differs from $x$ by a central element, $F[x,z]$ is a quaternion algebra over its center.

Since $F[x,z]$ is a subalgebra of a biquaternion algebra, it cannot be the entire algebra, and therefore its center is either $F$ or a quadratic field extension of $F$.
\end{proof}

\begin{lem}\label{E1}
If $x$ and $z$ are not commuting Artin-Schreier elements then there exists some $w \in V$ which is either Artin-Schreier or square-central and commutes with them both.
\end{lem}

\begin{proof}
If the center of $F[x,z]$ is a quadratic extension of $F$ then it is generated by some $w \in V$, and that finishes the proof.
Otherwise, according to Lemma \ref{E1lem} the center of $F[x,z]$ is $F$ and $A=F[x,z] \otimes F[w,u : w^2+w=c,u^2=d,w u+u w=u]$ for some $c,d \in F$, and this also finishes the proof.
\end{proof}

\begin{thm}\label{Onesteps}
Every two quadruples of generators are connected by a chain of up to three steps of type $\Lambda_3$.
\end{thm}

\begin{proof}
Let $(x,y,z,u)$ and $(x',y',z',u')$ be two quadruples of generators.
If $x$ and $x'$ are not commuting then according to Lemma \ref{E1} there exists some $w$ which is either Artin-Schreier or square-central commuting with $x$ and $x'$.

If $w$ is Artin-Schreier then according to Lemma \ref{instep0} there exist $s,t \neq 0$ such that $(x,s,w,t)$ is a quadruple of generators.

Similarly, there exist some $s',t'$ such that $(x',s',w,t')$ is a quadruple of generators.

Consequently, there is a chain $$(x,y,z,u) \stackrel{\Lambda_3}{\lra} (x,s,w,t) \stackrel{\Lambda_3}{\lra} (x',s',w,t') \stackrel{\Lambda_3}{\lra} (x',y',z',u').$$

If $w$ is square-central then according to Lemma \ref{instep} there exist $s,t \neq 0$ such that $(x,s,t,w)$ is a quadruple of generators.

Similarly, there exist some $s',t'$ such that $(x',s',t',w)$ is a quadruple of generators.

Consequently, there is a chain $$(x,y,z,u) \stackrel{\Lambda_3}{\lra} (x,s,t,w) \stackrel{\Lambda_3}{\lra} (x',s',t',w) \stackrel{\Lambda_3}{\lra} (x',y',z',u').$$

If $x$ and $x'$ are commuting then according to Lemma \ref{instep0} there exist $s,t \neq 0$ such that $(x,s,x',t)$ is a quadruple of generators.

Consequently, there is a chain $$(x,y,z,u) \stackrel{\Lambda_3}{\lra} (x,s,x',t) \stackrel{\Lambda_3}{\lra} (x',y',z',u').$$
\end{proof}

\subsection{Replacing a step of type $\Lambda_3$ with steps of types $\Pi$ and $\Lambda_2$}

In this section we shall show how a step of type $\Lambda_3$ can be obtained by up to three steps of type $\Pi$ and two of type $\Lambda_2$.

\begin{lem}\label{notcomm}
If $y$ and $y'$ are two non-commuting square-central elements in $A$ then $F[y,y']$ is a quaternion algebra either over $F$ or over a quadratic extension of $F$. In particular, there exists either an Artin-Schreier element or a square-central element that commutes with both of them.
\end{lem}

\begin{proof}
Let $t=y y'+y' y$ and $r=y y'+y' y+y'$. It is easy to see that $y t=t y$, $y' t=t y'$ and $y r+r y=t$. In particular $t$ is in the center of $F[y,y']$.
If $t=0$ then $y'=r$ and $y'$ commutes with $y$, but we assumed the contrary, and so $t \neq 0$.
For similar reasons $r \neq 0$.

Let $q=y r t^{-1}$. It is a straight-forward calculation to see that $q \in V$ and $q r+r q=r$. Consequently $q$ and $r$ generate a quaternion algebra over the center of $F[y,y']$. Since this center contains $t$, it is easy to see that $y$ and $y'$ belong to that quaternion algebra, and therefore $F[y,y']=K[q,r]$ where $K=Z(F[y,y'])$.
Since it is a subalgebra of a biquaternion algebra over $F$, its center can be either $F$ or a quadratic extension of $F$. In both cases there exists either an Artin-Schreier element or a square-central element that commutes with both $y$ and $y'$.
\end{proof}

\begin{thm} \label{TwoSteps}
Every step of type $\Lambda_3$ can be achieved by at most three steps of type $\Pi$ and two of type $\Lambda_2$.
\end{thm}

\begin{proof}
A step of type $\Lambda_3$ preserves either an Artin-Schreier generator or a square-central generator.

Assume that it preserves an Artin-Schreier generator, i.e. $$(x,y,z,u)\stackrel{\Lambda_3}{\lra}(x,y',z',w').$$

If $y' \in F[x,y]$ then $$(x,y,z,u)\stackrel{\Lambda_1}{\lra}(x,y',z,u)\stackrel{\Pi}{\lra}(x,y',z',u').$$

Otherwise, if $y'$ commutes with $y$ then $$(x,y,z,u)\stackrel{\Pi}{\lra}(x,y,?,y y')\stackrel{\Lambda_2}{\lra}(x,y',?,y y')\stackrel{\Pi}{\lra}(x,y',z',u').$$

Assume that they do not commute.
According to Lemma \ref{notcomm}, there exists either an Artin-Schreier element or a square-central element $t$ commuting with both $y$ and $y'$.

If $\mu=x t+t x+t \not \in F$ then it is a straight-forward calculation to show that $\mu$ commutes with $x$, $y$ and $y'$, and so $\mu$ generates a quadratic extension in both $F[z,u]$ and $F[z',u']$. If it is separable then $$(x,y,z,u)\stackrel{\Pi}{\lra}(x,y,\mu,?)\stackrel{\Lambda_2}{\lra}(x,y',\mu,?)\stackrel{\Pi}{\lra}(x,y',z',u'),$$ and if inseparable then $$(x,y,z,u)\stackrel{\Pi}{\lra}(x,y,?,\mu)\stackrel{\Lambda_2}{\lra}(x,y',?,\mu)\stackrel{\Pi}{\lra}(x,y',z',u').$$

Otherwise, $t$ could be picked such that $\mu=0$ and then $x t+t x=t$, and therefore $t$ must be square-central. In this case \begin{eqnarray*}(x,y,z,u)&\stackrel{\Pi}{\lra}&(x,y,?,t y)\stackrel{\Lambda_2}{\lra}(x,t,?,t y)\\&\stackrel{\Pi}{\lra}&(x,t,?,t y')\stackrel{\Lambda_2}{\lra}(x,y',?,t y')\stackrel{\Pi}{\lra}(x,y',z',u').
\end{eqnarray*}

Assume that the initial $\Lambda_3$-step preserves a square-centarl generator, i.e. $$(x,y,z,u)\stackrel{\Lambda_3}{\lra}(x',y,z',w').$$

If $x' \in F[x,y]$ then $$(x,y,z,u)\stackrel{\Lambda_1}{\lra}(x',y,z,w)\stackrel{\Pi}{\lra}(x',y,z',w').$$

Otherwise, if $x'$ commutes with $x$ then $$(x,y,z,u)\stackrel{\Pi}{\lra}(x,y,x+x',?)\stackrel{\Lambda_2}{\lra}(x',y,x+x',?)\stackrel{\Pi}{\lra}(x',y,z',u').$$

Assume that they do not commute.
According to Lemma \ref{E1lem}, there exists either an Artin-Schreier element or a square-central element $t$ commuting with both $x$ and $x'$.

Let $\mu=t+y t y^{-1}$. This element commutes with $x$, $x'$ and $y'$. If $\mu \not \in F$ then $\mu$ generates a quadratic extension in both $F[z,u]$ and $F[z',u']$. If it is separable then $$(x,y,z,u)\stackrel{\Pi}{\lra}(x,y,\mu,?)\stackrel{\Lambda_2}{\lra}(x,y',\mu,?)\stackrel{\Pi}{\lra}(x,y',z',u'),$$ and if inseparable then $$(x,y,z,u)\stackrel{\Pi}{\lra}(x,y,?,\mu)\stackrel{\Lambda_2}{\lra}(x,y',?,\mu)\stackrel{\Pi}{\lra}(x,y',z',u').$$

If $\mu=0$ then $t$ commutes with $y$ and hence $t \in F[z,u]$. If $t$ is square-central then $$(x,y,z,u)\stackrel{\Pi}{\lra}(x,y,t,?)\stackrel{\Lambda_2}{\lra}(x',y,t,?)\stackrel{\Pi}{\lra}(x',y,z',u'),$$ and if Artin-Schreier then $$(x,y,z,u)\stackrel{\Pi}{\lra}(x,y,?,t)\stackrel{\Lambda_2}{\lra}(x',y,?,t)\stackrel{\Pi}{\lra}(x',y,z',u').$$

If $\mu \in F^\times$ then $(\mu^{-1} t) y+y (\mu^{-1} t)=y$, which means that $\mu^{-1} t$ is Artin-Schreier, but $t$ was either Artin-Schreier or square-central to begin with, and therefore $\mu=1$.
In this case, $t+x,t+x' \not \in F$, because otherwise $x$ and $x'$ commute, and we assumed that they do not.
Now, $t+x$ commutes with both $x$ and $y$, which means that it generates a quadratic extension of $F$ inside $F[z,u]$, which means that either $a (t+x)$ is Artin-Schreier for some $a \in F^\times$ or $t+x$ is square-central.
Similarly, $t+x'$ commutes with both $x'$ and $y$, which means that it generates a quadratic extension of $F$ inside $F[z',u']$, which means that either $a' (t+x')$ is Artin-Schreier for some $a' \in F^\times$ or $t+x'$ is square-central.

If $a (t+x)$ and $a' (t+x')$ are Artin-Schreier then we have \begin{eqnarray*}(x,y,z,u)&\stackrel{\Pi}{\lra}&(x,y,a(t+x),?)\stackrel{\Lambda_2}{\lra}(t,y,a(t+x),?)\\&\stackrel{\Pi}{\lra}&(t,y,a'(t+x'),?)\stackrel{\Lambda_2}{\lra}(x',y,a'(t+x'),?)\\
&\stackrel{\Pi}{\lra}&(x',y,z',u').
\end{eqnarray*}

If $t+x$ and $t+x'$ are square-central then we have
\begin{eqnarray*}(x,y,z,u)&\stackrel{\Pi}{\lra}&(x,y,?,t+x)\stackrel{\Lambda_2}{\lra}(t,y,?,t+x)\\&\stackrel{\Pi}{\lra}&(t,y,?,t+x')\stackrel{\Lambda_2}{\lra}(x',y,?,t+x')\stackrel{\Pi}{\lra}(x',y,z',u').\end{eqnarray*}

If $a(t+x)$ is Artin-Schreier and $t+x'$ is square central then we have
\begin{eqnarray*}(x,y,z,u)&\stackrel{\Pi}{\lra}&(x,y,a(t+x),?)\stackrel{\Lambda_2}{\lra}(t,y,a(t+x),?)\\&\stackrel{\Pi}{\lra}&(t,y,?,t+x')\stackrel{\Lambda_2}{\lra}(x',y,?,t+x')\stackrel{\Pi}{\lra}(x',y,z',u').\end{eqnarray*}

The case of square-central $t+x$ and Artin-Schreier $a' (t+x')$ is essentially the same as the last one.
\end{proof}

\begin{rem}\label{twoentries}
As a result, every two quadruples of generators are connected by a chain of up to $9$ steps of type $\Pi$ and $6$ steps of type $\Lambda_2$.
\end{rem}

\subsection{Replacing a step of type $\Lambda_2$ with steps of types $\Omega_i$, $\Omega_s$, $\Omega_c$ and $\Lambda_1$}

I this section we shall show how a step of type $\Lambda_2$ can be obtained by up to three steps, one of which can be of type $\Omega_i$, $\Omega_s$ or $\Omega_c$ and the others are of type $\Lambda_1$. Since $\Pi$ can be obtained by up to three steps of type $\Lambda_1$, it means that every two quadruples of generators are connected by a chain of up to $45$ steps, where up to $6$ of them are of type $\Omega_i$, $\Omega _s$ or $\Omega_c$ and the rest are of type $\Lambda_1$.

\begin{lem}
If a step of type $\Lambda_2$ preserves two inseparable generators, i.e. $(x,y,z,u) \stackrel{\Lambda_2}{\lra} (x',y,z',u)$ then it can be achieved by at most two steps of type $\Lambda_1$ and one of type $\Omega_i$.
\end{lem}

\begin{proof}
The element $x z'+z' x+z'$ is nonzero because $(x z'+z' x+z') u+u (x z'+z' x+z')=u$. Consequently, $(x,y,x z'+z' x+z',u)$ is a quadruple of generators.
Similarly, $(x z'+z' x+x,y,z',u)$ is a quadruple of generators.
One can therefore do the following steps:
\begin{eqnarray*}(x,y,z,u) \stackrel{\Lambda_1}{\lra} (x,y,x z'+z' x+z',u) \stackrel{\Omega_i}{\lra} (x z'+z' x+x,y,z',u)\\ \stackrel{\Lambda_1}{\lra} (x',y,z',u).
\end{eqnarray*}

The element $r=x z'+z' x$ was added in the middle step to the Artin-Schreier generators. This element commutes with $y$ and $u$, and therefore it is in $F[u,y]$ and consequently of the form $a+b y+c u+d y u$. This element however also satisfies $x r+r x=z' r+r z'=r$. Hence, $a=b=c=0$.
\end{proof}

\begin{lem}
If a step of type $\Lambda_2$ preserves two Artin-Schreier generators, i.e. $(x,y,z,u)\stackrel{\Lambda_2}{\lra}(x,y',z,u')$ then it can be achieved by at most two steps of type $\Lambda_1$ and one of type $\Omega_s$.
\end{lem}

\begin{proof}
If $y u'+u' y=0$ then $y$ commutes with $u'$ and then one can do $$(x,y,z,u) \stackrel{\Lambda_1}{\lra} (x,y,z,u') \stackrel{\Lambda_1}{\lra} (x,y',z,u').$$

Otherwise, $y u'+u' y$ is square-central, and $(x,y,z,(y u'+u' y) y)$ is a quadruple of generators.
Similarly $(x,(y u'+u' y)^{-1} u',z,u')$ is a quadruple of generators.
One can therefore do
\begin{eqnarray*}
(x,y,z,u) \stackrel{\Lambda_1}{\lra} (u,y,z,(y u'+u' y) y) \stackrel{\Omega_s}{\lra} (x,(y u'+u' y)^{-1} u',z,u')\\ \stackrel{\Lambda_1}{\lra} (x,y',z,u').
\end{eqnarray*}

In the middle step, the square-central generators were multiplied by $q=y^{-1} (y u'+u' y)^{-1} w'$. This element commutes with $x$ and $z$ and therefore $q \in F[x,z]$ and consequently of the form $a+b x+c z+d x z$. However, $q$ commutes with $y u'$, and therefore $d=0$ and $b=c$.
\end{proof}

\begin{lem}
If a step of type $\Lambda_2$ preserves one Artin-Schreier generator and one square-central generator, i.e. $(x,y,z,u)\stackrel{\Lambda_2}{\lra}(x',y,z,u')$ then it can be achieved by either at most three steps of type $\Lambda_1$ or at most two steps of type $\Lambda_1$ and one of type $\Omega_c$.
\end{lem}

\begin{proof}
If $x u'+u' x+u'$ is zero then $x u'+u' x=u'$.
Therefore one can do $(x,y,z,u) \stackrel{\Lambda_1}{\lra} (x,y,z,y u') \stackrel{\Lambda_1}{\lra} (x',y,z,y u') \stackrel{\Lambda_1}{\lra} (x',y,z,u')$.

Let us assume $x u'+u' x+u' \neq 0$.
\begin{eqnarray*}
x (x u'+u' x+u')+(x u'+u' x+u') x=\\x^2 u'+x u' x+x u'+x u' x+u' x^2+u' x=\\(x+\alpha) u'+x u'+u' (x+\alpha)+u' x=\\x u'+\alpha u'+x u'+u' x+\alpha u'+u' x=0.
\end{eqnarray*}
Therefore $x$ commutes with $x u'+u' x+u'$. In fact, $(x,y,z,x u'+u' x+u')$ is a quadruple of generators, and in particular $x u'+u' x+u'$ is square-central.
Now $x (x u'+u' x+u')^{-1}(x u'+u' x)+(x u'+u' x+u')^{-1}(x u'+u'x) x=(x u'+u' x+u')^{-1}(x u'+u' x),$ and $(x u'+u' x+u')^{-1}(x u'+u' x)$ commutes with $z$ and $x u'+u' x+u'$, and therefore $(x u'+u' x+u')^{-1}(x u'+u' x)=b y+c x y$ for some $b,c \in F$, but $(x u'+u' x+u')^{-1}(x u'+u' x)$ also commutes with $x u'+u' x$ while $y (x u'+u' x)+(x u'+u' x) y=0$ and $x y (x u'+u' x)+(x u'+u' x) x y=y (x u'+u' x)$ and therefore $c=0$. In particular $x u'+u' x=b (x u'+u' x+u') y$.

It is a straight-forward calculation to check that $$(x+b y (x u'+u' x+u') u'^{-1} z) u'+u' (x+b y (x u'+u' x+u') u'^{-1} z)=0,$$
$$(x+b y (x u'+u' x+u') u'^{-1} z) z+z (x+b y (x u'+u' x+u') u'^{-1} z)=0,$$
and so $(x+b y (x u'+u' x+u') u'^{-1} z,y,z,u')$ is a quadruple of generators too.

Therefore we have the chain
\begin{eqnarray*}(x,y,z,u) \stackrel{\Lambda_1}{\lra} (x,y,z,x u'+u' x+u') \stackrel{\Omega_c}{\lra} \\(x+\beta y (x u'+u' x+u') u'^{-1} z,y,z,u') \stackrel{\Lambda_1}{\lra} (x',y,z,u').
\end{eqnarray*}
\end{proof}

\begin{thm}
Every two quadruples of elements are connected by a chain of up to $45$ steps, of which up to $6$ are of type $\Omega_i$, $\Omega_s$ or $\Omega_c$ and the rest are of type $\Lambda_1$.
\end{thm}

\chapter{Computational Aspects}

\section{Quadratic Elements and Quaternion Standard Equations}\label{sec1}

Let $\mathbb{H}=\mathbb{R}+\mathbb{R} i+\mathbb{R} j+\mathbb{R} k$ be the real quaternion algebra, with $i^2=j^2=-1$, $k=i j$ and $j i=-k$.

Every element $z$ in this algebra is therefore of the form $z=c_1+c_2 i+c_3 j+c_4 k$ where $c_1,c_2,c_3,c_4 \in \mathbb{R}$.
Let $\Re(z)=c_1$ and $\Im(z)=z-\Re(z)=c_2 i+c_3 j+c_4 k$. We call $\Re(z)$ the real part of $z$ and $\Im(z)$ the imaginary part.
If $\Re(z)=z$ then $z$ is called pure real and if $\Im(z)=z$ then $z$ is called pure imaginary.
Every element $z$ then can be written as the sum of two elements $r+x$ such that $r=c_1$ is pure real and $x=c_2 i+c_3 j+c_4 k$ is pure imaginary.
By easy calculation one can show that $x^2=-(c_2^2+c_3^2+c_4^2) \in \mathbb{R}$.

The conjugate of $z$ is defined to be $\bar{z}=r-x=c_1-c_2 i-c_3 j-c_4 k$.
The norm of $z$ is defined to be $\norm(z)=z \bar{z}=r^2-x^2=c_1^2+c_2^2+c_3^2+c_4^2 \in \mathbb{R}$.
The norm is known to be a multiplicative function, i.e. $f(z_1 z_2)=f(z_1) f(z_2)$, and for any $c \in \mathbb{R}$, $f(c z)=c^2 f(z)$.

A quaternion polynomial equation with one indeterminate $z$ is called standard if it is of the form $a_n z^n+\dots+a_1 z+a_0=0$ for some $a_0,\dots,a_n \in \mathbb{H}$. Notice that since the quaternion algebra is noncommutative, the order of multiplication is crucial, for instance the equations $a z^2-b=0$, $z a z-b=0$ and $z^2 a-b=0$ are three distinct equations.

In \cite{Janovska} Janovsk\'{a} and Opfer reduced the problem of solving any standard quaternion equation of degree $n$ to a real equation of degree $2 n$.
However, for the case of $n=2$ it is not optimal, since there are reductions into equations of degree $3$ instead of $4$ (see \cite{Huang}, \cite{Au-Yeung}).

Here we present a new method for solving quaternion standard equations.
For the case of $n=2$ it is very similar to the techniques appearing in \cite{Huang} and \cite{Au-Yeung}. For the case of $n=3$, if the equation has at least one pure imaginary root, then the problem is reduced to solving real equations of degrees no greater than $4$, as opposed to the degree $6$ equation that arises from the method in \cite{Janovska}.

Later in this section we shall use Wedderburn's decomposition method for standard quaternion polynomials.
The ring of standard (or left) quaternion polynomials $\mathbb{H}[z]$ is simply the ring obtained by adding the variable $z$ to the quaternion algebra with the relations $z a=a z$ for any $a \in \mathbb{H}$. The elements $a z^2$, $z a z$ and $z^2 a$ are the same inside this ring.
However, every polynomial $f(z)$ in that ring has a standard form, where the coefficients lie on the left-hand side of the variable, i.e. $f(z)=a_n z^n+\dots+a_1 z+a_0$ for some $a_0,\dots,a_n \in \mathbb{H}$.
When substituting an element $z_0 \in \mathbb{H}$ in $f(z)$ we substitute in the standard form, i.e. $f(z_0)=a_n z_0^n+\dots+a_1 z_0+a_0$.
We call $a$ a root of $f(z)$ if $f(a)=0$.
Consequently, finding the roots of a polynomial in this ring is equivalent to solving a standard quaternion equation.

It is important to mention that the substitution map $S_{z_0} : \mathbb{H}[z] \rightarrow \mathbb{H}$, taking $S_{z_0}(f(z))=f(z_0)$, is not a ring homomorphism if $z_0$ is not pure real.
For example, if $z_0=i$, $g(z)=z-j$, $h(z)=z+j$ and $f(z)=g(z) h(z)=z^2+1$ then $g(i) h(i)=(i-j) (i+j)=2 k \neq 0$ while $f(i)=0$.

The following statement is known to be true (see \cite{R}):
For given $f(z),g(z),h(z) \in \mathbb{H}[z]$, if $f(z)=g(z) h(z)$ and $a$ is a root of $f(z)$ but not of $h(z)$ then $h(a) a h(a)^{-1}$ is a root of $g(z)$.
Consequently, if $n=\deg(f)$ distinct roots of $f(z)$ are known
then we can factorize $f(z)$ completely to linear factors.
The opposite is not true, i.e. there is no simple algorithm for finding the roots of a polynomial knowing its factorization.

\subsection{Roots of a quaternion standard polynomial}\label{imsec}

Let there be a monic polynomial $f(z)=z^n+a_{n-1} z^{n-1}+\dots+a_1 z+a_0 \in \mathbb{H}[z]$ where $a_{k-1},\dots,a_0 \in \mathbb{H}$ and $a_0 \neq 0$.

Similarly to the ring of standard polynomials with one variable $\mathbb{H}[z]$, one can look at the ring of polynomials with two variables $\mathbb{H}[r,N]$ where $r a=a r$ and $N a=a N$ for any $a \in \mathbb{H}$.

\begin{lem}\label{twovar}
There exist polynomials $g,h \in \mathbb{H}[r,N]$ such that $f(z_0)=g(r_0,N_0) x_0+h(r_0,N_0)$ for any $z_0 \in \mathbb{H}$, $r_0=\Re(z)$, $x_0=\Im(z)$, $N_0=-x_0^2$.
\end{lem}

\begin{proof}
Let $z_0$ be some arbitrary element in $\mathbb{H}$. $f(z_0)=z_0^n+a_{n-1} z_0^{n-1}+\dots+a_1 z_0+a_0$.
Now, $z_0=r_0+x_0$ for some pure real $r_0$ and some pure imaginary $x_0$.
Since $r_0$ is real, it commutes with $x_0$.
Therefore $z_0^k=\sum_{m=0}^k \binom{k}{m} r_0^{m-k} x_0^m$.
Let $N_0=-x_0^2$. This element is pure real.
For all $1 \leq k \leq n$, $z_0^k=(\sum_{m=0}^{\lfloor \frac{k-1}{2} \rfloor} \binom{k}{2 m+1} (-1)^m N_0^m r_0^{k-(2 m+1)}) x_0+\sum_{m=0}^t \binom{k}{2 m+1} (-1)^m N_0^m r_0^{k-2 m}$.
Let $g_k(r,N)=\sum_{m=0}^{\lfloor \frac{k-1}{2} \rfloor} \binom{k}{2 m+1} (-1)^m N_0^m r_0^{k-(2 m+1)}$ and $h(r,N)=\sum_{m=0}^t \binom{k}{2 m+1} (-1)^m N_0^m r_0^{k-2 m}$.
Now let $g(r,N)=g_n(r,N)+a_{n-1} g_{n-1}(r,N)+\dots+a_1 g_1(r,N)$ and $h(r,N)=h_n(r,N)+a_{n-1} h_{n-1}(r,N)+\dots+a_1 h_1(r,N)+a_0$.
It is easy to see that $f(z_0)=g(r_0,N_0) x_0+h(r_0,N_0)$.
\end{proof}

\begin{thm} \label{normeq}
Given an element $z_0 \in \mathbb{H}$, $x_0,r_0,N_0$ are as in Lemma \ref{twovar}, $z_0$ is a root of $f(z)$ if and only if one of the following conditions is satisfied:
\begin{enumerate}
\item $(r_0,N_0)$ is a solution to both $h(r,N)=0$ and $g(r,N)=0$.
\item $(r_0,N_0)$ is a solution to the equation $-g(r,N) \overline{g(r,N)} g(r,N) N=h(r,N) \overline{g(r,N)} h(r,N)$ and $x_0=-g(r_0,N_0)^{-1} h(r_0,N_0)$.
\end{enumerate}
\end{thm}

\begin{proof}
If $h(r_0,N_0)=g(r_0,N_0)=0$ then $f(z_0)=g(r_0,N_0) x_0+h(r_0,N_0)=0 x_0+0=0$, i.e. $z_0$ is a root of $f(z)$.

If $h(r_0,N_0) \neq 0$ or $g(r_0,N_0) \neq 0$ while $f(z_0)=0$, then $h(r_0,N_0) \neq 0$ and $g(r_0,N_0) \neq 0$, because $g(r_0,N_0) x_0=-h(N_0)$.
Therefore $\overline{g(r_0,N_0)} g(r_0,N_0) x_0=\norm(g(r_0,N_0)) x_0=-\overline{g(r_0,N_0)} h(r_0,N_0)$.

Consequently $-\norm(g(r_0,N_0))^2 N_0=\overline{g(r_0,N_0)} h(r_0,N_0) \overline{g(r_0,N_0)} h(r_0,N_0)$, \\
i.e. $-g(r_0,N_0) \norm(g(r_0,N_0)) N_0=\overline{g(r_0,N_0)} h(r_0,N_0) \overline{g(r_0,N_0)} h(r_0,N_0)$.
This is surely not the trivial equation, because the difference between the lowest degree among the nonzero monomials of the right-hand side of the equation and the lowest degree among the nonzero monomials of the left-hand side of the equation is at least $1$.
Consequently, $(r_0,N_0)$ is a root of the equation $g(r,N) \overline{g(r,N)} g(r,N) N=h(r,N) \overline{g(r,N)} h(r,N)$.
\end{proof}

\subsection{Solving quadratic equations}

Let $f(z)=z^2+a z+b$. By replacing $z$ with $z-\frac{\Re(a)}{2}$, we may assume that $\Re(a)=0$.
The case of $a=0$ is simple: If $b$ is not pure real then the roots are $\pm \sqrt[4]{\norm{b}} e^{\frac{\theta}{2} \frac{\Im(b)}{\sqrt{\norm(b)}}}$ where $\theta$ is the phase of $b$ in its polar decomposition as a quaternion. If $b$ is pure real then if it is negative then the roots are all the pure imaginary elements whose norms are real square roots of $\norm(b)$. Otherwise, the roots are the real positive and negative square roots of $b$.

Therefore we assume $a \neq 0$. We assumed that $\Re(a)=0$ and therefore $a$ is a nonzero pure imaginary.
Taking $d=\frac{b+a b a^{-1}}{2}$, it is clear that $a d=-d a$ and $a (b-d)=(b-d) a$.
Since $b-d$ commutes with $a$, it is of the form $m+n a$ for some $m,n \in \mathbb{R}$.
The case of $d=0$ is again simple, because then $b$ commutes with $a$ and the equation can be solved over the field $\mathbb{R}[a]$.
Consequently we shall assume that $d \neq 0$.

Under this assumption, every equation of the form $z^2+a z+b=0$ with $a,b \in \mathbb{H}$ can be brought to the form $z^2+a z+m+n a+d=0$ with $\Re(a)=0$, $a d=-d a$ and $m,n \in \mathbb{R}$.

\begin{thm} \label{quadsolve}
Assume $a,d \neq 0$. Let $z_0$ be a root of $f(z)=z^2+a z+m+n a+d$.
If $n \neq 0$ then $r_0=\Re(z_0)$ is a solution to the equation $16 r^6+(-8 a^2+16 m) r^4+(-a^2 (4 m-a^2)+4 a^2 n^2+4 d^2) r^2-a^4 n^2=0$ and $\Im(z_0)=-(2 r_0+a)^{-1} (\frac{1}{2 r_0} a (r_0+n) (2 r_0+a)+d)$.
If $n=0$ then one of the following happens:
\begin{enumerate}
\item $r_0=0$, $N_0=-\Im(z_0)^2$ is a solution to the equation $0=N^2+(a^2-2 m) N+m^2-d^2$ and $\Im(z_0)=-a^{-1} (m+d-N_0)$
\item $r_0$ is a solution to the equation $0=16 r^4+(-8 a^2+16 m) r^2-a^2 (4 m-a^2)+4 d^2$ and $\Im(z_0)=-(2 r_0+a)^{-1} (\frac{1}{2} a (2 r_0+a)+d)$.
\end{enumerate}
\end{thm}

\begin{proof}
The polynomials obtained according to the proof of Lemma \ref{twovar} are in this case $g(r,N)=2 r+a$ and $h(r,N)=r^2-N+a r+b$.
Again $r_0=\Re(z_0)$, $x_0=\Im(z_0)$ and $N_0=-x_0^2$.

Obviously $g(r_0,N_0) \neq 0$, therefore for according to Theorem \ref{normeq}, $(r_0,N_0)$ is a solution to $-g(r,N) \overline{g(r,N)} g(r,N) N=h(r,N) \overline{g(r,N)} h(r,N)$.

We shall solve this equation then.
$-(2 r+a) (2 r-a) (2 r+a) N=(r^2-N+a r+b) (2 r-a) (r^2-N+a r+b)$.

Taking only the part of the equation which anti-commutes with $a$ we obtain

$0=d (2 r-a)(r^2-N+a r+m+n a)+(2 r-a)(r^2-N+a r+m+n a) d=((2 r+a)(r^2-N-a r+m-n a)+(2 r-a)(r^2-N+a r+m+n a))d$

Which means that $0=(2 r+a)(r^2-N-a r+m-n a)+(2 r-a)(r^2-N+a r+m+n a)=4 r^3-4 r N+4 r m-2 a^2 r-2 n a^2$.

If $n \neq 0$ then $r \neq 0$, and so $N=r^2+m-\frac{1}{2} a^2-\frac{1}{2 r} n a^2$.

$h(r,N)=r^2-N+a r+b=r^2-(r^2+m-\frac{1}{2} a^2-\frac{1}{2 r} n a^2)+a r+m+n a+d=\frac{1}{2} a^2+\frac{1}{2 r} n a^2+ a r+n a+d
=\frac{1}{2 r} a (r+n) (2 r+a)+d$

The equation of interest is $-(2 r+a) (2 r-a) (2 r+a) N=h(r,N) (2 r-a) h(r,N)$.
Its part which commutes with $a$ provides us with
$-(2 r+a) (2 r-a) (2 r+a) N=(\frac{1}{2 r} a (r+n) (2 r+a)) (2 r-a) (\frac{1}{2 r} a (r+n) (2 r+a))+d (2 r-a) d
=\frac{1}{4 r^2} (2 r+a) (4 r^2-a^2) a^2 (r+n)^2+(2 r+a) d^2$

Therefore $-(4 r^2-a^2) N=\frac{1}{4 r^2}(4 r^2-a^2) a^2 (r+n)^2+ d^2$, which means that
$0=\frac{1}{4 r^2}(4 r^2-a^2)(4 r^2 (r^2+m-\frac{1}{2} a^2-\frac{1}{2 r} n a^2)+a^2 (r+n)^2)+d^2
=\frac{1}{4 r^2}(4 r^2-a^2) (4 r^4+4 r^2 m-2 a^2 r^2-2 r n a^2+a^2 r^2+2 a^2 r n+a^2 n^2)+d^2
=\frac{1}{4 r^2}(4 r^2-a^2) (4 r^4+4 r^2 m-a^2 r^2+a^2 n^2)+d^2$.

Consequently $16 r^6+(-8 a^2+16 m) r^4+(-a^2 (4 m-a^2)+4 a^2 n^2+4 d^2) r^2-a^4 n^2=0$.

If $n=0$ then $0=4 r^3-4 r N+4 r m-2 a^2 r=r (4 r^2-4 N+4 m-2 a^2)$, which means that either $r=0$ or $N=r^2+m-\frac{1}{2} a^2$.
If $r=0$ then $a^3 N=(-N+m+d) (-a) (-N+m+d)$. Taking only the part which commutes with $a$ we obtain
$a^3 N=-(-N+m)^2 a+a d^2$, hence $a^2 N=-N^2+2 m N-m^2+d^2$, and consequently $0=N^2+(a^2-2 m) N+m^2-d^2$.

If $N=r^2+m-\frac{1}{2} a^2$ then $h(r,N)=r^2-N+a r+b=r^2-(r^2+m-\frac{1}{2} a^2)+a r+m+d=\frac{1}{2} a^2+a r+d=\frac{1}{2} a (2 r+a)+d$.
From $-(2 r+a) (2 r-a) (2 r+a) N=h(r,N) (2 r-a) h(r,N)$ we obtain $-(2 r+a) (2 r-a) (2 r+a) N=(\frac{1}{2} a (2 r+a)+d) (2 r-a) (\frac{1}{2} a (2 r+a)+d)$. Taking the part which commutes with $a$ we get
$-(2 r+a) (2 r-a) (2 r+a) N=\frac{1}{4} a^2 (2 r+a)^2 (2 r-a)+(2 r+a) d^2$.
Therefore $-(4 r^2-a^2) N=\frac{1}{4} a^2 (4 r^2-a^2)+ d^2$, hence
$0=\frac{1}{4} (4 r^2-a^2) (4 (r^2+m-\frac{1}{2} a^2)+a^2)+d^2=\frac{1}{4} (4 r^2-a^2) (4 r^2+ 4m-a^2)+d^2$
and consequently $0=16 r^4+(-8 a^2+16 m) r^2-a^2 (4 m-a^2)+4 d^2$.
\end{proof}

\subsection{Pure imaginary roots of a quaternion standard polynomial}

Let $f(z)$, $g(r,N)$ and $h(r,N)$ as in Lemma \ref{twovar}. Let $g(N)=g(0,N)$ and $h(N)=h(0,N)$. For every pure imaginary $z_0$, $f(z_0)=g(N_0) z_0+h(N_0)$ where $N_0=\norm(z_0)=-z_0^2$.
In particular, $\deg(g)=\lfloor \frac{\deg{f}-1}{2} \rfloor$ and $\deg{h} \leq \lfloor \frac{\deg{f}}{2} \rfloor$.

The following corollary is an easy result of Theorem \ref{normeq}:
\begin{cor}\label{pureim}
A pure imaginary element $z_0$ of norm $N_0$ is a root of $f(z)$ if and only if one of the following conditions is satisfied:
\begin{enumerate}
\item $N_0$ is a solution to both $h(N)=0$ and $g(N)=0$.
\item $N_0$ is a solution to the equation $-g(N) \overline{g(N)} g(N) N=h(N) \overline{g(N)} h(N)$ and $z_0=-g(N_0)^{-1} h(N_0)$.
\end{enumerate}
\end{cor}

\begin{prop}\label{infcor}
The polynomial $f(z)$ has infinitely many pure imaginary roots if and only if $h(N)=0$ and $g(N)=0$ have a common real solution.
\end{prop}

\begin{proof}
If $h(N)=0$ and $g(N)=0$ have a common real solution $N_0$ then every element $z_0 \in Q$ satisfying $-z_0^2=N_0$ is a root of $f(z)$.

If $h(N)$ and $g(N)$ have no common root, and $z_0$ is a pure imaginary root of $f(z)$ of norm $N_0$, then $h(N_0) \neq 0$ and $g(N_0) \neq 0$.
On the other hand, $N_0$ is a solution to the equation $g(N) \overline{g(N)} g(N) N=h(N) \overline{g(N)} h(N)$.
The degree of the left-hand side of this equation is $3 \deg(g)+1$, while the degree of the right-hand side is $2 \deg(h)+\deg(g)$. There is an equality only if $2 \deg(g)+1=2 \deg(h)$, but that can never happen, therefore the equation is not trivial, which means that by splitting the equation into four (according to the structure of $\mathbb{H}$ as a vector space over $\mathbb{R}$, i.e. $\mathbb{R}+\mathbb{R} i+\mathbb{R} j+\mathbb{R} k$) we have at least one nontrivial equation.
Consequently, the number of roots of this system is finite, and therefore the number of pure imaginary roots of $f(z)$ is finite.
\end{proof}

\begin{rem}
If $z_0$ is a pure imaginary root then $\norm(g(N_0)) z_0=-\overline{g(N_0)} h(N_0)$.
Since $\Re(z_0)=0$, we obtain $0=\Re(-\overline{g(N_0)} h(N_0))$.
If this equation is not trivial, then it has a finite set of roots which contains all the pure imaginary roots of the original equation.
\end{rem}

\subsection{Solving cubic quaternion equations with at least one pure imaginary root}

\begin{lem} \label{aplem}
For any polynomial $p(z) \in \mathbb{H}[z]$,
if $z_0 \neq a$ is a root of $f(z)=p(z)(z-a)$ then $0=z_0^2-(a+b) z_0+b a$ for some root $b$ of $p(z)$
\end{lem}

\begin{proof}
According to Wedderburn's method, $b=(z_0-a) z_0 (z_0-a)^{-1}$ is a root of $p(z)$.
Hence $b (z_0-a)=(z_0-a) z$, i.e. $0=z_0^2-(a+b) z_0+b a$.
\end{proof}

\begin{rem}
If the decomposition into linear factors of a given polynomial $f(z) \in \mathbb{H}[z]$ is known, then the question of finding its roots becomes (inductively) a sequence of quadratic equations one has to solve. Over the quaternion algebra the quadratic equations are solvable and so one can obtain the roots of any standard polynomial if he knows its decomposition into linear factors.
\end{rem}

Let $f(z)$ be a quaternion standard cubic polynomial.
The equation $-g(N) \overline{g(N)} g(N) N=h(N) \overline{g(N)} h(N)$ (from Corollary \ref{pureim}) is with one variable $N$ and is of degree $4$ at most, and therefore its real roots can be expressed in terms of radicals, which means that the pure imaginary roots of $f(z)$ can also be expressed in those terms.

Assume $f(z)$
has one such root $a$, then $f(z)=p(z) (z-a)$.
The polynomial $p(z)$
is quadratic and therefore its roots can be formulated.
Consequently, $p(z)$ can be fully factorized into linear factors and so is $f(z)$.
Furthermore, according to Lemma \ref{aplem}, the roots of $f(z)$ are at hand.

\subsubsection{Example}
Consider the polynomial $f(z)=z^3+(2+i j) z+i-j \in \mathbb{H}[z]$.

$g(N)=-N+2+i j$ and $h(N)=i-j$.
They have no common root, so we turn to solve
$-g(N) \norm(g(N)) N=h(N) \overline{g(N)} h(N)$
i.e.
$-(-N+2+i j) ((-N+2)^2+1) N=(i-j) (-N+2-i j) (i-j)$, which means

$-(-N+2+i j) (N^2-4 N+5)  N=(-N+2+i j) (i-j) (i-j)$, and consequently

$-(N^2-4 N+5)  N=-2$, and therefore

$N^3-4 N^2+5 N-2=0$.

In general this equation could be split into up to four equations according to the basis of $\mathbb{H}$ as an $\mathbb{R}$-vector space.
However, in this case, $N^3-4 N^2+5 N-2$ is pure real and has no imaginary part, which means that we have to solve only one cubic real equation.

Therefore either $N=1$ or $N=2$.
According to Theorem \ref{normeq}, the corresponding roots are $-g(N)^{-1} h(N)$, i.e
$z_1=-\frac{1}{2}(1-i j)(i-j)=-\frac{1}{2}(i-j-j-i)=j$
for $N=2$ we have $z_2=-(i j)^{-1}(i-j)=i+j$.

Consequently $f(z)=p(z) (z-j)$. Next goal is to calculate $p(z)$.

\begin{rem} \label{factrem}
Let us recall how $f(z)$ is decomposed into $p(z) (z-a)$ given a root $a$:

$f(z)=z^n+c_{n-1} z^{n-1}+\dots+c_0$

$f(a)=a^n+c_{n-1} a^{n-1}+\dots+c_0$

$f(z)=f(z)-0=f(z)-f(a)=(z^n-a^n)+c_{n-1} (z^{n-1}-a^{n-1})+\dots+c_1 (z-a)
=((z^{n-1}+a z^{n-2}+\dots+a^{n-1})+c_{n-1} (z^{n-2}+\dots+a^{n-2})+\dots+c_1)(z-a)$

$p(z)=(z^{n-1}+a z^{n-2}+\dots+a^{n-1})+c_{n-1} (z^{n-2}+\dots+a^{n-2})+\dots+c_1$.
\end{rem}

Consequently $p(z)=(z^2+j z-1)+2+i j=z^2+j z+1+i j$.

$-i-j$ is a root of $f(z)$ but not of $z-j$, hence according to Remark \ref{factrem}, $(-i-2 j) (-i-j) (-i-2 j)^{-1}=\frac{1}{5}(-i-2 j) (-i-j) (i+2 j)=\frac{1}{5} (-1-2 i j+i j-2)(i+2 j)=\frac{1}{5} (-3-i j) (-i-j)=\frac{1}{5} (3 i+3 j+j-i)=\frac{1}{5}(2 i+4 j)$ is a root of $p(z)$.

The second and final root of $p(z)$ (which can be obtained using the methods) is $i$.

Again, due to Wedderburn, $p(z)=(z+i+1+i j)(z-i)$, which means that $f(z)=(z+1+i+i j)(z-i)(z-j)$

Let $z_0$ be some root of $f(z)$.
According to Lemma \ref{aplem}, since $i$ is a root of $p(z)$ and is different from $\frac{1}{5}(2 i+4 j)$, $z_0$ must correspond to it, which means that $z_0^2-(j+i) z_0+i j=0$.
$j$ is a root, however it is already known to be a root of $f(z)$ so we look for the other one.
Let $t=z_0-j$ and so $t^2-i t+t j=0$.
Let $r=t^{-1}$ and so $1-r i+j r=0$.
$r=c_1+c_i i+c_j j+c_{i j} i j$, so we obtain the following linear system
\begin{eqnarray}
1+c_i-c_j&=&0\\
-c_1+c_{i j}&=&0\\
-c_{i j}+c_1&=&0\\
c_j-c_i=0
\end{eqnarray}
This system has no solution. Therefore, $f(z)$ has no roots besides $j$ and $i+j$.

\subsection{A note on quadratic two-sided polynomials}

A two-sided polynomial is a polynomial of the form $f(z)=z^n+a_{n-1} z^{n-1} b_{n-1}+\dots+a_1 z b_1+c$.
Unlike the polynomials in $\mathbb{H}[z]$, when substituting an element $z_0 \in \mathbb{H}$ in the two-sided polynomial we follow the two-sided form instead of moving all the coefficients to the left, i.e. $f(z_0)=z_0^n+a_{n-1} z_0^{n-1} b_{n-1}+\dots+a_1 z_0 b_1+c$.
In \cite{Janovska2} Janovsk\'{a} and Opfer provided an example of a quadratic two-sided polynomial with more than two roots with pairwise distinct norms. (These are called essential roots in that paper.)

This is apparently impossible with pure imaginary roots, as the following proposition shows:

\begin{prop}
The number of pure imaginary roots of $f(z)=z^2+a z b+c$, assuming $a,b,c \neq 0$, with pairwise distinct norms, is at most two.
\end{prop}

\begin{proof}
Let $z_0$ be a pure imaginary root of norm $N_0$.
Therefore $-N_0+a z_0 b+c=0$, i.e. $N_0-c=a z_0 b$,
hence $a^{-1} b^{-1} N_0-a^{-1} c b^{-1}=z_0$, which means that $a^{-1} b^{-1} a^{-1} b^{-1} N_0^2-(a^{-1} b^{-1} a^{-1} c b^{-1}+a^{-1} c b^{-1} a^{-1} b^{-1}) N_0+a^{-1} c b^{-1} a^{-1} c b^{-1}=-N_0$.
Consequently, $N_0$ is a root of the non-trivial polynomial $p(N)=a^{-1} b^{-1} a^{-1} b^{-1} N^2+(1-a^{-1} b^{-1} a^{-1} c b^{-1}-a^{-1} c b^{-1} a^{-1} b^{-1}) N+a^{-1} c b^{-1} a^{-1} c b^{-1}$.
Hence, the number of pure imaginary roots of $f(z)$ with pairwise distinct norms does not exceed $2$.
\end{proof}

\section{General Polynomials and Left Eigenvalues}

\subsection{Polynomial rings over division algebras}

Let $F$ be an infinite field and $D$ be a division algebra over $F$ of degree $d$.
We adopt the terminology in \cite{GM}. Let $D_L[z]$ denote the usual
ring of polynomials over $D$ where the variable
$z$ commutes with every $y \in D$. When substituting a value we
consider the coefficients as though they are placed on the
left-hand side of the variable. The substitution map $S_y: D_L[z] \rightarrow D$ is not a ring
homomorphism in general. For example, if $f(z)=a z$ and
$a y \neq y a$ then $S_y(f^2)=S_y(a^2 z^2)=a^2 y^2$ while
$(S_y(f))^2=(S_y(a z))^2=(a y)^2 \neq S_y(f^2)$.

The ring $D_G[z]$ is, by definition, the (associative) ring of
polynomials over $D$, where $z$ is assumed to commute with every
$y \in F=Z(D)$, but not with arbitrary elements of $D$. For
example, if $y \in D$ is non-central, then $y z^2$, $z y z$ and $z^2
y$ are distinct elements of this ring. There is a ring epimorphism
$D_G[z] \rightarrow D_L[z]$, defined by $z \mapsto z$ and $y \mapsto y$ for every $y \in D$, whose
kernel is the ideal generated by the commutators $[y,z]$ ($y \in
D$). Unlike the situation of $D_L[z]$, the substitution maps from
$D_G[z]$ to $D$ are all ring homomorphisms. Polynomials from
$D_G[z]$ are called ``general polynomials",
for example $z i z+j z i+ z i j+5 \in \mathbb{H}_G[z]$.

Polynomials in $D_L[z]$ and polynomials in $D_G[z]$ which ``look like" polynomials in $D_L[z]$, i.e. the coefficients are placed on the left-hand side of the variable, are called ``left" or ``standard polynomials", for example $z^2+i z+j \in \mathbb{H}_G[z]$.

Let $\fring$ be the ring of multi-variable polynomials, where for every $1 \leq i \leq N$, $x_i$ commutes with every $y \in D$ and is not assumed to commute with $x_j$ for $i \neq j$. This is the group ring
of the free monoid $\freegen{x_1,\dots,x_N}$ over $D$. The
commutative counterpart is $D_L[x_1,\dots,x_N]$, which is the ring
of multi-variable polynomials where for every $1 \leq i \leq N$, $x_i$ commutes
with every $y \in D$ and with every $x_j$ for $i \neq j$.

For further reading on what is generally known about polynomial
equations over division rings see \cite{L}.

\subsection{Left eigenvalues of matrices over division algebras}

Given a matrix $A \in M_n(D)$, a left eigenvalue of $A$ is an element $\lambda \in D$ for which there exists a nonzero vector $v \in D^{n \times 1}$ such that $A v=\lambda v$.

For the special case of $D=\mathbb{H}$ (The algebra of real quaternions) and $n=2$ it was proven by Wood in \cite{Wood} that the left eigenvalues of $A$ are the roots of a standard quadratic quaternion polynomial.
In \cite{So} it is proven that for $n=3$, the left eigenvalues of $A$ are the roots of a general cubic quaternion polynomial.

In \cite{Macias-Virgos}, Mac\'{i}as-Virg\'{o}s and Pereira-S\'{a}ez gave another proof for Wood's result. Their proof makes use of the Study determinant.

Given a matrix $A \in M_n(\mathbb{H})$, there exist unique matrices $B,C \in M_n(\mathbb{C})$ such that $A=B+C j$.
The Study determinant of $A$ is
$\det \left[ \begin{array}{lr} B  & -\overline{C} \\ C &  \overline{B}  \end{array}
\right]$.
The Dieudonn\'{e} determinant is (in this case) the square root of the Study determinant. (In \cite{Macias-Virgos} the Study determinant is defined to be what we call the Dieudonn\'{e} determinant.)

For further information about these determinants see \cite{Aslaksen}.

\subsection{The isomorphism between the ring of general polynomials and the group ring of the free monoid with $[D:F]$ variables}

Let $N=d^2$, i.e. $N$ is the dimension of $D$ over its center $F$. In particular there exist $a_1,\dots,a_{N-1} \in D$ such that $D=F+a_1 F+\dots+a_{N-1} F$.

Let $h : D_G[z] \rightarrow \fring$ be the homomorphism for which $h(y)=y$ for all $y \in D$, and $h(z)=x_1+a_1 x_2+\dots+a_{N-1} x_N$.
$D_L[x_1,\dots,x_N]$ is a quotient ring of $\fring$.
Let $g: \fring \rightarrow D_L[x_1,\dots,x_N]$ be the standard epimorphism.

In \cite[Theorem~6]{GM} it says that if $D$ is a division algebra then the homomorphism $g \circ h : D_G[z] \rightarrow D_L[x_1,\dots,x_N]$ is an epimorphism.

In \cite{Chapman2} we proved the following:

\begin{thm} \label{isom}
The homomorphism $h : D_G[z] \rightarrow \fring$ is an isomorphism, and therefore $D_G[z] \cong \fring$.
\end{thm}

We also proposed the following algorithm for finding the co-image of $x_k$ for any $1 \leq k \leq N$:

\begin{algo}
Let $p_1=z$, therefore $h(p_1)=x_1+a_1 x_2+\dots+a_{N-1} x_N$.
We shall define a sequence $\set{p_j : j=1,\dots,n} \sub G_1$ as follows:
If there exists a monomial in $h(p_j)$ whose coefficient $a$ does not commute with the coefficient of $x_k$, denoted by $c$, then we shall define $p_{j+1}=a p_j a^{-1}-p_j$, by which we shall annihilate at least one monomial (the one whose coefficient is $a$), and yet the element $x_k$ will not be annihilated, because $c x_k$ does not commute with $a$.

If $c$ commutes with all the other coefficients then we shall pick some monomial which we want to annihilate. Let $b$ denote its coefficient. Now we shall pick some $a \in D$ which does not commute with $c b^{-1}$ and define $p_{j+1}=b a p_j b^{-1} a^{-1}-p_j$.

The element $x_k$ is not annihilated in this process,
because if we assume that it does at some point, let us say it is annihilated in $h(p_{j+1})$, then $b a c b^{-1} a^{-1}-c=0$.
Therefore $c^{-1} b a c b^{-1} a^{-1}=1$, hence $c b^{-1} a^{-1}=(c^{-1} b a)^{-1}=a^{-1} b^{-1} c$ and, since $b$ commutes with $c$, $a$ commutes with $c b^{-1}$ and that is a contradiction.

In each iteration the length of $h(p_j)$ (the number of monomials in it) decreases by at least one, and yet the element $x_k$ always remains, and since the length of $h(p_1)$ is finite, this process will end with some $p_m$ for which $h(p_m)$ is a monomial.
In this case, $h(q_m)=c x_k$ and consequently $x_k=h(c^{-1} q_m)$.
\end{algo}

\subsection{Real Quaternions}\label{example}
Let $D=\mathbb{H}=\mathbb{R}+i \mathbb{R}+j \mathbb{R}+i j \mathbb{R}$. Now $h(z)=x_1+x_2 i+x_3 j+x_4 i j$\\
$h(z-j z j^{-1})=h(z+j z j)=2 x_2 i+2 x_4 i j$\\
$h((z+j z j)-i j (z+j z j) (i j)^{-1})=2 x_2 i+2 x_4 i j-i j (2 x_2 i+2 x_4 i j) (i j)^{-1}=4 x_2 i$\\ therefore $h^{-1}(x_2)=-\frac{1}{4} i ((z+j z j)+i j (z+j z j) i j))=-\frac{1}{4} (i z+i j z j-j z i j+z i)$.

Similarly, $h^{-1}(x_1)=\frac{1}{4} (z-i z i-j z j-i j z i j)$, $h^{-1}(x_3)=-\frac{1}{4} (j z-i j z i+i z i j+z j)$ and $h^{-1}(x_4)=-\frac{1}{4} (i j z-i z j+j z i+z i j)$.
Consequently, $\overline{z}=\overline{h^{-1}(x_1+x_2 i+x_3 j+x_4 i j)}=h^{-1}(x_1-x_2 i-x_3 j-x_4 i j)
=-\frac{1}{2} (z+i z i+j z j+i j z i j)$.

\subsection{The characteristic polynomial}\label{char}
Let $D,F,d,N$ be the same as they were in the previous subsection.

There is an injection of $D$ in $M_d(K)$ where $K$ is a maximal subfield of $D$.
(In particular, $[K:F]=d$.)
More generally, there is an injection of $M_k(D)$ in $M_{k d}(K)$ for any $k \in \mathbb{N}$.
Let $\widehat{A}$ denote the image of $A$ in $M_{k d}(K)$ for any $A \in M_k(D)$.

The determinant of $\widehat{A}$ is equal to the Dieudonn\'{e} determinant of $A$ to the power of $d$.
(The reduced norm of $A$ is defined to be the determinant of $\widehat{A}$.)

Therefore $\lambda \in D$ is a left eigenvalue of $A$ if and only if $\det(\widehat{A-\lambda I})=0$.
Considering $D$ as an $F$-vector space $D=F+F a_1+\dots+F a_{N-1}$, we can write $\lambda=x_1+x_2 a_1+\dots+x_N a_{N-1}$ for some $x_1,\dots,x_N \in F$. Then $\det(\widehat{A-\lambda I}) \in F[x_1,\dots,x_N]$.
It can also be considered as a polynomial in $\fring$.
Now, there is an isomorphism $h : D_G[z] \rightarrow \fring$, and so $h^{-1} (\det(\widehat{A-\lambda I})) \in D_G[z]$.

Defining $p_A(z)=h^{-1} (\det(\widehat{A-\lambda I}))$ to be the characteristic polynomial of $A$, the left eigenvalues of $A$ are precisely the roots of $p_A(z)$.

The degree of the characteristic polynomial of $A$ is therefore $k d$.

\begin{rem} \rm
If one proves that the Dieudonn\'{e} determinant of $A-\lambda I$ is the absolute value of some polynomial $q(x_1,\dots,x_N) \in D_L[x_1,\dots,x_N]$ then we will be able to define the characteristic polynomial to be $h^{-1}(q(x_1,\dots,x_N))$ and obtain a characteristic polynomial of degree $k$.
\end{rem}

\subsection{The left eigenvalues of a $4 \times 4$ quaternion matrix}

Let $Q$ be a quaternion division $F$-algebra.
Calculating the roots of the characteristic polynomial as defined in Subsection \ref{char} is not always the best way to obtain the left eigenvalues of a given matrix.

The reductions Wood did in \cite{Wood} and So did in \cite{So} suggest that in order to obtain the left eigenvalues of a $2 \times 2$ or $3 \times 3$ matrix one can calculate the roots of a polynomial of degree $2$ or $3$ respectively, instead of calculating the roots of the characteristic polynomial whose degree is $d$ times greater.

In the next proposition we show how (under a certain condition) the eigenvalues of a $4 \times 4$ quaternion matrix can be obtained by calculating the roots of three polynomials of degree $2$ and one of degree $6$.

In \cite{Chapman2} we proved the following:

\begin{prop}
If $M=
\left[ \begin{array}{lr} A  & B \\ C &  D  \end{array}
\right]$ where $A,B,C,D \in M_2(\mathbb{H})$ and $C$ is invertible then $\lambda$ is a left eigenvalue of $M$ if and only if either $e(\lambda)=f(\lambda) g(\lambda)=0$ or $e(\lambda) \neq 0$ and $e(\lambda)\overline{e(\lambda)} h(\lambda)-g(\lambda) \overline{e(\lambda)} f(\lambda)=0$ where $C (A-\lambda I) C^{-1} (D-\lambda I)-C B=
\left[ \begin{array}{lr} e(\lambda)  & f(\lambda) \\ g(\lambda) &  h(\lambda)  \end{array}
\right]$
\end{prop}

As we saw in Subsection \ref{example}, $\overline{e(\lambda)}$ is also a quadratic polynomial,
which means that $e(\lambda)\overline{e(\lambda)} h(\lambda)-g(\lambda) \overline{e(\lambda)} f(\lambda)$ is a polynomial of degree $6$, while the characteristic polynomial of $M$ as defined in Subsection \ref{char} is of degree $8$.

\bibliographystyle{amsalpha}
\bibliography{phdbib}
\end{document}